\documentclass{amsart}

\usepackage{paralist}
\usepackage{color}
\usepackage{mathrsfs}
\usepackage{amssymb}
\usepackage{amsmath}
\usepackage{amsfonts}
\usepackage{stmaryrd}

\newtheorem*{thma}{Theorem~A}
\newtheorem*{thmb}{Theorem~B}
\newtheorem*{thmc}{Theorem~C}
\newtheorem*{thmd}{Theorem~D}
\newtheorem{theorem}{Theorem}[section]
\newtheorem{lemma}[theorem]{Lemma}
\newtheorem{proposition}[theorem]{Proposition}
\newtheorem{corollary}[theorem]{Corollary}
\newtheorem{claim}[theorem]{Claim}
\newtheorem{subclaim}[theorem]{Subclaim}
\newtheorem{fact}[theorem]{Fact}
\newtheorem{question}[theorem]{Question}

\newtheorem{conjecture}[theorem]{Conjecture}

\theoremstyle{definition}
\newtheorem{definition}[theorem]{Definition}
\newtheorem{remark}[theorem]{Remark}

\usepackage[pdftex]{hyperref}
\hypersetup{
    colorlinks=true, 
    linktoc=all,     
    linkcolor=blue,  
    allcolors=blue,
}

\usepackage[framemethod=tikz]{mdframed}

\newcommand{\cf}{\mathrm{cf}}
\newcommand{\dom}{\mathrm{dom}}
\newcommand{\rng}{\mathrm{rng}}
\newcommand{\bb}{\mathbb}

\newcommand{\otp}{\mathrm{otp}}

\newcommand{\pred}{\mathrm{pred}}
\newcommand{\height}{\mathrm{ht}}

\newcommand{\mc}{\mathcal}
\newcommand{\power}{\ensuremath{\mathscr{P}}}
\newcommand{\sub}{\subseteq}
\newcommand{\ra}{\rightarrow}

\newcommand{\Add}{\mathrm{Add}}
\newcommand{\Coll}{\mathrm{Coll}}
\newcommand{\R}{\bb{R}}
\newcommand{\Q}{\bb{Q}}
\renewcommand{\P}{\bb{P}}

\newcommand{\K}{\bb{K}}

\newcommand{\ZFC}{\sf ZFC}

\newcommand{\CH}{\sf CH}
\newcommand{\GCH}{\sf GCH}

\newcommand{\wKH}{\sf wKH}

\newcommand{\PSP}{\mathsf{PSP}}
\newcommand{\KH}{\mathsf{KH}}

\newcommand{\set}[2]{\ensuremath{\{#1 \,|\, #2 \}}}
\newcommand{\seq}[2]{\ensuremath{\langle #1 \,|\, #2 \rangle}}
\newcommand{\rest}[0]{\restriction}

\title{Kurepa trees, continuous images, and perfect set properties}
\author{Chris Lambie-Hanson}
\address[Lambie-Hanson]{
Institute of Mathematics, 
Czech Academy of Sciences, 
{\v Z}itn{\'a} 25, Prague 1, 
115 67, Czech Republic
}
\email{lambiehanson@math.cas.cz}
\urladdr{https://users.math.cas.cz/~lambiehanson/}

\author{{\v S}{\'a}rka Stejskalov{\'a}}
\address[Stejskalov{\'a}]{
Charles University, Department of Logic,
Celetn{\' a} 20, Prague~1, 
116 42, Czech Republic
}
\email{sarka.stejskalova@ff.cuni.cz}
\urladdr{logika.ff.cuni.cz/sarka}

\address{Institute of Mathematics, Czech Academy of Sciences, {\v Z}itn{\'a} 25, Prague 1, 115 67, Czech Republic}

\thanks{Both authors were supported by the Czech Academy of Sciences 
(RVO 67985840) and the GA\v{C}R project 23-04683S}

\subjclass[2020]{03E05, 03E35, 03E47}
\keywords{Kurepa trees, higher Baire space, perfect set properties, continuous images,
full trees}

\begin{document}
\begin{abstract}
  Building upon work of L\"{u}cke and Schlicht, we study (higher) Kurepa trees 
  through the lens of higher descriptive set theory, focusing in particular on 
  various perfect set properties and representations of sets of branches through 
  trees as continuous images of function spaces. Answering a question of L\"{u}cke 
  and Schlicht, we prove that it is consistent with $\CH$ that there exist 
  $\omega_2$-Kurepa trees and yet, for every $\omega_2$-Kurepa tree 
  $T \subseteq {^{<\omega_2}}\omega_2$, the set $[T] \subseteq {^{\omega_2}}\omega_2$ 
  of cofinal branches through $T$ is not a continuous image of ${^{\omega_2}}\omega_2$. 
  We also produce models indicating that the existence of Kurepa trees is not 
  necessary to produce closed subsets of ${^{\omega_1}}\omega_1$ failing to satisfy 
  strong perfect set properties, and prove a number of consistency results regarding 
  \emph{full} and \emph{superthin} trees.
\end{abstract}

\maketitle

\section{Introduction}

Two fundamental features of closed subsets of the classical Baire space ${^\omega}\omega$ 
are the following:
\begin{enumerate}
  \item every nonempty closed subset of ${^{\omega}}\omega$ is a continuous image (and, in fact, 
  a retract) of ${^\omega}\omega$;
  \item every closed subset of ${^{\omega}}\omega$ has the perfect set property.
\end{enumerate}
When passing to the higher Baire spaces of the form ${^{\kappa}}\kappa$ for regular uncountable 
cardinals $\kappa$, both of these properties at least consistently fail. Regarding property (1), 
for an arbitrary regular uncountable cardinal $\kappa$, it is proven in 
\cite[Proposition 1.4]{lucke_schlicht_continuous} that there exists a nonempty closed subset of 
${^{\kappa}}\kappa$ that is not a retract of ${^{\kappa}}\kappa$. If we assume moreover 
that $\kappa^{<\kappa} = \kappa$, which is a standard assumption in the study of higher 
Baire spaces, then it is proven in \cite[Theorem 1.5]{lucke_schlicht_continuous} that there 
exists a nonempty closed subset of ${^{\kappa}}\kappa$ that is not a continuous image of 
${^{\kappa}}\kappa$.

When discussing higher analogues of the perfect set property, we need to fix an appropriate notion 
of perfectness. Unlike in the classical Baire space, where there is a single, unambiguous notion 
of ``perfectness", in the higher Baire space ${^{\kappa}}\kappa$ there exist various degrees 
of perfectness, leading to a corresponding spectrum of perfect set properties. These will be discussed 
in more detail below in Section \ref{section: vaananen}, but for now let us consider the strongest and 
arguably most natural version. We say that a \emph{$\kappa$-perfect} subset of ${^\kappa}\kappa$ is 
the set of all cofinal branches through some cofinally-splitting $({<}\kappa)$-closed subtree 
of ${^{<\kappa}}\kappa$ (see Section \ref{section: prelim} for a precise definition). A set 
$X \subseteq {^{\kappa}}\kappa$ has the \emph{$\kappa$-perfect set property}, denoted 
$\mathsf{PSP}_\kappa(X)$, if either $|X| \leq \kappa$ or $X$ contains a $\kappa$-perfect subset.

As noted in \cite[\S 7]{lucke_definability}, if $\kappa$ is a regular uncountable cardinal 
and every closed subset of ${^{\kappa}}\kappa$ has the perfect set property, then 
$\kappa^+$ is inaccessible in $L$. On the other hand, by an argument of Schlicht 
(cf.\ \cite[\S 9]{lucke_definability}), if $\kappa$ is a regular uncountable cardinal, 
$\nu > \kappa$ is inaccessible, and $G$ is generic for the L\'{e}vy collapse 
$\mathrm{Coll}(\kappa, {<}\nu)$, then, in $V[G]$ $\mathsf{PSP}_\kappa(X)$ holds for every 
closed $X \subseteq {^{\kappa}}\kappa$ (in fact, for every $\Sigma^1_1$ $X \subseteq 
{^{\kappa}}\kappa$. Thus, the $\kappa$-perfect set property for closed subsets of 
${^{\kappa}}\kappa$ for some regular uncountable $\kappa$ is equiconsistent with the 
existence of an inaccessible cardinal.

A starring role in the story of the preceding two paragraphs is played by \emph{Kurepa trees}. 
Indeed, suppose that $\kappa$ is a regular uncountable cardinal that is not strong limit, 
e.g., a successor cardinal. If $T \subseteq {^{<\kappa}}\kappa$ is a $\kappa$-Kurepa tree\footnote{See 
Section \ref{section: prelim} for a precise definition.} then the set $[T]$ of all cofinal 
branches through $T$ is a closed subset of ${^\kappa}\kappa$ that fails to satisfy the 
$\kappa$-perfect set property (see \cite{lucke_definability} for a proof of this; the idea 
of using Kurepa trees as a counterexample for the $\kappa$-perfect set property goes back 
at least to \cite{vaananen} and \cite{mekler_vaananen}). 
The fact that $\mathsf{PSP}_\kappa(X)$ holding for all 
closed $X \subseteq {^{\kappa}}\kappa$ implies that $\kappa^+$ is inaccessible in $L$ 
then follows from an argument of Solovay (cf.\ \cite[\S 4]{jech_trees}).

It is thus natural to also consider Kurepa trees in the context of item (1) above, leading 
to the following general question: for a fixed regular uncountable cardinal $\kappa$, 
given a $\kappa$-Kurepa tree $T$, under what circumstances is 
$[T]$ a continuous image of, or even a retract of, the space ${^{\kappa}}\kappa$. This and 
related questions were considered by L\"{u}cke and Schlicht in \cite{lucke_schlicht_descriptive}, 
where they obtained a number of interesting results, including the following:
\begin{itemize}
  \item If there is $\mu < \kappa$ such that $\mu^\omega \geq \kappa$, and $T \subseteq 
  {^{<\kappa}}\kappa$ is a $\kappa$-Kurepa tree with $|[T]| > \kappa^{<\kappa}$, then 
  $[T]$ is not a continuous image of ${^\kappa}\kappa$. In particular, assuming instances of 
  $\mathsf{GCH}$, the question has an easy negative answer if $\kappa$ is a successor of a cardinal 
  of countable cofinality, e.g., $\omega_1$.
  \item If $V = L$, then, for every regular uncountable cardinal $\kappa$ that is not the successor of 
  a cardinal of countable cofinality, there exists a $\kappa$-Kurepa tree $T \subseteq 
  {^{<\kappa}}\kappa$ such that $[T]$ is a retract of ${^\kappa}\kappa$.
  \item If $\mu < \kappa$ are regular uncountable cardinals, with $\kappa$ inaccessible, then 
  there is a forcing extension in which (1) $\kappa = \mu^+$, (2) there exists a $\kappa$-Kurepa 
  tree $T \subseteq {^{<\kappa}}\kappa$ such that $[T]$ is a retract of ${^\kappa}\kappa$; and 
  (3) there exists a $\kappa$-Kurepa tree $S \subseteq T$ such that $[S]$ is not a continuous 
  image of ${^\kappa}\kappa$.
  \item If there is a $\kappa$-Kurepa tree, then there is a $\kappa$-Kurepa tree 
  $T \subseteq {^{<\kappa}}\kappa$ such that $[T]$ is not a retract of ${^\kappa}\kappa$.
\end{itemize}

The work of L\"{u}cke and Schlicht left open a number of avenues for further research. 
One question, explicitly asked as \cite[Question 6.2]{lucke_schlicht_descriptive}, is whether 
$\CH$ together with the existence of an $\omega_2$-Kurepa tree implies the existence of 
an $\aleph_2$-Kurepa tree $S \subseteq {^{<\omega_2}}\omega_2$ such that $[S]$ is a continuous 
image of ${^{\omega_2}}\omega_2$. We provide a negative answer to this question (see Theorem 
B below).

We also undertake further explorations of topics connecting Kurepa trees, 
(variations on) the perfect set property, and continuous images of higher Baire spaces. 
Here we make use of a hierarchy of generalizations of perfectness introduced by 
V\"{a}\"{a}n\"{a}nen in \cite{vaananen} in terms of the existence of winning strategies for 
certain two-player games. In particular, given a regular cardinal $\kappa$, and an ordinal 
$\delta$ with $\omega \leq \delta \leq \kappa$, V\"{a}\"{a}n\"{a}nen defines a notion of 
\emph{$\delta$-perfectness} for subsets of ${^{\kappa}}\kappa$. These notions strengthen as 
$\delta$ increases, with $\omega$-perfectness being classical perfectness (i.e., closed and 
having no isolated points), and $\kappa$-perfectness being as above. 

Recall the classical Cantor-Bendixson theorem: If $E \subseteq {^\omega}\omega$ is closed, 
then there is a countable set $X \subseteq E$ such that $E \setminus X$ is perfect. This 
generalizes straightforwardly to higher Baire spaces: if $E \subseteq {^\kappa}\kappa$ is 
closed, then there is $X \subseteq E$ such that $|X| \leq \kappa^{<\kappa}$ and 
$E \setminus X$ is perfect. We prove in Section \ref{section: vaananen} that, if 
$E$ is a continuous image of ${^\kappa}\kappa$, then we can obtain a stronger 
conclusion.\footnote{For more on stronger versions of the Cantor-Bendixson theorem 
for higher Baire spaces, see \cite{vaananen}.}

In the other direction, recall that $\kappa$-Kurepa trees yield examples of closed 
subsets of ${^\kappa}\kappa$ of cardinality greater than $\kappa$ that do not contain 
$\kappa$-perfect subsets; in fact, if $\kappa = \mu^+$ and $T \subseteq {^{<\kappa}}\kappa$ 
is a $\kappa$-Kurepa tree, then $[T]$ does not contain a $(\mu+1)$-perfect subset.
In Section \ref{section: weak_kurepa}, we prove results showing 
that Kurepa trees are not necessary for the existence of such sets by producing models in 
which $\GCH$ holds, the Kurepa Hypothesis fails (i.e., there are no $\omega_1$-Kurepa trees), but there is a 
closed subset of ${^{\omega_1}}\omega_1$ of cardinality $\omega_2$ with no 
$(\omega+1)$-perfect subset. In fact, we prove the following stronger theorem:\footnote{Recall 
that a \emph{weak Kurepa tree} is a tree of height and size $\omega_1$ with at least 
$\omega_2$-many cofinal branches.}

\begin{thma}
  Suppose that there is an inaccessible cardinal. Then there is a forcing extension in which
  \begin{enumerate}
    \item $\GCH$ holds;
    \item the Kurepa Hypothesis fails;
    \item there is a weak Kurepa tree that does not contain a copy of ${^{<\omega+1}}2$.
  \end{enumerate}
\end{thma}

We actually produce two models witnessing the conclusion of Theorem A. In the first, every 
weak Kurepa tree contains an Aronszajn (or even Suslin) subtree. In the second, there is a 
weak Kurepa tree that contains neither a copy of ${^{<\omega+1}}2$ nor an Aronszajn subtree.

In Section \ref{sect: cts_image}, we provide a negative answer to the aforementioned 
question of L\"{u}cke and Schlicht. In addition, we show that, in the model we 
construct, ${^{\omega_2}}\omega_2$ satisfies the strongest possible analogue of the 
Cantor-Bendixson theorem compatible with the existence of an $\omega_2$-Kurepa tree 
(recall that, if $T \subseteq {^{<\omega_2}}\omega_2$ is an $\omega_2$-Kurepa tree), 
then $[T]$ cannot contain an $(\omega_1+1)$-perfect subset. In particular, we prove 
the following theorem.

\begin{thmb}
  If there exists an inaccessible cardinal, then there is a forcing extension in which 
  $\GCH$ holds, there exists and $\omega_2$-Kurepa tree, and, for every $\omega_2$-Kurepa tree 
  $S \subseteq {^{<\omega_2}}\omega_2$, $[S]$ is not a continuous image of ${^{\omega_2}}\omega_2$. In addition, in the forcing extension, for every closed subset 
  $E \subseteq {^{\omega_2}}\omega_2$, there is $X \subseteq E$ with $|X| \leq \omega_2$ 
  such that $E \setminus X$ is $\omega_1$-perfect.
\end{thmb}

In Section \ref{section: full}, we prove some new results about full trees. 
A tree $T$ is \emph{full} if, for every limit ordinal $\beta$ below the height of $T$, 
there is at most one branch through $T \restriction \beta$ that is not continued at level 
$\beta$ (see Section \ref{section: full} for a precise definition). Kunen asked whether 
there could consistently exist a full $\kappa$-Suslin tree for some regular uncountable cardinal 
$\kappa$ (cf.\ \cite{miller_list}). This was answered positively by Shelah in \cite{shelah_full} 
for inaccessible $\kappa$ and recently by Rinot, Yadai, and You in \cite{rinot_full} for successors 
of regular uncountable cardinals. Here, we are interested in full, splitting trees that may contain 
some cofinal branches. For example, in Section \ref{section: full}, we prove the following theorem
about full trees of height $\omega_1$:

\begin{thmc}
  \begin{enumerate}
    \item If $\diamondsuit$ holds, then, for every cardinal $\nu \in \omega \cup \{\omega, \omega_1, 
    2^{\omega_1}\}$, there is a normal, full, splitting tree $T \subseteq {^{<\omega_1}}\omega_1$ 
    with exactly $\nu$-many cofinal branches.
    \item $\CH$ does not suffice for the conclusion of clause (1). In particular, 
    $\CH$ is compatible with the assertion that every full, splitting tree of height 
    $\omega_1$ contains a copy of ${^{<\omega_1}}2$, and hence has $2^{\omega_1}$-many cofinal 
    branches.
  \end{enumerate}
\end{thmc}

We then move up a cardinal to prove the consistency of the existence of a normal, splitting, full, 
superthin $\omega_2$-Kurepa tree. We refer the reader to Section \ref{section: full} for the definition 
of \emph{superthin}; we only note that superthin Kurepa trees play a central role in the 
investigations of L\"{u}cke and Schlicht in \cite{lucke_schlicht_descriptive}. In particular, they 
prove there that if there exists a superthin $\kappa$-Kurepa tree, then there is a 
superthin $\kappa$-Kurepa tree $T \subseteq {^{<\kappa}}\kappa$ such that $[T]$ is a retract of 
${^{\kappa}}\kappa$. In Section \ref{section: superthin}, we prove the following result, 
producing a model having $\omega_2$-Kurepa trees with very different behavior from that 
produced in Theorem B:

\begin{thmd}
  Suppose that $\GCH$ holds and $\kappa$ is the successor of a regular uncountable cardinal. 
  Then there is a cofinality-preserving forcing extension in which
  \begin{enumerate}
    \item $\GCH$ holds;
    \item there exists a $\kappa$-Kurepa tree;
    \item every $\kappa$-Kurepa tree contains a normal superthin $\kappa$-Kurepa subtree.
  \end{enumerate}
\end{thmd}

Finally, in Section \ref{section: questions}, we record some closing remarks and a few questions that remain open.

\subsection{Notation and conventions} Our notation is for the most part standard. 
We refer the reader to \cite{jech} for undefined notions and notations in set theory.
We let $\mathrm{Card}$ denote the class of all cardinals.
Given a well-ordered set $X$, we denote its order type by $\otp(X)$. If 
$X$ and $Y$ are sets, $f:X \ra Y$, and $A \subseteq X$, then both $f[A]$ and 
$f``A$ denote the pointwise image of $A$, i.e., $\{f(a) \mid a \in A\}$.

Given ordinals $\alpha$ and $\beta$, we let ${^{\alpha}}\beta$ denote the set of 
all functions 
$f:\alpha \ra \beta$, and we let ${^{<\alpha}}\beta$ denote $\bigcup_{\eta < \alpha} 
{^{\eta}}\beta$. We use ${^{\leq \alpha}}\beta$ and ${^{<(\alpha+1)}}\beta$ 
interchangeably. If $\sigma,\tau \in {^{<\alpha}}\beta$, then we let $\sigma 
\sqsubseteq \tau$ denote the assertion that $\sigma$ is an \emph{initial segment} of 
$\tau$, i.e., $\dom(\sigma) \leq \dom(\tau)$ and $\sigma = \tau \restriction \dom(\sigma)$.
Given a nonempty set $A \subseteq {^{\leq \alpha}}\beta$, we let 
$\bigwedge A \in {^{\leq \alpha}}\beta$ denote the (unique) 
$\sqsubseteq$-maximal element of ${^{\leq \alpha}}\beta$ that is an 
initial segment of every element of $A$, i.e., 
\[
  \bigwedge A = \bigcup \{\sigma \in {^{\leq \alpha}}\beta \mid 
  \forall \tau \in A ~ (\sigma \sqsubseteq \tau)\}.
\]
If $\sigma$ and $\tau$ are two functions whose domains are ordinals 
$\alpha$ and $\beta$, respectively, then $\sigma ^\frown \tau$ denotes the 
concatenation of $\sigma$ and $\tau$, i.e., the function $\rho$ with domain 
$\alpha + \beta$ such that $\rho(\eta) = \sigma(\eta)$ for all $\eta < \alpha$ 
and $\rho(\alpha + \xi) = \tau(\xi)$ for all $\xi < \beta$. We will sometimes think 
of functions with ordinal domains as sequences, e.g., a sequence of the form 
$\langle \gamma \rangle$ will be thought of as a function $\sigma$ with domain $1$ 
and $\sigma(0) = \gamma$. We will sometimes write, e.g., $\sigma ^\frown \gamma$ 
instead of $\sigma ^\frown \langle \gamma \rangle$ when there is no risk of confusion.

We make use of various standard forcing notions throughout the paper. In particular, 
if $\kappa$ is a regular infinite cardinal, then $\Add(\kappa, 1)$ is the forcing to add 
a single Cohen subset of $\kappa$. If, in addition, $\mu > \kappa$ is a 
cardinal, then $\Coll(\kappa, {<}\mu)$ is the L\'{e}vy collapse that collapses 
every cardinal in the interval $(\kappa, \mu)$ to have cardinality $\kappa$. We 
assume that every forcing notion $\P$ has a maximum element, which we denote $1_{\P}$.
If we define a forcing notion $\P$ that does not have a maximum element, then we 
implicitly add $\emptyset$ as $1_{\P}$.

\section{Combinatorial and topological preliminaries} \label{section: prelim}

In this section, we introduce some of the basic notions forming the subject matter of 
this paper, particularly regarding trees, higher Baire spaces, and forcing.

\begin{definition}
  A \emph{tree} is a partial order $(T, \leq_T)$ such that, for every $t \in T$, 
  the set $\pred_T(t) := \{s \in T \mid s <_T t\}$ is well-ordered by $\leq_T$. 
  We will often abuse notation and simply refer to the tree by $T$, without 
  explicitly mentioning the tree order $\leq_T$.
  Given $t \in T$, the \emph{height} of $t$ in $T$ is 
  $\height_T(t) := \otp(\pred_T(t))$. Given an ordinal $\alpha$, the 
  \emph{$\alpha^{\mathrm{th}}$ level} of $T$, denoted $T_\alpha$, 
  is $\{t \in T \mid \height_T(t) = \alpha\}$. The \emph{height} of $T$, 
  denoted $\height(T)$, is the least ordinal $\beta$ such that $T_\beta = \emptyset$.
  A \emph{subtree} of $T$ is a $\leq_T$-downward closed subset of $T$, with 
  the tree order inherited from $\leq_T$.
  
  We say that a tree $T$ is \emph{normal} if it satisfies the following two conditions:
  \begin{itemize}
    \item for all $\alpha < \beta < \height(T)$ and all $s \in T_\alpha$, there is 
    $t \in T_\beta$ such that $s <_T t$; and
    \item for all limit ordinals $\alpha < \height(T)$ and all $s,t \in T_\alpha$, 
    if $\mathrm{pred}_T(s) = \mathrm{pred}_T(t)$, then $s = t$.
  \end{itemize}
  
  A \emph{branch} through $T$ is a maximal linearly ordered subset of $T$. 
  We say that a branch $b$ through $T$ is a \emph{cofinal branch} through $T$ 
  if $b \cap T_\alpha \neq \emptyset$ for all $\alpha < \height(T)$. The set of 
  all cofinal branches through $T$ is denoted by $[T]$.
  
  Given two nodes $s,t \in T$, we write $s \perp t$ to denote the assertion that 
  $s$ and $t$ are $\leq_T$-incomparable. An \emph{antichain} of $T$ is a set $A \subseteq T$ 
  such that $s \perp t$ for all distinct $s,t \in A$.
\end{definition}

\begin{definition}
  Suppose that $\kappa$ is an infinite cardinal. A tree $T$ is a \emph{$\kappa$-tree} 
  if $\height(T) = \kappa$ and $|T_\alpha| < \kappa$ for every $\alpha < \kappa$. 
  A \emph{$\kappa$-Aronszajn tree} is a $\kappa$-tree with no cofinal branches, and 
  a \emph{$\kappa$-Suslin tree} is a $\kappa$-tree with no cofinal branches and no 
  antichains of cardinality $\kappa$.  	  
  
  A \emph{$\kappa$-Kurepa tree} is a $\kappa$-tree $T$ with at least $\kappa^+$-many 
  cofinal branches. A \emph{weak $\kappa$-Kurepa tree} is a tree $T$ of height and size 
  $\kappa$ with at least $\kappa^+$-many cofinal branches.
  
  The \emph{$\kappa$-Kurepa Hypothesis} ($\KH_\kappa$) is the assertion that there exists 
  a $\kappa$-Kurepa tree. The \emph{weak $\kappa$-Kurepa Hypothesis} ($\wKH_\kappa$) is the 
  assertion that there exists a weak Kurepa tree. If the parameter $\kappa$ is omitted in 
  any of the above, then it should be understood that $\kappa = \omega_1$, e.g., 
  a \emph{Suslin tree} is an $\omega_1$-Suslin tree, 
  a \emph{Kurepa tree} is an $\omega_1$-Kurepa tree, and $\KH$ is $\KH_{\omega_1}$.
\end{definition}

\begin{definition}
Let $T$ be a normal tree. For $t,s \in T$, the meet $s \wedge t$ is the unique node $r \le s, t$ such that there is no $r' >r$ with $r' \le s,t$.
\end{definition}

We will naturally interpret ${^{<\alpha}}\beta$ as a tree by setting, for all $\sigma, \tau \in 
{^{<\alpha}}\beta$, $\sigma \leq \tau$ if and only if $\sigma$ is an initial segment of 
$\tau$ (denoted $\sigma \sqsubseteq \tau$). When we say that a tree $T$ is a 
subtree of ${^{<\alpha}}\beta$, we will implicitly assume that $T$ has height 
$\alpha$. In particular, in this context, the set $[T]$ of cofinal branches through 
$T$ can and will be identified with the set
\[
  \{b \in {^{\alpha}}\beta \mid \forall \eta < \alpha ~ (b \restriction \eta 
  \in T)\}.
\]

If $\kappa$ and $\lambda$ are cardinals, with $\kappa$ 
regular, then we implicitly interpret ${^{\kappa}}\lambda$ as a topological space, where each 
copy of $\lambda$ is given the discrete topology and the product is given the $({<}\kappa)$-box 
topology. In other words, the topology on ${^{\kappa}}\lambda$ is generated by all basic open sets 
of the form
\[
  N_\sigma = \{x \in {^{\kappa}}\lambda \mid \sigma \sqsubseteq x\},
\]
where $\sigma \in {^{<\kappa}}\lambda$. We will be most interested in the space ${^\kappa}\kappa$ 
for a regular uncountable cardinal $\kappa$; this space is typically called the 
\emph{higher Baire space at $\kappa$}.

Fix for now cardinals $\kappa$ and $\lambda$, with $\kappa$ regular. Given a nonempty set 
$X \subseteq {^{\kappa}}\lambda$, let 
\[
  T(X) := \{x \restriction \alpha \mid x \in X \text{ and } \alpha < \kappa\}.
\]
Note that $T(X)$ is a subtree of ${^{<\kappa}}\lambda$. 

In the other direction, given a tree $T \subseteq {^{<\kappa}}\lambda$, 
it is a well-known and easily verified 
fact that $[T]$ is a closed subset of ${^{\kappa}}\lambda$ and, moreover, all closed 
subsets of ${^{\kappa}}\lambda$ are of this form:

\begin{fact}
  Suppose that $\kappa$ and $\lambda$ are cardinals, with $\kappa$ regular. Then a subset 
  $X \subseteq {^{\kappa}}\lambda$ is closed if and only if it is of the form 
  $[T]$ for some tree $T \subseteq {^{<\kappa}}\lambda$.
\end{fact}

We now recall a number of relevant definitions from \cite{schlicht_sziraki}.

\begin{definition} \label{def: basic_trees}
  Suppose that $T$ is a subtree of ${^{<\kappa}}\lambda$.
  \begin{enumerate}
    \item $T$ is \emph{normal} if for all $\sigma \in T$ and all $\alpha < \kappa$, there is 
    $\tau \in T$ such that $\sigma \sqsubseteq \tau$ and $\dom(\tau) \geq \alpha$.
    \item Given $\sigma, \tau \in T$, we let $\sigma \perp \tau$ denote the assertion that 
    $\sigma$ and $\tau$ are not comparable in $T$.
    \item $T$ is \emph{splitting} if, for all $\sigma \in T$, if $|\sigma| + 1 < \height(T)$, 
    then there are distinct $i,j < \lambda$ such that $\sigma^\frown i, 
    \sigma^\frown j \in T$. It is \emph{infinitely splitting} if, for all such
    $\sigma \in T$, there are infinitely many $i < \lambda$ such that 
    $\sigma^\frown i \in T$.
    \item $T$ is \emph{cofinally splitting} if, for every $\sigma \in T$, there are 
    $\tau_0, \tau_1 \in T$ such that $\sigma \sqsubseteq \tau_0$, $\sigma \sqsubseteq \tau_1$, 
    and $\tau_0 \perp \tau_1$.
    \item Given a cardinal $\mu \leq \kappa$, $T$ is ${<\mu}$-closed if, for every $\eta < \mu$ 
    and every increasing sequence $\langle \sigma_\xi \mid \xi < \eta \rangle$ from $T$, we have 
    $\bigcup \{\sigma_\xi \mid \xi < \eta\} \in T$.
    \item $T$ is \emph{$\kappa$-perfect} if it is normal, cofinally splitting, and ${<}\kappa$-closed.
    \item \label{def: kappa_perfect} 
    A nonempty closed subset $X \subseteq {^{\kappa}}\lambda$ is said to be 
    $\kappa$-perfect if $T(X)$ is $\kappa$-perfect.
  \end{enumerate}
\end{definition}

\begin{definition}
  Suppose that $S$ and $T$ are trees and $\iota:S \ra T$.
  \begin{enumerate}
    \item $\iota$ is \emph{strict order preserving} if, for all $\sigma, \tau \in S$, if 
    $\sigma < \tau$, then $\iota(\sigma) < \iota(\tau)$.
    \item $\iota$ is \emph{$\perp$-preserving} if, for all $\sigma, \tau \in S$, if 
    $\sigma \perp \tau$, then $\iota(\sigma) \perp \iota(\tau)$.
    \item If $\iota$ is both strict order preserving and $\perp$-preserving, 
    then we call $\iota$ an \emph{isomorphic embedding} of $S$ into $T$. 
    We say that $T$ \emph{contains a copy of $S$} if there is an isomorphic embedding 
    $\iota : S \ra T$; in this situation, the image $\iota[S]$ will be referred to 
    as \emph{a copy of $S$ in $T$}.
  \end{enumerate}
\end{definition}

\begin{remark}
  We note that, in the above definition, the maps $\iota : S \ra T$ need not 
  preserve levels, i.e., there may be $\sigma \in S$ such that 
  $\height_S(\sigma) \neq \height_T(\iota(\sigma)$. In particular, a copy of 
  $S$ in $T$ need not be a subtree of $T$, as it need not be downward closed.
  Also note that $S$ and $T$ need not have the same height, e.g., it could 
  be the case that $S$ is a subtree of ${^{<\alpha}}\beta$ and $T$ is a subtree of 
  ${^{<\gamma}}\delta$ for cardinals $\alpha \neq \gamma$ and $\beta \neq \delta$. 
  Note, however, that if $\iota : S \ra T$ is strict order preserving, then we will 
  have $\height_S(s) \leq \height_T(\iota(s))$ for all $s \in S$.
\end{remark}

\begin{definition} \label{def: psp}
  Let $\kappa$ and $\lambda$ be cardinals, with $\kappa$ regular, and suppose that 
  $X \subseteq {^{\kappa}}\lambda$. We say that $X$ has the \emph{$\kappa$-perfect set property}, 
  denoted $\PSP_\kappa(X)$, if either $|X| \leq \kappa$ or $X$ contains a $\kappa$-perfect 
  subset.
\end{definition}

Note that, in the classical Baire space ${^\omega}\omega$, a nonempty closed subset 
of ${^\omega}\omega$ is perfect in the classical sense (i.e., is 
closed with no isolated points) if and only if it is $\omega$-perfect in the sense 
of Definition \ref{def: basic_trees}(\ref{def: kappa_perfect}), and a set 
$X \subseteq {^\omega}\omega$ has the classical perfect set property if and only if 
it satisfies $\PSP_\omega(X)$ as in Definition \ref{def: psp}.

The following well-known proposition is easily proven.

\begin{proposition}
  Let $\kappa$ and $\lambda$ be cardinals, with $\kappa$ regular, and suppose that 
  $X \subseteq {^{\kappa}}\lambda$ is closed. Then $\PSP_\kappa(X)$ holds if and only if 
  $|X| \leq \kappa$ or $T(X)$ contains a copy of ${^{<\kappa}}2$.
\end{proposition}

We end this section with some preliminaries on forcing. We first recall 
the definition of \emph{strategic closure}. 

\begin{definition}
  Let $\bb{P}$ be a partial order and let $\beta$ be an ordinal.
  \begin{enumerate}
    \item $\Game_\beta(\bb{P})$ is the two-player game in which Players I and II alternate 
    playing conditions from $\bb{P}$ to attempt to construct a $\leq_{\bb{P}}$-decreasing 
    sequence $\langle p_\alpha \mid \alpha < \beta \rangle$. Player I plays at odd stages, and 
    Player II plays at even stages (including limit stages). Player II is required to play 
    $p_0 = 1_{\bb{P}}$. If, during the course of play, a limit ordinal $\alpha < \beta$ is 
    reached such that $\langle p_\eta \mid \eta < \alpha \rangle$ has no lower bound in 
    $\bb{P}$, then Player I wins. Otherwise, Player II wins.
    \item $\bb{P}$ is said to be \emph{$\beta$-strategically closed} if Player II has a winning 
    strategy in $\Game_\beta(\bb{P})$.
  \end{enumerate}
\end{definition}

We now recall the following lemma, due to Silver. The lemma is usually stated 
with the hypothesis that $\bb{R}$ is $\tau^+$-closed (cf.\ \cite[Lemma 4]{U:Fragility}),
but the standard proof is easily seen to work under the weaker assumption that 
$\bb{R}$ is $(\tau+1)$-strategically closed, so we leave it to the reader.

\begin{lemma} \label{lemma: strat_closed_silver}
  Suppose that $\tau < \nu$ are infinite regular cardinals, with $2^\tau \geq \nu$. 
  Suppose that $T$ is a $\nu$-tree and $\bb{R}$ is a forcing poset that is 
  $(\tau + 1)$-strategically closed. Then forcing with $\bb{R}$ cannot add a branch 
  of length $\nu$ through $T$, i.e., every cofinal branch through $T$ in 
  $V^{\bb{R}}$ is in $V$.
\end{lemma}

\section{V\"{a}\"{a}n\"{a}nen's game} \label{section: vaananen}

The following game was introduced by V\"a\"an\"anen in \cite{vaananen}
in order to generalize the notion of perfectness.

\begin{definition}
  Suppose that $\kappa$ is a regular cardinal, $E$ is a subset of 
  ${^\kappa}\kappa$, $x_0 \in {^\kappa}\kappa$, and $\delta \leq \kappa$
  is an infinite ordinal. Then the two-player game $G_\kappa(E, x_0, \delta)$ 
  is defined as follows. The game consists of rounds indexed by ordinals 
  $\xi$ with $1 \leq \xi < \delta$. In round $\xi$, Player I first plays an 
  ordinal $\alpha_\xi < \kappa$, and then Player II plays an element 
  $x_\xi \in E$. The plays must satisfy the following requirements:
  \begin{itemize}
    \item $\langle \alpha_\xi \mid 1 \leq \xi < \delta \rangle$ is an increasing, 
    continuous sequence of ordinals;
    \item for all $0 \leq \eta < \xi < \delta$, we have 
    \begin{itemize}
      \item $x_\eta \neq x_\xi$;
      \item $x_\eta \restriction \alpha_{\eta + 1} = x_\xi \restriction 
      \alpha_{\eta+1}$.
    \end{itemize}
  \end{itemize}
  Player II wins if they can successfully play $x_\xi$ for all $1 \leq \xi < 
  \delta$; otherwise, Player I wins.
\end{definition}

\begin{definition} \label{def: game_perfect}
  Suppose that $\kappa$ is a regular cardinal, $E$ is a subset of 
  ${^\kappa}\kappa$, and $\delta \leq \kappa$ is an infinite ordinal. 
  Then we say that $E$ is \emph{$\delta$-perfect} if it is closed and,
  for every $x_0 \in E$, Player II has a winning strategy in 
  $G_\kappa(E, x_0, \delta)$.
\end{definition}

Note that $G_\kappa(E, x_0, \delta)$ becomes harder for Player II to win 
as $\delta$ increases; thus, if $\omega \leq \delta_0 \leq \delta_1 \leq \kappa$ 
and $E \subseteq {^{\kappa}}\kappa$ is $\delta_1$-perfect, then it is also 
$\delta_0$-perfect.

\begin{remark}
  Given a regular uncountable cardinal $\kappa$, there is a slight discrepancy between the 
  notion of $\kappa$-perfect isolated in Definition \ref{def: game_perfect} 
  notion of $\kappa$-perfect isolated in Definition \ref{def: basic_trees}(\ref{def: kappa_perfect}), 
  which we will call \emph{strongly $\kappa$-perfect} when we need to distinguish it from the 
  notion in Definition \ref{def: game_perfect}. 
  For example, as noted in \cite{vaananen}, the set of all $x \in {^{\kappa}}\kappa$ such that 
  $x(\alpha) = 0$ for only finitely many $\alpha$ is $\kappa$-perfect but not strongly 
  $\kappa$-perfect. However, it is readily verified that every strongly $\kappa$-perfect 
  subset of ${^{\kappa}}\kappa$ is $\kappa$-perfect and, conversely, every $\kappa$-perfect 
  subset of ${^{\kappa}}\kappa$ contains a strongly $\kappa$-perfect subset. In particular, the 
  $\kappa$-perfect set property is equivalent when defined with either notion. Hence, for the purposes 
  of this paper, it will not be necessarily to distinguish between the two.
\end{remark}

In light of the above remark, the following definition generalizes Definition \ref{def: psp} 
in the setting of ${^\kappa}\kappa$.

\begin{definition}
  Suppose that $\kappa$ is a regular infinite cardinal, $\delta \leq \kappa$ is an infinite 
  ordinal, and $X \subseteq {^\kappa}\kappa$. We say that $X$ has the \emph{$\delta$-perfect 
  set property}, denoted $\PSP_\delta(X)$, if either $|X| \leq \kappa$ or $X$ contains 
  a $\delta$-perfect subset.
\end{definition}

The following proposition is readily verified; we leave the proof to the reader.

\begin{proposition} \label{prop: copy}
  Suppose that $\kappa$ is a regular cardinal, $\delta \leq \kappa$ is an 
  infinite ordinal, and $T \subseteq {^{<\kappa}}\kappa$ is a tree. If 
  $[T]$ is $\delta$-perfect, then $T$ contains a copy of ${^{<\delta}}2$. \qed
\end{proposition}

Note that a set $E \subseteq {^\kappa}\kappa$ is $\omega$-perfect if and only 
if it is perfect in the classical sense. Unlike the case with ${^\omega}\omega$, 
it is not necessarily the case that perfect sets have full cardinality. For example, 
if $\kappa$ is regular and uncountable and 
\[
  E = \{x \in {^\kappa}\kappa \mid |\{\alpha < \kappa \mid x(\alpha) \neq 0\}| 
  < \kappa\},
\] 
then $E$ is readily seen to be a perfect subset of ${^\kappa}\kappa$. However, 
$|E| = \kappa^{<\kappa}$, so if $\kappa^{<\kappa} < 2^\kappa$, then 
$|E| < 2^\kappa$. Nonetheless, we do recover a version of the Cantor-Bendixson 
theorem at higher $\kappa$:

\begin{proposition} \label{prop: kpsp}
  Suppose that $E \subseteq {^\kappa}\kappa$ is closed. Then there is 
  $X \subseteq E$ such that $|X| \leq \kappa^{<\kappa}$ and 
  $E \setminus X$ is perfect.
\end{proposition}

\begin{proof}
  Let $\Sigma = \{\sigma \in {^{<\kappa}}\kappa \mid |E \cap N_\sigma| 
  \leq \kappa^{<\kappa}\}$, and let $X = \bigcup \{E \cap N_\sigma \mid 
  \sigma \in \Sigma\}$. Then $X$ is as desired.
\end{proof}

\begin{corollary} \label{cor: omega_psp}
  If $\kappa$ is an infinite regular cardinal such that $\kappa^{<\kappa} = \kappa$, 
  then every closed set $E \subseteq {^\kappa}\kappa$ satisfies the $\omega$-perfect set 
  property.
\end{corollary}

As noted in the introduction, in ${^\omega}\omega$, every nonempty closed set is a continuous 
image of the entire space. For uncountable $\kappa$, this is no longer the case. However, if we 
know that a closed set $E$ is a continuous image of ${^\kappa}\kappa$, then we can 
slightly improve upon the conclusion of the preceding proposition. The proof of the following 
theorem is a variation on that of \cite[Theorem 1.1]{lucke_schlicht_descriptive}

\begin{theorem} \label{thm: omega_plus_one_perfect}
  Suppose that $\kappa$ is an uncountable regular cardinal and $E \subseteq 
  {^\kappa}\kappa$ is a closed set that is a continuous image of 
  ${^\kappa}\kappa$. Then there is $X \subseteq E$ with $|X| \leq 
  \kappa^{<\kappa}$ such that $E \setminus X$ is closed and, for every $x_0 \in E \setminus X$, 
  Player II has a winning strategy in $G_\kappa(E, x_0, \omega+1)$.  
\end{theorem}

\begin{proof}
  Let $f: {^\kappa}\kappa \ra E$ be a continuous surjection. Let $X$ be as in the proof of 
  Proposition \ref{prop: kpsp}. Let $X_0$ be the set of $x \in E$ such that, 
  for every $y \in f^{-1}\{y\}$, there is $\alpha < \kappa$ such that 
  $|f[N_{y \restriction \alpha}]| \leq \kappa^{<\kappa}$.
  
  \begin{claim} \label{claim: small_x_0}
    $|X_0| \leq \kappa^{<\kappa}$.
  \end{claim}
  
  \begin{proof}
    Let 
    \[
      \Sigma = \{\sigma \in {^{<\kappa}}\kappa \mid |f[N_\sigma]| \leq 
      \kappa^{<\kappa}\},
    \]
    and let $Y = \bigcup \{f[N_\sigma] \mid \sigma \in \Sigma\}$. Then 
    $|Y| \leq \kappa^{<\kappa}$ and $X_0 \subseteq Y$.
  \end{proof}
  
  \begin{claim} \label{claim: closure}
    $X \subseteq X_0$, and $E \setminus X$ is the closure of $E \setminus X_0$ in 
    ${^{\kappa}}\kappa$.
  \end{claim}
  
  \begin{proof}
    The facts that $X \subseteq X_0$ and $E \setminus X$ is closed follow immediately from 
    the definitions. To show that $E \setminus X$ is the closure of $E \setminus X_0$, it 
    suffices to show that, for every $x \in E \setminus X$ and every open set $U$ 
    with $x \in U$, there is $y \in U \cap (E \setminus X_0)$. Thus, fix such $x$ and $U$. 
    Without loss of generality, we may assume that $U = N_{x \restriction \alpha}$ for 
    some $\alpha < \kappa$. Since $x \notin X$, we know that $|U \cap E| > \kappa^{<\kappa}$. 
    Thus, by Claim \ref{claim: small_x_0}, we know that $U \cap (E \setminus X_0) \neq \emptyset$.
  \end{proof}
  
  It remains to show that 
  Player II has a winning strategy in $G_\kappa(E, x_0, \omega+1)$ 
  for all $x_0 \in E \setminus X$. We first establish the following claim.
  
  \begin{claim} \label{claim: diverse_image}
    Suppose that $U \subseteq {^{\kappa}}\kappa$ is an open set such that $|f[U]| > 
    \kappa^{<\kappa}$. Then there is $y \in U$ such that $|f[N_{y \restriction \alpha}]| 
    > \kappa^{<\kappa}$ for all $\alpha < \kappa$.
  \end{claim}
  
  \begin{proof}
    If not, then, for every $y \in U$, there is $\alpha_y < \kappa$ such that $|f[N_{y \restriction 
    \alpha_y}]| \leq \kappa^{<\kappa}$. Then we have $|f[U]| \leq | \bigcup \{ 
    f[N_{y \restriction \alpha_y}] \mid y \in U\}| \leq \kappa^{<\kappa}$, which is 
    a contradiction.
  \end{proof}
  
  \begin{claim} \label{claim: noniso}
    Suppose that $x \in E \setminus X$ and $\alpha < \kappa$. Then there is $x' \in 
    E \setminus X_0$ such that $x \neq x'$ and $x'\restriction \alpha = x \restriction \alpha$.
  \end{claim}
  
  \begin{proof}
    If $x \in X_0$ then we can choose any $x'$ in the set $(E \setminus X_0) \cap 
    N_{x \restriction \alpha}$, which is nonempty by Claim \ref{claim: closure}. Thus, we can 
    assume that $x \notin X_0$. Choose $y \in {^{\kappa}}\kappa$ such that $f(y) = x$ 
    and $|f[N_{y \restriction \beta}]| > \kappa^{<\kappa}$ for all $\beta < \kappa$. By the 
    continuity of $f$, we can find $\beta < \kappa$ such that $f[N_{y \restriction \beta}] 
    \subseteq N_{x \restriction \alpha}$. Let $U = N_{y \restriction \beta} \setminus 
    f^{-1}\{x\}$. Then $U$ is an open set such that $|f[U]| > \kappa^{<\kappa}$. 
    Therefore, by Claim \ref{claim: diverse_image} we can find $y' \in U$ such that 
    $|f[N_{y' \restriction \gamma}]| > \kappa^{<\kappa}$ for all $\gamma < \kappa$. 
    Then $x' = f(y')$ is as desired.
  \end{proof}
  
  Fix $x_0 \in E \setminus X$. We will describe a winning strategy for Player 
  II in $G_\kappa(E \setminus X, x_0, \omega+1)$. In the course of the game, 
  as the players play the sequences $\langle \alpha_n \mid 1 \leq n \leq \omega \rangle$ 
  and $\langle x_n \mid 1 \leq n \leq \omega \rangle$,
  Player II will also construct a sequence $\langle y_n \mid 1 \leq n < \omega \rangle$ 
  of elements of ${^\kappa}\kappa$
  and increasing sequences $\langle \beta_n \mid 1 \leq n < \omega \rangle$ and 
  $\langle \gamma_n \mid 1 \leq n < \omega \rangle$ of ordinals 
  below $\kappa$ such that, for all $1 \leq n < \omega$, we have
  \begin{itemize}
    \item $f(y_n) = x_n \in E \setminus X_0$;
    \item $|f[N_{y_n \restriction \gamma}]| > \kappa^{<\kappa}$ for all $\gamma < \kappa$;
    \item $f[N_{y_n \restriction \beta_n}] \subseteq N_{x_n \restriction 
    \beta_n}$;
    \item $\alpha_n \leq \beta_n$;
    \item $x_n \restriction \beta_n \neq x_{n-1} \restriction \beta_n$;
    \item $\max\{\alpha_{n+1}, \beta_n\} \leq \gamma_n \leq \beta_{n+1}$;
    \item $f[N_{y_n \restriction \gamma_n}] \subseteq N_{x_n \restriction \alpha_{n+1}}$;
    \item $y_{n+1} \restriction \gamma_n = y_n \restriction \gamma_n$.
  \end{itemize}   
  We begin by describing Player II's first play. After Player I plays an ordinal 
  $\alpha_1$, apply Claim \ref{claim: noniso} to find $x_1 \in E \setminus X_0$ 
  such that $x_1 \neq x_0$ and $x_1 \restriction \alpha_1 = x_0 \restriction \alpha_0$.  
  Then fix $y_1 \in {^\kappa}\kappa$ such that $f(y_1) = x_1$ and $|f[N_{y_1 \restriction \gamma}]| 
  > \kappa^{<\kappa}$ for all $\gamma < \kappa$. Finally, using the continuity of $f$ 
  and the fact that $\kappa$ has uncountable cofinality fix $\beta_1$ with $\alpha_1 \leq 
  \beta_1 < \kappa$ such that $f[N_{y_1 \restriction \beta_1}] \subseteq N_{x_1 \restriction 
  \beta_1}$ and $x_1 \restriction \beta_1 \neq x_0 \restriction \beta_0$.
  
  Now suppose that $1 \leq n < \omega$ 
  and $\langle x_0, \alpha_1, x_1, \ldots, \alpha_{n+1} \rangle$ is a partial play 
  of the game, with Player II playing so far according to the strategy that we 
  are about to describe, and that Player II has also specified 
  $\langle (y_i, \beta_i) \mid 1 \leq i \leq n \rangle$. We specify how to choose 
  $x_{n+1}$, as well as $\gamma_n$, $\beta_{n+1}$, and $y_{n+1}$. First, choose $\gamma_n \geq 
  \max\{\alpha_{n+1}, \beta_n\}$ such that $f[N_{y_n \restriction \gamma_n}] 
  \subseteq N_{x_n \restriction \alpha_{n+1}}$, and let
  \[
    U = N_{y_n \restriction \gamma_n} \setminus \bigcup\{f^{-1}\{x_i\} \mid i \leq n\}.
  \]
  Then $U$ is an open set such that $|f[U]| > \kappa^{<\kappa}$, so, by Claim \ref{claim: diverse_image}, 
  we can choose $y_{n+1} \in U$ such that $|f[N_{y_{n+1} \restriction \gamma}]| > \kappa^{<\kappa}$ 
  for all $\gamma < \kappa$. Let $x_{n+1} = f(y_{n+1})$, and choose $\beta_{n+1} > \gamma_n$ such that
  \begin{itemize}
    \item $f[N_{y_{n+1} \restriction \beta_{n+1}}] \subseteq N_{x_{n+1} \restriction 
    \beta_{n+1}}$; and
    \item $x_{n+1} \restriction \beta_{n+1} \neq x_n \restriction \beta_{n+1}$.
  \end{itemize}
  It is readily verified that this satisfies all of the requirements of the construction, and 
  we can move on to the next round of the game.
  
  This completely describes Player II's strategy at rounds indexed by natural 
  numbers. It remains to show that, if they play according to this strategy, 
  then they guarantee that they will be able to play in round $\omega$. To this 
  end, suppose that $\langle (\alpha_n, x_n) \mid 1 \leq n < \omega \rangle$ 
  is an initial segment of the game of length $\omega$, with Player II playing 
  according to the described strategy. Suppose that 
  $\langle (\beta_n, y_n) \mid 1 \leq n < \omega \rangle$ are the auxiliary objects 
  specified by this strategy. 
  
  Now let $\beta_\omega = \sup\{\beta_n \mid n < \omega\} = \sup\{\gamma_n \mid n < \omega\} \geq 
  \alpha_\omega = \sup\{\alpha_n \mid n < \omega\}$. We know that, for all 
  $1 \leq n < \omega$, we have $y_{n+1} \restriction \gamma_n = y_n \restriction 
  \gamma_n$. We can therefore find a $y_\omega \in {^\kappa}\kappa$ such 
  that, for all $1 \leq n < \omega$, we have $y_\omega \restriction \gamma_n 
  = y_n \restriction \gamma_n$. Let $x_\omega = f(y_\omega)$. Then 
  $x_\omega \in E$ and, moreover, for all $1 \leq n < \omega$, we have 
  \[
    f[N_{y_\omega \restriction \gamma_n}] = f[N_{y_n \restriction 
    \gamma_n}] \subseteq N_{x_n \restriction \alpha_{n+1}}.
  \]
  
  Therefore, for all $n < \omega$ we have $x_\omega \restriction \alpha_{n+1} = x_n \restriction 
  \alpha_{n+1}$. By a similar argument, using the fact that $y_\omega \restriction \beta_n 
  = y_n \restriction \beta_n$ for all $1 \leq n < \omega$, we know that $x_\omega \restriction 
  \beta_n = x_n \restriction \beta_n \neq x_{n-1} \restriction \beta_n$, and hence 
  $x_\omega \notin \{x_n \mid n < \omega\}$. Thus, $x_\omega$ is a valid play for 
  Player II in round $\omega$, completing our description of a winning strategy for Player II 
  in $G(E, x_0, \omega+1)$.
\end{proof}

For some observations about this theorem and some remaining open questions, see 
Section \ref{section: questions} below.

\section{Kurepa trees and the perfect set property} \label{section: weak_kurepa}

As mentioned in the introduction, $\kappa$-Kurepa trees often provide natural examples 
of closed subsets of ${^\kappa}\kappa$ that fail to have various perfect set properties. 
Indeed, suppose that $\kappa$ is a regular uncountable cardinal that is not 
strongly inaccessible and $T \subseteq {^\kappa}\kappa$ is a $\kappa$-Kurepa tree. Let 
$\lambda < \kappa$ be the least cardinal such that $2^\lambda \geq \kappa$. Then we claim 
that $[T]$ cannot contain a $(\lambda + 1)$-perfect subset. If it did, then, by Proposition 
\ref{prop: copy}, we could find an isomorphic embedding $\iota : {^{<\lambda + 1}}2 \ra T$.
Since $2^{<\lambda} < \kappa$, we can find $\alpha < \kappa$ such that 
$\iota[{^{<\lambda}}2] \subseteq T_{<\alpha}$. Then $\{\iota(x) \restriction \alpha \mid 
x \in {^{\lambda}}2\}$ is a subset of $T_\alpha$ of cardinality $2^\lambda \geq \kappa$, 
contradicting the fact that $T$ is a $\kappa$-tree.

In this section, we show that $\kappa$-Kurepa trees are not needed to produce closed subsets 
of ${^{\kappa}}\kappa$ failing various perfect set properties. For concreteness and readability, 
we will focus on the important special case $\kappa = \omega_1$, but the techniques can readily 
be generalized to produce similar results at other cardinals. 

Beginning in a model with an inaccessible cardinal, we will produce two models 
in which $\GCH$ holds, the Kurepa Hypothesis fails, and there is a 
closed subset of ${^{\omega_1}}\omega_1$ that fails to have the $(\omega+1)$-perfect set 
property.\footnote{Note that this is sharp, by Corollary \ref{cor: omega_psp}} In both models, 
the closed subset of ${^{\omega_1}}\omega_1$ failing to have the $(\omega_1)$-perfect set property 
will be of the form $[T]$ for some weak Kurepa tree $T \subseteq {^{<\omega_1}}_{\omega_1}$. 
The fact that $\mathsf{PSP}_{\omega+1}([T])$ fails will be implied by the fact that, in both case, 
$T$ will fail to contain a copy of ${^{<\omega + 1}}2$. In the first model, constructed in 
Theorem \ref{Th:negKH+wKH+A-subtree}, every weak Kurepa tree will contain an Aronszajn subtree. 
In the second, constructed in Theorem \ref{Th:negKH+noAsubtrees}, the weak Kurepa tree $T$ we 
satisfying $\neg \mathsf{PSP}_{\omega + 1}([T])$ will not contain any Aronszajn subtree. Note 
that either theorem individually will establish Theorem A from the introduction.

We first need the following preliminary result.

\begin{theorem}\label{TreesCohen}
Assume that $G$ is $\Q=\Add(\omega,1)$-generic over $V$. Suppose $T$ is an arbitrary normal tree in $V$. Then $T$ does not contain a copy of ${}^{<\omega+1}2^{V[G]}$ in $V[G]$.
\end{theorem}

\begin{proof}
We will proceed by contradiction and will assume that, in $V[G]$, there is an isomorphic embedding from ${}^{<\omega+1}2^{V[G]}$ into $T$. We first observe that it is suffices to consider embeddings which preserve meets and are continuous at limits:

\begin{claim}\label{cl:meets} 
Let $T$ be a normal tree, and let $f$ be an isomorphic embedding from  ${^{<\omega+1}}2$ to $T$.\footnote{For the purposes of the proof of Theorem \ref{TreesCohen}, think of this Claim as being applied in $V[G]$.} Then there is an isomorphic embedding $g$ which moreover preserves meets and is continuous at limits, i.e.:
\begin{enumerate}[(i)]
\item (preserves meets)\footnote{Note that if $g$ is an isomorphic embedding then $g$ preserves meets for all comparable nodes in  ${^{<\omega+1}}2$.} for all $a, b \in {^{<\omega+1}}2$, $g(a\wedge b)=g(a)\wedge g(b)$;
\item (continuous at limits) for all $x\in {^{\omega}}2$, $g(x)= \sup\set{g(x\rest n)}{n<\omega}$.
\end{enumerate}
\end{claim}

\begin{proof}
We first define $g$ which preserves meets on ${^{<\omega}}2$, and then argue that we can extend $g$ to the whole tree ${^{<\omega+1}}2$. For $a\in {^{<\omega}}2$, define $$g(a)= f(a^\smallfrown 0)\wedge f(a^\smallfrown 1).$$ First note that $f(a)\le g(a)$ for all $a\in {^{<\omega}}2$; from this it is easy to see that $g$ is $\perp$-preserving because $f$ is $\perp$-preserving. Moreover for all $a\subset b\in {^{<\omega}}2$ we have $g(a)< f(b)$ since $g(a)<f(b\rest (\dom(a)+1))\le f(b)$. Therefore for $a\sub b\in  {^{<\omega}}2$ we have $g(a)\le f(b)\le g(b)$, so $g$ preserves $\le$. To see that $g$ preserves meets, it is enough to verify that $g$ preserves meets for incomparable nodes in  ${^{<\omega}}2$, since we already verified that $g$ is an isomorphic embedding from ${^{<\omega}}2$ to $T$. Assume that  $a,b\in {^{<\omega}}2$ are incomparable. Then $g(a)\wedge g(b)=f(a)\wedge f(b)=g(a\wedge b)$: the first equality holds since $f(a)\le g(a)$ and $f(b)\le g(b)$ and the second holds by the definition of $g$ (observe that for all $a,a',b,b' \in T$, with $a,b$ incomparable, $a \le a', b \le b'$ implies $a \wedge b = a' \wedge b'$).

We extend the definition of $g$ to $x\in  {^{\omega}}2$ by letting $g(x)=\sup\set {g(x\rest n)}{n<\omega}$. This makes $g$ continuous at limits by definition provided we argue that the suprema exist: This is true because $g(x\rest n)\le f(x)$ for all $n<\omega$, and by normality of $T$, there is a unique node in $T$ on the supremum of the levels of $g(x\rest n)$, and this node is $\le f(x)$. Finally notice that if $g$ is an isomorphic embedding from ${^{<\omega+1}}2$ into $T$ which preserves meets on ${^{<\omega}}2$, then it also preserves meets for all $x,y\in {^{\omega}}2$, so $g$ is as required.
\end{proof}

\begin{remark}
Note that, if $T$ is a normal tree and $f$ is an isomorphic embedding from ${^{<\omega+1}}2$ to $T$ which is continuous at limits, then $f$ is determined by its restriction to ${^{<\omega}2}$.
\end{remark}

Let us work in $V[G]$. Let $T$ be a normal tree such that $T \in V$. Assume for a contradiction that  $f: ({}^{<\omega+1}2)^{V[G]} \to T$ is an isomorphic embedding; by Claim \ref{cl:meets}, we can assume that $f$ preserve meets and is continuous at limits. 

Let $p \in G$ force this property about some name $\dot{f}$ for $f$. Let us work in $V$ now. For each $x\in ({^{\omega}}2)^V$, there is a condition $q_x \le p$ in $\Q$ which decides the value of $\dot{f}(x)$. Since the forcing is countable, the Baire category theorem implies that there is some $q \in \Q$ such that $$Y =\set {x\in ({^{\omega}}2)^V}{q_x = q}$$ is not nowhere dense (we say it is \emph{somewhere dense}), i.e.\
$$\exists a^*\in {}^{<\omega}2 \; \forall a \in {}^{<\omega}2 \; (a^* \sub a \to \exists x \in Y \; x \rest |a| = a),$$ equivalently
\begin{equation} \label{eq:baire}\exists a^*\in {}^{<\omega}2 \; \forall a \in {}^{<\omega}2 \; (a^* \sub a \to \exists x \in ({}^\omega 2)^V\; x \rest |a| = a \mbox{ and } q \mbox{ decides } \dot{f}(x)).\end{equation}

We will now show that the fact that $\dot{f}$ is forced to be continuous allows us to prove that $$\mbox{$q$ decides $\dot{f}[N^v_{a^*}]$},$$ where $N^v_{a^*}$ denotes the set of all $x \in ({}^{\omega}2)^V$ which extend $a^*$.
 
Consider the subset $C=\set{x \wedge y}{ x,y \in Y, \ x \neq y}$ of ${^{<\omega}}2$. This set is in $V$, since $Y$ is in $V$. More importantly, $q$ decides the value of $\dot{f}(a)$ for all $a\in C$ since $q$ decides $\dot{f}(x)$ for all $x\in Y$ and $\dot{f}$ is forced to preserve meets. Since $Y$ satisfies (\ref{eq:baire}), it holds that $$a^*\uparrow=\set{a\in {^{<\omega}}2 }{a^*\sub a}$$ is a subset of $C$. Since $q$ decides $\dot{f}(a)$ for all $a\in C$, $q$ decides the values of $\dot{f}(a)$ for all $a\supseteq a^*$. Since $\dot{f}$ is forced to be continuous at limits, $q$ also decides the values of $\dot{f}(x)$ for all $x\in ({^{\omega}}2)^V$ such that $a^*\sub x$. 

Now we finish the proof by arguing that we can read off in $V$ the Cohen subset $c=\bigcup G$ added by $\Q$, which gives the desired contradiction. In $V[G]$, $c$ can be used to define a cofinal branch $c'$ through $a^*\uparrow$ as follows: Identify $c$ with a function from $\omega$ to $2$ and define $c'= (a^*)^\smallfrown c$; i.e.\ letting $n=\dom (a^*)$, $c'(k)= a^*(k)$ for $k<n$ and $c'(k)=c(k-n)$ for $k\ge n$. Recall that $f$ is an isomorphic embedding from $({^{<\omega+1}}2)^{V[G]}$ into $T$; therefore there is $t^*\in T$ such that $f(c')=t^*$. Since $t^*$ is in $T$, $t^*$ is in $V$. Define $d:\omega\to 2$ in $V$ inductively as follows:
\begin{enumerate}[(a)]
\item $d(0)=0$, provided $q\Vdash \dot{f}((a^*)^\smallfrown\langle n,0\rangle)\le t^*$; otherwise let $d(0)=1$. 
\item If $d\rest k$ is defined, set $d(k)=0$, provided $q\Vdash \dot{f}((a^*)^\smallfrown d\rest k^\smallfrown \langle k,0 \rangle )\le t^*$; otherwise let $d(k)=1$. 
\end{enumerate}
Now $q\Vdash\dot{f}((a^*)^\smallfrown d)\le t^*$ since $\dot{f}$ is forced to be an isomorphic embedding. We claim that $q\Vdash d=\dot{c}$: Since $q\Vdash\dot{f}((a^*)^\smallfrown d)\le t^*$ and $q\Vdash\dot{f}(\dot{c}')=t^*$ and $\dot{f}$ is forced to be an isomorphic embedding, $q\Vdash(a^*)^\smallfrown d=\dot{c}'$ and hence by the definition of $c'$, $q\Vdash d=\dot{c}$. This yields the desired contradiction.
\end{proof}

We now prove the first of the two main theorems of this section.

\begin{theorem}\label{Th:negKH+wKH+A-subtree}
  Suppose that there is an inaccessible cardinal $\kappa$. Then there is a forcing extension 
  in which 
  \begin{enumerate}
    \item $\kappa = \omega_2$;
    \item $\GCH$;
    \item $\neg \KH$;
    \item every weak Kurepa tree contains an Aronszajn subtree, moreover if we assume that $\Diamond$ holds in $V$, every weak Kurepa tree contains a Suslin subtree.
    \item there is a weak Kurepa tree that does not contain a copy of ${^{< \omega+1}}2$.
  \end{enumerate}
\end{theorem}

\begin{proof}

Assume that $\GCH$ holds. Let $\P = \Coll(\omega_1, {<}\kappa)$ and let $\Q=\Add(\omega,1)$. We claim that the generic extension by $\P\times \Q$ is the desired forcing extension in which $(1)$--$(5)$ hold. \textbf{Item (1)} and \textbf{item (2)} are clear, and \textbf{item (3)} follows by standard arguments for the tree property as in \cite{U:Fragility} or \cite{HS:TPbelow}.

\medskip

\textbf{The proof of item (4).} To prove item (4), it suffices to show that every weak Kurepa tree in a generic extension by $\P\times\Q$  contains a copy of $({^{<\omega_1}}2)^V$. To see this, let $G\times H$ be $\P\times\Q$-generic over $V$. In $V$, we can construct a special $\omega_1$-Aronszajn tree as a subtree of $({^{<\omega_1}}2)^V$ and this tree is preserved in all forcing extensions which preserve $\omega_1$. In particular, it is still a special $\omega_1$-Aronszajn tree in $V[G][H]$. Moreover, note that $\P\times\Q$ does not add cofinal branches to any $\omega_1$-tree in $V$ since $\P$ is $\omega_1$-closed in $V$ and $\Q$ is $\omega_1$-Knaster in $V[G]$. It follows that all (not only special) $\omega_1$-Aronszajn subtrees of $({^{<\omega_1}}2)^V$  remain Aronszajn  in $V[G][H]$.  In particular, if $S$ is an $\omega_1$-Suslin subtree of $({^{<\omega_1}}2)^V$ in $V$, it is still Aronszajn in $V[G][H]$. In fact it is still a Suslin tree in $V[G][H]$: by Easton's Lemma it is Suslin in $V[G]$ since $\P$ is $\omega_1$-closed and it is still Suslin in $V[G][H]$ since $\Q$ is $\omega_1$-Knaster and $S$ is ccc in $V[G]$.

Let us now proceed to show that every weak Kurepa tree in a generic extension by $\P\times\Q$  contains a copy of $({^{<\omega_1}}2)^V$. If $\dot{T}$ is a $\P\times\Q$-name for a weak Kurepa tree then, since $\P\times\Q$ is $\kappa$-cc, $\dot{T}$ is a $\P_\theta\times \Q$-name for some regular cardinal $\theta < \kappa$, where we denote $\P_\theta=\Coll(\omega_1,{<}\theta)$. Let $G_\theta$ be $\P_\theta$-generic over $V$ and $H$ be $\Q$-generic over $V[G_\theta]$. We will work in $V[G_\theta]$ and we will show that $T$ contains a copy of $({^{<\omega_1}}2)^{V[G_\theta]}$ in $V[G_\theta][H]$. Note that $({^{<\omega_1}}2)^V=({^{<\omega_1}}2)^{V[G_\theta]}$ since $\P_\theta$ is $\omega_1$-closed.

Note that, in $V[G_\theta][H]$, we have $2^{\omega_1} < \kappa$, and hence $T$ has fewer than 
$\kappa$-many cofinal branches in $V[G_\theta][H]$. Since $\dot{T}$ is a $\P\times\Q$-name for a weak Kurepa tree, it is forced to have $\kappa$-many branches in the extension by $\P\times\Q \cong \P_\theta*(\Q\times\P^\theta)$, where $\P^\theta=\Coll(\omega_1,[\theta,{<}\kappa))^{V[G_\theta]}$. Therefore, we can fix in $V[G_\theta]$ conditions $p^*\in \P^\theta$ and $q^* \in H$ and a $\Q\times\P^\theta$-name $\dot{b}$ for a cofinal branch through $\dot{T}$ such that \begin{equation} \label{eq:pforce} (q^*,p^*) \Vdash \dot{b} \not \in V[G_\theta][\dot{H}].\end{equation}

Without loss of generality, we can assume that the underlying set of $\dot{T}$ is forced to be a 
subset of $\omega_1 \times \omega_1$ and that if $(\alpha,\gamma)\in\dot{T}$, then $(\alpha,\gamma)\in \dot{T}_\gamma$. In $V[G_\theta]$, we will build by induction on $\omega_1$ the following objects:

\begin{itemize}
\item a labeled tree $\mathcal{T}=\set{p_s}{s\in {^{<\omega_1}}2 }$ of conditions in $\P^\theta$, all of them extending $p^*$ from (\ref{eq:pforce});
\item a labeled tree $\set{\gamma_s}{ s\in {^{<\omega_1}}2 }$ of ordinals below $\omega_1$;
\item a maximal antichain $A_s$ of conditions in $\Q$ below $q^*$ from (\ref{eq:pforce}) for each $s\in   {^{<\omega_1}}2$;
\end{itemize}

such that the following hold for each $s\in {^{<\omega_1}}2$:

\begin{enumerate}[(a)]

\item $p_t\le p_s$ for each $s\subseteq t$ in  ${^{<\omega_1}}2$;
\item $\gamma_s<\gamma_t$ for each $s\subset t$  in  ${^{<\omega_1}}2$;
\item for each $q\in A_s$, the conditions $(q,p_{s^{\smallfrown}0})$  and  $(q,p_{s^{\smallfrown}1})$ decide $\dot{b}$ up to $\gamma_{s}$ differently; i.e., there are $\gamma
\le\gamma_s $ and $\tau_{q,p_{s^{\smallfrown}0}} \neq \tau_{q,p_{s^{\smallfrown}1}}$ both 
forced by $q$ to be in $\dot{T}_\gamma$ such that $(q,p_{s^{\smallfrown}0})\Vdash \dot{b}(\gamma)= \tau_{q,p_{s^{\smallfrown}0}}$ and $(q,p_{s^{\smallfrown}1})\Vdash \dot{b}(\gamma)= \tau_{q,p_{s^{\smallfrown}1}}$.
\end{enumerate}

The construction of $\mathcal{T}$ uses the standard method of diagonalizing over antichains in $\Q$ while taking lower bounds in $\P^\theta$, using the $\omega_1$-closure of $\P^\theta$, but we will give details to make the argument self-contained.

Set $p_\emptyset=p^*$. First assume that $\alpha$ is a limit ordinal and for every $\beta<\alpha$ and every $s\in {^\beta}2$ the conditions $p_s$ have been constructed. For $s\in {^\alpha}2$ let $p_s$ be a lower bound of $\seq{p_{s\rest\beta}}{\beta<\alpha}$. Note that $A_s$ and $\gamma_s$ for $s\in {^\alpha}2$ will be constructed in the successor stage.

Now, assume that $\alpha$ is a successor ordinal $\alpha=\beta+1$ and for every  $s\in {^\beta}2$, $p_s$ has been constructed, and for every $s\in {^{<\beta}}2$, $A_s$ and $\gamma_s$ have been constructed. Given $s\in{^\beta}2$, we describe the construction of $p_{s^{\smallfrown}0}$, $p_{s^{\smallfrown}1}$, $A_s$ and $\gamma_s$.

\begin{claim}\label{Cl:diff}
For every $q \le q^*$ in $\Q$, all $r^0,r^1\in\P^\theta$ with $r^0,r^1 \le p^*$, and all 
$\gamma'<\omega_1$, there are $\gamma'<\gamma<\omega_1$, $(q',p^0)\le (q,r^0)$  and $(q',p^1)\le (q,r^1)$ such that $(q',p^0)$ and $(q',p^1)$ decide $\dot{b}(\gamma)$ differently.
\end{claim}

\begin{proof}
Let $q\in\Q$, $r^0, r^1\in \P^\theta$ and $\gamma'<\omega_1$ be given. Since $\dot{b}$ is forced by $(p^*,q^*)$ to be a cofinal branch through $\dot{T}$ that is not in $V[G_\theta][\dot{H}]$, there are $(\bar{q},\bar{p}^0)\le (q,r^0)$, $(\bar{q},\bar{p}^1)\le (q,r^0)$ and $\gamma>\gamma'$ such that $(\bar{q},\bar{p}^0)$ and $(\bar{q},\bar{p}^1)$ decide $\dot{b}(\gamma)$ differently; i.e.\ there are $\tau^0\neq\tau^1$, both forced by $\bar{q}$ to be in $\dot{T}_\gamma$, such that $(\bar{q},\bar{p}^0)\Vdash \dot{b}(\gamma)=\tau^0$ and $(\bar{q},\bar{p}^1)\Vdash \dot{b}(\gamma)=\tau^1$. Now, consider the condition $(\bar{q},r^1)$: since $\dot{b}$ is a $\P^\theta\times\Q$ name for a cofinal branch through $\dot{T}$, there is an extension $(q', p^1)\le (\bar{q},r^1)$ which decides $\dot{b}(\gamma)$. Since $\tau^0\neq \tau^1$, $(q', p^1)$ cannot decide $\dot{b}(\gamma)$ as being equal to both of them. Let $p^0$ be $\bar{p}^i$, for some $i<2$, such that $(q', p^1)$ and $(q',\bar{p}^i)$ disagree on $\dot{b}(\gamma)$. Then $q',p^0,p^1$ and $\gamma'$ are as required.
\end{proof}

We use the previous claim to inductively construct in $\delta$-many stages for some $\delta < \omega_1$ a maximal antichain $A_s=\set{q_i\in\Q}{i<\delta}$ below $q^*$, an increasing sequence of ordinals $\seq{\gamma_i<\omega_1}{i<\delta}$ whose supremum will be $\gamma_s$, and decreasing sequences $\seq{p^0_i\in\P^\theta}{i<\delta}$ and $\seq{p^1_i\in\P^\theta}{i<\delta}$ with lower bounds $p_{s^{\smallfrown}0}$ and $p_{s^{\smallfrown}1}$, respectively. 

Let us initialize the construction and define the required objects for $i = 0$. First set $\gamma'_s=\sup\set{\gamma_{s\rest\beta'}}{\beta'<\beta}$ (in case $\beta=0$, take $\gamma'_s=0$). By Claim \ref{Cl:diff} there are $\gamma'_s<\gamma_0<\omega_1$ and $(q_0,p^0_0), (q_0,p^1_0)\le (q^*,p_s)$ such that $(q_0,p^0_0)$ and $(q_0,p^1_0)$ decide $\dot{b}(\gamma_0)$ differently. The condition $q_0$, $p^0_0$, $p^1_0$ and the ordinal $\gamma_0$ are as required.

Now assume that $0 < i < \omega_1$ and for all $j<i$ we already have $q_j$, $p^0_j$, $p^1_j$ and $\gamma_j$. 

If there is $q\in \Q$ below $q^*$ such that $q$ is incompatible with all $q_j$ for $j<i$, let us fix such $q$. If $i$ is a limit ordinal, fix some lower bounds $r^0$ and $r^1$ of $\seq{p^0_j\in\P^\theta}{j<i}$ and
 $\seq{p^1_j\in\P^\theta}{j<i}$, respectively, and a supremum $\gamma'$ of $\seq{\gamma_j<\omega_1}{j<i}$. If $i$ is a successor of $j$, set $r^0=p^0_j$, $r^1=p^1_j$ and $\gamma'=\gamma_j$. By Claim \ref{Cl:diff}, there are $\gamma'<\gamma_i<\omega_1$ and $(q_i,p^0_i)\le (q,r^0)$, $(q_i,p^1_i)\le (q,r^1)$ such that $(q_i,p^0_i)$ and $(q_i,p^1_i)$ decide $\dot{b}(\gamma_i)$ differently. The conditions $q_i$, $p^0_i$, $p^1_i$ and the ordinal $\gamma_i$ are as required.

If there is no $q\in \Q$ below $q^*$ such that $q$ is incompatible with $q_j$ for all $j<i$, we stop the construction and set $\delta=i$ and $A_s=\set{q_j\in\Q}{j<\delta}$. If $i$ is a limit ordinal, we set $p_{s^{\smallfrown}0}$ and $p_{s^{\smallfrown}1}$ to be lower bounds of $\seq{p^0_j\in\P^\theta}{j<\delta}$ and  $\seq{p^1_j\in\P^\theta}{j<\delta}$, respectively, and $\gamma_s$ to be a supremum of $\seq{\gamma_j<\omega_1}{j<\delta}$. If $i$ is a successor of $j$, then we set $p_{s^{\smallfrown}0}=p_j^0$, $p_{s^{\smallfrown}1}=p_j^1$ and $\gamma_s=\gamma_j$. Note that the construction will end after countably many steps since $\Q$ is ccc, and hence $\delta<\omega_1$.

This ends the construction of the labeled tree $\mathcal{T}$ and the related objects.

In $V[G_\theta][H]$, we define a copy of  $({^{<\omega_1}}2)^{V[G_\theta]}$ in $T = \dot{T}^{V[G_\theta][H]}$ using the tree $\mathcal{T}$. This copy is given by an embedding $h: ({^{<\omega_1}}2)^{V[G_\theta]}\to T$ which maps sequences $s^\smallfrown 0$ and $s^\smallfrown 1$ to the nodes $\tau_{q,p_{s^{\smallfrown}0}}$ and $\tau_{q,p_{s^{\smallfrown}1}}$, respectively where $q$ is the unique element of $H \cap A_s$ (see item (c) in the properties of $\mathcal{T}$ for definitions). The definition of $h$ extends continuously the limit levels: if $s \in ({}^\gamma 2)^{V[G_\theta]}$ for a limit ordinal $\gamma < \omega_1$, then let $h(s)$ be the supremum of $\{h(s \restriction \alpha) \mid \alpha < \gamma\}$. This supremum exists in $T$ by normality and the fact that, for instance, 
$h(s ^\frown 0)$ is above $h(s \restriction \alpha)$ for all $\alpha < \delta$.
Since the $q$'s are chosen from $H$, the forcing statements from item (c)
$$(q,p_{s^{\smallfrown}0})\Vdash \dot{b}(\gamma)= \tau_{q,p_{s^{\smallfrown}0}} \mbox{ and }(q,p_{s^{\smallfrown}1})\Vdash \dot{b}(\gamma)= \tau_{q,p_{s^{\smallfrown}1}}$$ respect the tree $T$, i.e.\ the embedding $h$ preserves the strict ordering and the incompatibility of nodes between the trees $({}^{<\omega_1}2, \sub)^{V[G_\theta]}$ and $(T,<_T)$. It follows that $T$ contains a copy of $({}^{<\omega_1}2)^{V[G_\theta]} = ({}^{<\omega_1}2)^V$ as required.

\medskip
 
\textbf{The proof of item (5).} Item (5) is a consequence of Theorem \ref{TreesCohen}: In $V[G]$ there are normal weak Kurepa trees, e.g., $({^{<\omega_1}}2)^V$. Such a tree remains a weak Kurepa tree in 
$V[G][H]$ and, by Theorem \ref{TreesCohen}, cannot contain a copy of $({}^{<\omega+1}2)^{V[G][H]}$.
\end{proof}

We now turn to proving the second main theorem of this section, which will produce a model similar to that of Theorem \ref{Th:negKH+wKH+A-subtree}, except we will obtain a weak Kurepa tree which does not contain a copy of  ${}^{<\omega+1}2$ and does not contain an Aronszajn subtree. First we define a forcing which adds a weak Kurepa tree with these properties and establish some basic facts about the forcing.

\begin{definition}  
Let $\lambda$ be an uncountable cardinal. We define a poset $\K_\lambda$ which adds a tree with size and height $\omega_1$ with $\lambda$-many cofinal branches. Conditions $q \in \K_\lambda$ are pairs $(T_q, f_q)$ such that
  \begin{itemize}
    \item there is $\eta_q < \omega_1$ such that $T_q$ is a normal, infinitely splitting subtree of 
    ${^{<\eta_q + 1}}\omega_1$ 
    that does not contain a copy of ${^{< \omega + 1}}2$;
    \item $f_q$ is a countable partial function from $\lambda$ to $T_q \cap {^{\eta_q}}\omega_1$.
  \end{itemize}
  If $q_0, q_1 \in \Q$, then $q_1 \leq q_0$ if and only if
  \begin{itemize}
    \item $\eta_{q_1} \geq \eta_{q_0}$;
    \item $T_{q_1} \cap {^{<\eta_{q_0} + 1}}\omega_1 = T_{q_0}$;
    \item $\dom(f_{q_1}) \supseteq \dom(f_{q_0})$;
    \item for all $\alpha \in \dom(f_{q_0})$, $f_{q_1}(\alpha) \supseteq f_{q_0}(\alpha)$.
  \end{itemize}
We also include the pair $(\emptyset, \emptyset)$ in $\K_\lambda$ as $1_{\K_\lambda}$.  
\end{definition} 

Note that $\K_\lambda$ is not $\omega_2$-cc: in fact, it collapses $2^{\omega_1}$ to $\omega_1$ since we can code subsets of $\omega_1$ in the ground model into the levels of the generic tree added by $\K_\lambda$. Therefore if $\lambda\le 2^{\omega_1}$, then the generic tree added by $\K_\lambda$ is not a weak Kurepa tree.

\begin{lemma}\label{L:K-Knaster}
Let $\lambda$ and $\mu$ be uncountable cardinals such that $\mu>2^{\omega_1}$ is regular and $\mu>\gamma^\omega$ for all $\gamma<\mu$. Then $\K_\lambda$ is $\mu$-Knaster. In particular $\K_\lambda$ is $(2^{\omega_1})^+$-Knaster.
\end{lemma}

\begin{proof}
Let a set of conditions $\set{q_\alpha=(T_\alpha,f_\alpha)\in\K_\lambda}{\alpha<\mu}$ be given. Since $\mu>2^{\omega_1}$ is regular and there are only $2^{\omega_1}$-many possibilities for $T_\alpha$'s, there is a tree $T\sub {^{<{\eta+1}}\omega_1}$ of countable height and $I\sub\mu$ of size $\mu$ such that $T_\alpha=T$ for all $\alpha\in I$.

Since $\gamma^\omega<\mu$ for all $\gamma<\mu$, there is $I'\sub I$ of size $\mu$ such that the set $\set{\dom(f_\alpha)}{\alpha\in I'}$ forms a $\Delta$-system with root $a\sub\lambda$. Since $a$ is at most countable, there are at most $2^\omega<\mu$ many functions from $a$ to $T_\eta$ and therefore there is a countable $f$ from $a$ to $T_\eta$ and $J\sub I$ of size $\mu$ such $f=f_\alpha\cap f_\beta$ for all $\alpha\neq\beta\in J$. Then all conditions in $\set{q_\alpha}{\alpha\in J}$ are compatible.
\end{proof}

\begin{lemma}\label{L:K-closed}
Let $\lambda$ be an uncountable cardinal. If $\CH$ holds, then $\K_\lambda$ is $\omega_1$-closed.
\end{lemma}
  
\begin{proof}
    Let $\langle q_n \mid n < \omega \rangle$ be a decreasing sequence from $\K_\lambda$. To avoid 
    trivialities, assume that the sequence $\langle \eta_{q_n} \mid n < \omega \rangle$ is strictly 
    increasing. Let $\eta := \sup\{\eta_{q_n} \mid n < \omega\}$. We will construct a lower bound 
    $q$ for $\langle q_n \mid n < \omega \rangle$ such that $\eta_q = \eta$. Let $T = 
    \bigcup_{n < \omega} T_{q_n}$. To define $T_q$, we simply need to decide which cofinal 
    branches through $T$ should continue. 
    
    We first note that there may be countably many branches that we are obliged to extend 
    because of the functions $\{f_{q_n} \mid n < \omega\}$. Namely, let $a = \bigcup_{n < \omega} 
    \dom(f_{q_n})$ and, for each $\alpha \in a$, let 
    \[
      b_\alpha := \bigcup \{f_{q_n}(\alpha) \mid n < \omega \wedge \alpha \in \dom(f_{q_n})\}.
    \]
    Then each $b_\alpha$ is a cofinal branch through $T$, and we are obliged to put $b_\alpha$ 
    in $T_q$.
    
    We may need to extend additional branches through $T$ in order to ensure that $T_q$ is 
    normal; i.e., for each $\sigma \in T$, we need to ensure that there is $\tau \in T_q \cap 
    {^{\eta}}\omega_1$ with $\sigma \subseteq \tau$. However, when doing so, we must be careful 
    not to add a copy of ${^{<\omega + 1}}2$ to $T_q$. We will do so through the use of the following 
    bookkeeping device.
    
    Let $\langle \iota_\xi \mid \xi < \omega_1 \rangle$ enumerate all isomorphic embeddings 
    $\iota$ from ${^{<\omega}}2$ to $T$ such that, for every 
    $x \in {^{\omega}}2$, the union $\bigcup \{\iota(x \restriction n) \mid n < \omega\}$ 
    is a cofinal branch through $T$. Note that this is possible, due to the fact that 
    $\CH$ holds and $|T| \leq \omega_1$. For each $\xi < \omega_1$, let 
    \[
      [\iota_\xi] := \left\{\bigcup \{\iota_\xi(x \restriction n) \mid n < \omega\} \ \middle| \ x \in 
      {^{\omega}}2\right\}.
    \]
    Note that $[\iota_\xi] \subseteq [T]$, and $|[\iota_\xi]| = 2^\omega = \omega_1$. Let us enumerate 
    $T$ as $\langle \sigma_\xi \mid \xi < \omega_1 \rangle$, with repetitions if $T$ is countable.
    We now recursively construct disjoint subsets $\{c_\xi \mid \xi < \omega_1\}$ and 
    $\{d_\xi \mid \xi < \omega_1\}$ of $[T]$. Suppose that $\xi < \omega_1$ and we have 
    constructed $c_\zeta$ and $d_\zeta$ for all $\zeta < \xi$. First, choose $c_\xi \in [T]$ such 
    that
    \begin{itemize}
      \item $\sigma_\xi \subseteq c_\xi$;
      \item $c_\xi \notin \{d_\zeta \mid \zeta < \xi\}$.
    \end{itemize}
    Note that this is possible to do: $T$ is normal and splitting, and $\cf(\gamma) = \omega$, so 
    there are $2^\omega$ many elements of $[T]$ extending $\sigma_\xi$. Next, choose 
    $d_\xi \in [T]$ such that
    \begin{itemize}
      \item $d_\xi \in [\iota_\xi]$;
      \item $d_\xi \notin \{b_\alpha \mid \alpha \in a\} \cup \{c_\zeta \mid \zeta \leq \xi\}$.
    \end{itemize}
    At the end of the construction, set $T_q := T \cup \{b_\alpha \mid \alpha \in a\} \cup 
    \{c_\xi \mid \xi < \omega_1\}$. Let $f_q$ be such that $\dom(f_q) = a$ and, for all 
    $\alpha \in a$, $f_q(a) = b_\alpha$. 
    
We claim that $q$ is a condition in $\bb{\K_\lambda}$ and it is a lower bound of $\langle q_n \mid n < \omega \rangle$. The only nontrivial thing to verify is the fact that $T_q$ does not contain a copy of ${^{<\omega+1}}2$. Assume for a contradiction that $T_q$ contains a copy of ${^{<\omega+1}}2$. Since, for every $n < \omega$, we know that $T_{q_n}$ does not contain a copy of ${^{<\omega+1}}2$, there must be a $\xi<\omega_1$ such that 
$[\iota_\xi]\sub (T_q)_\eta=\set{b_\alpha}{\alpha \in a } \cup \set{c_\zeta}{\zeta < \omega_1}$. However, this means that $d_\xi \in\set{b_\alpha}{\alpha \in a } \cup \set{c_\zeta}{\zeta < \omega_1}$ which contradicts our choice of $d_\xi$.
\end{proof}

\begin{lemma}\label{L:wKT-projection}
Let $\lambda<\kappa$ be uncountable cardinals. Then there is a projection from $\K_\kappa$ to $\K_\lambda$.
\end{lemma}

\begin{proof}
We define $\pi$ from $\K_\kappa$ to $\K_\lambda$ by letting $\pi(T_q,f_q)=(T_q,f_q\rest\lambda)$ for all $q\in\K_\kappa$. It is routine to verify that $\pi$ is order-preserving and that $\pi(1_{\K_\kappa})=1_{\K_\lambda}$.

Let $q\in\K_\kappa$ and $r\in\K_\lambda$ be such that $r\le \pi(q)=(T_q,f_q\rest\lambda)$. Define $r'\le q$ by first letting $T_{r'}= T_r$ (note that $T_r$ is an end-extension of $T_q$). Let $\dom (f_{r'})=\dom(f_{q})\cup\dom(f_r)$ and define 
\begin{itemize}[--]
\item $f_{r'}(\alpha)=f_r(\alpha)$ for every $\alpha\in\dom(f_r)$ and
\item  $f_{r'}(\alpha)=\tau$, where $\tau\in (T_{r'})_{\eta_{r'}}$ with $\tau\supseteq f_q(\alpha)$, for every $\alpha\in \dom(f_q)\setminus\dom(f_r)$. 
\end{itemize}
It is easy to check that $\pi(r')=r$. Therefore $\pi$ is a projection.
\end{proof}

Let $H$ be a $\K_\lambda$-generic filter, $T=\bigcup\set{T_q}{q\in H}$, and let $\K_{\kappa}/H=\set{r\in\K_\kappa}{\pi(r)\in H}$ be the quotient given by $H$ and the projection $\pi$. Then $\K_\kappa$ is forcing equivalent to a two step iteration $\K_\lambda* \K_\kappa/H$. It is easy to see that in $V[H]$, $\K_\kappa/H$ is forcing equivalent to the forcing notion $\K_{\lambda,\kappa}$, where conditions in $\K_{\lambda,\kappa}$ are pairs $r=(\eta_r,f_r)$ such that $f_r$ is a countable partial function from $\kappa\setminus\lambda$ to $T_{\eta_r}$ and $r\le q$ if and only if $\eta_r\ge\eta_q$, $\dom(f_q)\sub\dom(f_r)$, and $f_q(\alpha)\sub f_r(\alpha)$ for all $\alpha\in\dom(f_q)$.

\begin{lemma}
Let $\lambda<\kappa$ be uncountable cardinals and $H$ be a $\K_\lambda$-generic filter. Then $\K_{\kappa}/H$ is $\omega_1$-distributive in $V[H]$.
\end{lemma}

\begin{proof}
The forcing $\K_\kappa$ is forcing equivalent to a two step iteration $\K_\lambda*\K_{\kappa}/\dot{H}$. Since $\K_\kappa$ is $\omega_1$-closed by Lemma \ref{L:K-closed}, $\K_{\kappa}/H$ cannot add new countable sequences of ordinals over $V[H]$, hence it is $\omega_1$-distributive in $V[H]$.
\end{proof}

We fix the following notation: for $r\in \K_\kappa$ and $\lambda<\kappa$, let $r\rest\lambda$ denote $(T_r,f_r\rest\lambda)$, i.e.\  $r\rest\lambda=\pi(r)$ for the projection $\pi$ defined in Lemma \ref{L:wKT-projection}.

\begin{lemma}\label{L:into}
Let $\lambda<\kappa$ be uncountable cardinals and $\dot{H}$ be a canonical $\K_\lambda$-name for the $\K_\lambda$-generic filter. Let $q\in \K_\lambda$ and $r\in\K_\kappa$. Then the following are equivalent.
\begin{enumerate}[(i)]
\item  $q\Vdash r\in \K_\kappa/\dot{H}$;
\item $q\le r\rest\lambda$.
\end{enumerate}
\end{lemma}

\begin{proof}
To see that (i) implies (ii), note that if $q\not\le r\rest\lambda$, then by the separativity of $\K_\lambda$ there is $q'\le q$ which is incompatible with $r\rest\lambda$, and hence $q$ does not force $r$ into $\K_\kappa/\dot{H}$. The other direction is clear, since $q\le r\rest\lambda$ means that $q\le\pi(r)$.
\end{proof}

\begin{lemma}\label{C:wKurepaTree}
Assume $\GCH$ and let $\lambda > \omega_2$ be a cardinal. Let $H$ be a $\K_\lambda$-generic filter over $V$. Then the generic tree $T=\bigcup\set{T_q}{q\in H}$ is a weak Kurepa tree with $\lambda$-many cofinal branches which does not contain a copy of ${^{<\omega+1}}2$.
\end{lemma}

\begin{proof}
Since  $\K_\lambda$ is $\omega_1$-closed, $\omega_1$ is preserved by $\K_\lambda$ and the generic tree $T=\bigcup\set{T_q}{q\in H}$ is thus a tree with height and size $\omega_1$. By a standard density argument $T$ has $\lambda$-many cofinal branches in $V[H]$. Since $\GCH$ holds in the ground model, $\K_\lambda$ is $\omega_3$-Knaster by Lemma \ref{L:K-Knaster}, hence all cardinals greater than $\omega_2$ are preserved (recall that $2^{\omega_1}$ is always collapsed); in particular $\lambda$ is preserved. Since $\lambda>\omega_2$ in the ground model, $\lambda>\omega_1$ in $V[H]$; therefore the generic tree $T$ is a weak Kurepa tree in $V[H]$.

By Claim \ref{cl:meets}, if there is a copy of ${^{<\omega+1}2}$ in $T$, there is one which is bounded in the height of $T$ and therefore there is a condition $q\in H$ such that $T_q$ contains a copy of ${^{<\omega+1}2}$; this would contradict the definition of the forcing $\K_\lambda$. 
\end{proof}

We show now that the generic tree added by $\K_\lambda$ does not contain an Aronszajn subtree and that this property is preserved by all $\omega_1$-closed forcings.  

\begin{lemma}\label{L:NoA}
Assume $\CH$.
Let $\lambda> \omega_1$ be a cardinal and let $H$ be a $\K_\lambda$-generic filter over $V$. Assume that $\P$ is an $\omega_1$-closed forcing in $V[H]$. Then the generic tree $T=\bigcup\set{T_q}{q\in H}$ does not contain an Aronszajn subtree in the generic extension of $V[H]$ by $\P$. In particular $T$ does not contain an Aronszajn subtree in $V[H]$.
\end{lemma}

\begin{proof}
We work in $V$.
Let $\dot{T}$ denote a canonical $\K_\lambda$-name for the generic tree added by $\K_\lambda$. Assume for a contradiction that $\dot{S}$ is a $\K_\lambda*\dot{\P}$-name for an Aronszajn subtree of $\dot{T}$ and $(q,\dot p)\in\K_\lambda*\dot{\P}$ is a condition which forces this. 

We begin with some easy observations. For every $r\in \K_\lambda$ and all $\sigma, \tau \in 
{^{<\omega_1}}\omega_1$, it holds that $r\Vdash \sigma\le_{\dot{T}} \tau$ if and only if $\sigma\sub\tau\in T_r$. Moreover, if  $(r,\dot{p}')\le (q,\dot{p})$ decides that level $\dot{S}_\gamma=x$ for some $\gamma<\omega_1$, then $\height(T_r)\ge\gamma+1$ and $x\sub (T_r)_\gamma$.

\begin{claim}\label{C:indstep}
Let $(r,\dot{p}')\le (q,\dot{p})$ decide that level $\dot{S}_\gamma=x$ for some $\gamma<\omega_1$. Then there is a condition $(r^*,\dot{p}^*)=((T_{r^*},f_{r^*}),\dot{p}^*)$ stronger than $(r,\dot{p}')$ such that the following holds:
\begin{enumerate}[(i)]
\item for each $\sigma\in x$ there is $\alpha\in\dom(f_{r^*})$ such that $f_{r^*}(\alpha)\supseteq \sigma$;\label{C:1}
\item for each $\alpha\in\dom(f_{r^*})$, $(r^*,\dot{p}^*)\Vdash f_{r^*}(\alpha)\not\in\dot{S}$.\label{C:2}
\end{enumerate}
\end{claim}
\begin{proof}
Note that we can assume that $(r,\dot{p}')$ is such that for each $\sigma\in x$ there is $\alpha\in\dom(f_r)$ with $f_r(\alpha)\supseteq\sigma$; in particular $(r,\dot{p}')$ forces that $\sigma$ is in some cofinal branch of $T$. If this is not the case, then we can extend $f_r$ appropriately using that $T_r$ is normal. We now build by induction on $\omega$ a decreasing sequence of conditions $\seq{(r_n,\dot{p}_n)\in\K_\lambda*\dot{\P}}{n<\omega}$ such that the desired $(r^*,\dot{p}^*)$ will be a lower bound of this sequence.)

Begin by letting $(r_0, \dot{p}_0) = (r,\dot{p}')$. Now fix $n < \omega$ and suppose that we have 
constructed $(r_n, \dot{p}_n)$. Note that, for each $\alpha \in \dom(f_{r_n})$, 
$\bigcup \{f_s \mid s \in \dot{H}\}$ is forced to be a cofinal branch through $\dot{T}$, where 
$\dot{H}$ is a name for the $\K_\lambda$-generic filter. Therefore,
since $\dot{S}$ is forced by $(r_n,\dot{p}_n)$ to be an Aronszajn subtree of $\dot{T}$, 
and since $\dom(f_{r_n})$ is countable, there is $(r_{n+1},\dot{p}_{n+1})\le (r_n,\dot{p}_n)$ 
such that, for all $\alpha \in \dom(f_{r_n})$, we have $(r_{n+1},\dot{p}_{n+1})\Vdash f_{r_{n+1}}(\alpha)\not\in \dot{S}$. Let $(q^*,\dot{p}^*)$ be a lower bound of $\seq{(r_n*\dot{p}_n)\in\K_\lambda*\dot{\P}}{n<\omega}$ such that $\dom(f_{q^*}) = \bigcup\{\dom(f_{r_n}) \mid n < \omega\}$. Then it is easy to see that $(q^*,\dot{p}^*)$ satisfy condition (\ref{C:1})  and (\ref{C:2}) above.
\end{proof}

Now, we build a decreasing sequence of conditions $\seq{(r^*_n,\dot{p}^*_n)\in \K_\lambda*\dot{\P}}{n<\omega}$ such that each condition will satisfy an instance of Claim \ref{C:indstep}. Let $(r_0,\dot{p}_0)\le (q,\dot{p})$ decide that level $\dot{S}_{\gamma_0}=x_{\gamma_0}$ for some $\gamma_0<\omega_1$. Then take $(r^*_0,\dot{p}^*_0)\le (r_0,\dot{p}_0)$ to be a condition which satisfies Claim \ref{C:indstep} for $\gamma_0$. Note that $\height(T_{r^*_0})\ge\gamma_0$. Let $(r^*_n,\dot{p}^*_n)$ be constructed for some $n<\omega$. Let $(r_{n+1},\dot{p}_{n+1})$ be a condition which decides $\dot{S}_{\gamma_{n+1}}=x_{\gamma_{n+1}}$ for some $\gamma_{n+1}>\height(T_{r^*_n})$, and then take $(r^*_{n+1},\dot{p}^*_{n+1}) 
\leq (r_{n+1}, \dot{p}_{n+1})$ to be a condition which satisfies  Claim \ref{C:indstep} for $\gamma_{n+1}$. Note that $\height(T_{r^*_{n+1}})\ge\gamma_{n+1}$. 

It now follows from the construction that $T^\gamma=\bigcup_{n<\omega} T_{r^*_{n}}$ is a countable subtree of ${^{<\gamma}}\omega_1$ of limit height $\gamma$, where $\gamma = \sup\set{\gamma_n}{n<\omega}$. Let 
\[
  S^\gamma=\bigcup\set{\sigma\in {^{<\gamma}}\omega_1 }{\exists n<\omega \ (r^*_{n},\dot{p}^*_n)\Vdash \sigma\in\dot{S}}. 
\]
Note that $S^\gamma$ is a subtree of $T^\gamma$ with countable levels and height $\gamma$ in $V$, and 
that any lower bound for $\langle (r^*_n, \dot{p}^*_n) \mid n < \omega \rangle$ will force that 
$\dot{S} \cap {^{<\gamma}}\omega_1 = S^\gamma$.

Let $a=\bigcup_{n<\omega} \dom(f_{r^*_n})$, then for each $\alpha\in a$ 
      $$b_\alpha = \bigcup \set{f_{r^*_n}(\alpha) }{ n < \omega \wedge \alpha \in \dom(f_{r^*_n})}.$$
is a cofinal branch through $T^\gamma$. However, $b_\alpha$ is not a cofinal branch through $S^\gamma$ since there is $n<\omega$ such that $(r^*_n,\dot{p}^*_n)\Vdash f_{r^*_n}(\alpha)\not\in\dot{S}$. Hence we are not obligated to put any cofinal branch through $S^\gamma$ into lower bound of $\seq{r^*_n\in \K_\lambda}{n<\omega}$.

We claim that we can find a lower bound $r^*$ of $\seq{r^*_n}{n<\omega}$ such that $T_{r^*}$ has height $\gamma+1$ and for each $\alpha\in\dom(f_{r^*})$, $f_{r^*}(\alpha)$ is not a cofinal branch through $S^\gamma$. Consider $r=(T_r, f_r)$, where $T_r=T^\gamma\cup \set{b_\alpha}{\alpha\in a}$, $\dom(f_r)=a$ and for each $\alpha\in a$, $f_r(\alpha)=b_\alpha$. Each lower bound of $\seq{r^*_n\in \K_\lambda}{n<\omega}$ has to extend $r$, but $r$ itself does not have to be a condition in $\K_\lambda$, since $T_r$ does not have to be normal. However for each $\sigma\in S^\gamma$ there is $\alpha\in a$ such that $\sigma\sub b_\alpha$, hence to extend $T_r$ to a normal tree we only need to extend nodes in $T^\gamma\setminus S^\gamma$. As in the proof that $\K_\lambda$ is $\omega_1$-closed, extend $T_r$ to a normal tree $T^* = T \cup \{c_\tau \mid \tau \in T^\gamma \setminus S^\gamma\}$ by carefully picking cofinal branches $c_\tau$ in $T^\gamma$ for each $\tau\in T^\gamma\setminus S^\gamma$ to ensure that $T^*$ does not contain a copy of ${^{<\omega+1}2}$. Since we add only nodes to level $\gamma$ above $T^\gamma\setminus S^\gamma$, we did not add any cofinal branch through $S^\gamma$ into $T^*$. Therefore every cofinal branch $b$ through $S^\gamma$ is vanishing in $T^*$; i.e.\ there is no node on level $\gamma$ above $b$. 

Let $r^*=(T^*,f_r)$. Then $r^*$ is a condition in $\K_\lambda$ which is a lower bound of $\seq{r^*_n\in \K_\lambda}{n<\omega}$. Consider $(r^*,\dot{p}^*)$, where $\dot{p}^*$ is chosen so that $(r^*,\dot{p}^*)$ is a lower bound of $\seq{(r^*_n,\dot{p}^*_n)\in \K_\lambda*\dot{\P}}{n<\omega}$. Then the condition $(r^*,\dot{p}^*)$ forces that $\dot{S}\cap \dot{T}_
\gamma=\emptyset$, which is a contradiction since $(q,\dot{p})$ forces that $\dot{S}$ is an Aronszajn subtree of $\dot{T}$ and in particular unbounded in $\dot{T}$. 
\end{proof}

We are now ready to prove the following variation on Theorem \ref{Th:negKH+wKH+A-subtree}.

\begin{theorem}\label{Th:negKH+noAsubtrees}
  Suppose that there is an inaccessible cardinal $\kappa$. Then there is a forcing extension 
  in which 
  \begin{enumerate}
    \item $\kappa = \omega_2$;
    \item $\GCH$;
    \item $\neg \KH$;
    \item there is a weak Kurepa tree $T \subseteq {^{<\omega_1}}\omega_1$ that does not contain 
    a copy of ${^{< \omega+1}}2$ and does not contain an Aronszajn subtree.
  \end{enumerate}
\end{theorem}

\begin{proof}
Let $\kappa$ be an inaccessible cardinal, and assume that $\GCH$ holds.
 Let $\P = \Coll(\omega_1,<\kappa)$ and let $G\times H$ be $\P\times\K_\kappa$-generic over $V$. $\P\times\K_\kappa$ is $\omega_1$-closed and $\kappa$-Knaster and hence preserves $\omega_1$ and all cardinals greater or equal $\kappa$. Since $\P$ collapses cardinals between $\omega_1$ and $\kappa$ to $\omega_1$, $\kappa$ is $\omega_2$ in $V[G][H]$. Since $\P \times\K_\kappa$ is $\omega_1$-closed and has cardinality $\kappa$, it follows that $\GCH$ holds in $V[G][H]$. By Lemma \ref{C:wKurepaTree}, $T=\bigcup\set{T_q}{q\in H}$ is a weak Kurepa tree with $\kappa$-many cofinal branches which does not contain a copy of ${^{<\omega+1}}2$ in $V[H]$. Note that $T$ is still a weak Kurepa tree in $V[H][G]$ since $\P$ preserves $\omega_1$ and $\kappa$ over $V[H]$. Moreover, since $\P$ is $\omega_1$-closed in $V[H]$, $T$ does not contain any Aronszajn subtree in $V[H][G]=V[G][H]$ by Lemma \ref{L:NoA}. To see that $T$ does not contain a copy of ${^{<\omega+1}2}$ in $V[H][G]$, we will prove the stronger assertion that no $\omega_1$-distributive forcing can add a copy of ${^{<\omega+1}2}$ to a tree of height $\omega_1$ which does not contain such a copy in the ground model.

\begin{claim}
Let $T$ be a normal tree of height $\omega_1$ which does not contain a copy of ${^{<\omega+1}2}$ and let $P$ be an $\omega_1$-distributive forcing. Then $P$ does not add a copy of ${^{<\omega+1}2}$ to $T$.
\end{claim}

\begin{proof}
This follows from Claim \ref{cl:meets} which implies that we can assume that a copy of ${^{<\omega+1}2}$ is determined by its restriction to ${^{<\omega}2}$. Let us give some details. Let $F$ be $P$-generic over $V$ and let $f$ be an isomorphic embedding from ${^{<\omega+1}2}$ to $T$ in $V[F]$. Then since ${^{<\omega}2}$ has size $\omega$ and $P$ is $\omega_1$-distributive, $f\rest {^{<\omega}2}$ is in the ground model and we can extend it to an isomorphic embedding $g$ from ${^{<\omega+1}2}$ by defining $g(x)=\sup\set{x\rest n}{n<\omega}$ for $x\in {^\omega2}$ and $g(a)=f(a)$ for $a\in {^{<\omega}2}$. Note that $g(x)$ is well defined for all $x\in {^\omega2}$ since $f(x)$ is above $\set{x\rest n}{n<\omega}$ in $V[F]$ and $T$ is normal. The mapping $g$ determines a copy of ${}^{<\omega+1}2$ in the ground model, which is a contradiction.
\end{proof}

Since $\P$ is $\omega_1$-closed in $V[H]$, it follows by the previous claim that $T$ does not contain a copy of ${^{<\omega+1}2}$ in $V[H][G]$ because it has no such copy in $V[H]$.

To finish the proof we need to show that there are no Kurepa trees in $V[G][H]$. Assume that $S$ is an $\omega_1$-tree in $V[G][H]$. Since $S$ has size $\omega_1$ and $\P\times\K_\kappa$ is $\kappa$-Knaster, there is a nice $\P\times\K_\kappa$-name $\dot{S}$ of size less than $\kappa$ for $S$. Since $\dot{S}$ has size less than $\kappa$,  $\dot{S}$ is a $\P_\theta\times \K_\theta$-name for some regular cardinal $\theta < \kappa$, where $\P_\theta=\Coll(\omega_1,{<}\theta)$. Let $G_\theta$ denote the $\P_\theta$-generic over $V$ determined by $G$; i.e.\ $G_\theta=\set{p\rest\theta}{p\in G}$, and let $H_\theta$ denote the $\K_\theta$-generic filter over $V[G_\theta]$ determined by $H$ and $\pi$, where $\pi$ is the projection from $\K_\kappa$ to $\K_\theta$ from Lemma \ref{L:wKT-projection}. The $\omega_1$-tree $S$ is an element of $V[G_\theta][H_\theta]$ and has at most $(2^{\omega_1})^{V[G_\theta][H_\theta]}$-many cofinal branches here, which is less than $\kappa$, since $\kappa$ is still inaccessible in $V[G_\theta][H_\theta]$. We show that the quotient forcing $\P^\theta\times \K_\kappa/H_\theta$, where $\P^\theta$ denotes $\Coll(\omega_1,[\theta,{<}\kappa))$, cannot add a cofinal branch to $S$ over $V[G_\theta][H_\theta]$. It will follow that $S$ has ${<}\kappa=\omega_2$-many cofinal branches in $V[G][H]$, i.e., it is not a Kurepa tree. Since $S$ was an arbitrary $\omega_1$-tree, this will conclude the proof of Theorem \ref{Th:negKH+noAsubtrees}. The proof that there are no Kurepa trees in $V[G][H]$ is relatively long; for easier reading, it is divided into several Claims (Claims \ref{C:diff} to \ref{cl:last}).

First we observe that we can express the generic extension $V[G][H]$ as a forcing extension given by $(\P_\theta\times\K_\theta)*(\K_{\kappa}/\dot{H}_\theta\times \P^\theta)$: Since $\P$ is forcing equivalent to $\P_\theta\times \P^\theta$, we can view $G$ as a filter $G_\theta\times g$  which is $\P_\theta\times \P^\theta$-generic over $V$. Similarly, since $\K_\kappa$ is forcing equivalent to $\K_\theta*\K_{\kappa}/\dot{H}_\theta$, we can view $H$ as a filter $H_\theta*h$ which is generic over $V[G_\theta][g]$. Therefore $V[G][H]$ is equal to the generic extension $V[G_\theta][g][H_\theta][h]$. Moreover, since $\P^\theta$ and $\K_\theta$ live in $V[G_\theta]$ and $H_\theta$ is generic over $V[G_\theta][g]$, $V[G_\theta][g][H_\theta][h]=V[G_\theta][H_\theta][g][h]$. Similarly, we can exchange $g$ and $h$ since both $\P^\theta$ and $\K_{\kappa}/H_\theta$ live in $V[G_\theta][H_\theta]$ and $h$ is generic over $V[G_\theta][H_\theta][g]$. It follows that $V[G_\theta][H_\theta][g][h]=V[G_\theta][H_\theta][h][g]$.

Since $\P^\theta$ and $\P_\theta\times \K_\kappa$ are both $\omega_1$-closed in $V$, $\P^\theta$ is still $\omega_1$-closed in $V[G_\theta][H_\theta][h]$. By Lemma \ref{lemma: strat_closed_silver}, $\omega_1$-closed forcing cannot add new cofinal branches to an $\omega_1$-tree and hence any cofinal branch through $S$ is already in $V[G_\theta][H_\theta][h]$.

To show that $\K_{\kappa}/H_\theta$ cannot add a cofinal branch to $S$ over $V[G_\theta][H_\theta]$, we will work in $V[G_\theta]$. Note that $\K_{\kappa}/H_\theta$ is only $\omega_1$-distributive in $V[G_\theta][H_\theta]$ and therefore we cannot use Fact \ref{lemma: strat_closed_silver}. We will show that if there is a $\K_\theta*(\K_{\kappa}/\dot{H}_\theta)$-name $\dot{b}$ and a condition $(q^*,r^*)\in \K_\theta*(\K_\kappa/\dot{H}_\theta)$ which forces that $\dot{b}$  is a cofinal branch through $\dot{S}$ that is not in $V[G_\theta][\dot{H}_\theta]$, then there is a condition $q\le q^*$ which forces that $\dot{S}$ has an uncountable level. This is a contradiction, since we can assume that $(q^*,r^*)$ is in $H_\theta*h$ and that $q^*$ forces that $\dot{S}$ is an $\omega_1$-tree.

The proof is similar to Silver's argument for $\omega_1$-closed forcings: we build a tree $\mathcal{T}$ of conditions in $\K_\kappa$ labeled by ${^{<\omega}2}$, but since we are working with the quotient $\K_{\kappa}/\dot{H}_\theta$ and we work in $V[G_\theta]$ and not in $V[G_\theta][H_\theta]$, we will guide this construction by a decreasing sequence of length $\omega$ of conditions in $\K_\theta$, which will force conditions from $\mathcal{T}$ into the quotient $\K_{\kappa}/\dot{H}_\theta$.

We make a natural identification and view $(\dot{S}, <_{\dot{S}})$ as a name for a tree with underlying set $\omega_1$.

Work in $V[G_\theta]$. Assume that $\dot{b}$ is a $\K_\theta*(\K_{\kappa}/\dot{H}_\theta)$-name and $(q^*,\dot{r}^*)\in \K_\theta*(\K_{\kappa}/\dot{H}_\theta)$ forces that $\dot{b}$ is a cofinal branch through $\dot{S}$ that is not in $V[G_\theta][\dot{H}_\theta]$. We will build by induction on $\omega$ the following objects:

\begin{itemize}
\item a decreasing sequence $\seq{q_n}{ n<\omega }$  of conditions in $\K_\theta$ with $q_0 \le q^*$;
\item a labeled tree $\mathcal{T}=\set{r_s}{s\in {^{<\omega}}2 }$ of conditions in $\K_{\kappa}$, with $r_s \le r^*$ for all $s$;
\item a strictly increasing sequence $\seq{\gamma_n}{ n<\omega }$ of ordinals below $\omega_1$,
\end{itemize}

such that the following hold for all $n$ and all $s\in {^{<\omega}}2$ of length $n$:

\begin{enumerate}[(a)]
\item $q_n=r_s\rest\theta$; in particular $q_n\Vdash r_s\in \K_\kappa/\dot{H}_\theta$;
\item the conditions $(q_{n+1},r_{s^{\smallfrown}0})$  and  $(q_{n+1},r_{s^{\smallfrown}1})$ decide $\dot{b}$ up to $\gamma_{n+1}$ differently; i.e.\ there are $\delta
\le\gamma_{n+1} $ and $\tau_{s^{\smallfrown}0}\neq \tau_{s^{\smallfrown}1}$ forced to be in $\dot{S}_{\delta}$ such that $(q_n,r_{s^{\smallfrown}0})\Vdash \dot{b}(\delta)= \tau_{s^{\smallfrown}0}$ and $(q_n,r_{s^{\smallfrown}1})\Vdash \dot{b}(\delta)= \tau_{s^{\smallfrown}1}$;
\item $(q_{|t|},r_t)\le (q_{|s|},r_s)$ for all $s\subseteq t$ in  ${^{<\omega}}2$;
\item $\rng (f_{q_s}\rest [\theta,\kappa))\cap \rng (f_{q_t}\rest [\theta,\kappa))=\emptyset$ for distinct $s,t$ in ${^{<\omega}}2$ of the same length.
\end{enumerate}

\begin{definition}\label{def:system}
Let us call a system above, satisfying conditions (a)--(d), a \emph{labeled system} for $\dot{b}$. We say just a labeled system if $\dot{b}$ is clear from the context.
\end{definition}

A labeled system will be constructed below, using Claims \ref{C:diff} and \ref{cl:conditions}. First we prove an auxiliary claim which will be useful for the definition of the construction.

\begin{claim}\label{C:diff}
For every $(q,r^0) , (q,r^1) \le (q^*,r^*)\in \K_\theta*(\K_{\kappa}/\dot{H}_\theta)$ and $\gamma'<\omega_1$ there are $\gamma'<\gamma<\omega_1$, $(q',p^0)\le (q,r^0)$  and $(q',p^1)\le (q,r^1)$ such that $(q',p^0)$ and $(q',p^1)$ decide $\dot{b}(\gamma)$ differently and the condition $q'$ extends both $p^0\rest\theta$ and $p^1\rest\theta$.
\end{claim}

\begin{proof}
Let $(q,r^0) , (q,r^1) \le (q^*,r^*)$ and $\gamma'<\omega_1$ be given. Fix for the moment a $\K_\theta$-generic $F$ over $V[G_\theta]$ such that $q\in F$ and work in $V[G_\theta][F]$. Since $\dot{b}$ is forced by $(q^*,r^*)$ to be a new cofinal branch through $\dot{S}$ and $q^*\in F$, there are $\bar{p}^0\le r^0$, $\bar{p}^1\le r^0$ and $\gamma>\gamma'$ such that $\bar{p}^0$ and $\bar{p}^1$ decide $\dot{b}(\gamma)$ differently; i.e.\ there are $\tau^0\neq\tau^1$ in $\dot{S}_{\gamma}$ such that $\bar{p}^0\Vdash \dot{b}(\gamma)=\tau^0$ and $\bar{p}^1\Vdash \dot{b}(\gamma)=\tau^1$. 

Since $\bar{p}^0$ and $\bar{p}^1$ are in $\K_{\kappa}/F$, $\bar{p}^0\rest\theta$ and $\bar{p}^1\rest\theta$ are in $F$ and hence they are compatible. Let $\bar{q}$ be a common extension of $\bar{p}^0\rest\theta$, $\bar{p}^1\rest\theta$ and $q$ such that, for each $i < 2$, 
we have $(\bar{q}, \bar{p}^i) \Vdash \dot{b}(\gamma) = \tau^i$.

Now, we return back to working in $V[G_\theta]$. Since $\bar{q}$ extends both $\bar{p}^0\rest\theta$ and $\bar{p}^1\rest\theta$, it forces both of them into the quotient $\K_{\kappa}/\dot{H}_\theta$. It follows $(\bar{q},\bar{p}^0)$ and $(\bar{q},\bar{p}^1)$ are conditions in $\K_\theta*(\K_{\kappa}/\dot{H}_\theta)$ which extend $(q,r^0)$ and decide $\dot{b}(\gamma)$ differently.

Now, consider the condition $(\bar{q},r^1)$. Since $\dot{b}$ is a $\K_\theta*(\K_{\kappa}/\dot{H}_\theta)$ name for a cofinal branch through $\dot{S}$, there is an extension $(q',p^1)\le (\bar{q},r^1)$ which decides $\dot{b}(\gamma)$. Since $\tau^0\neq \tau^1$, $(q', p^1)$ cannot decide $\dot{b}(\gamma)$ as being equal to both of them. Let $p^0$ be $\bar{p}^i$, for $i<2$, such that $(q',p^1)$ and $(q',\bar{p}^i)$ disagree on $\dot{b}(\gamma)$. Then $q',p^0,p^1$ and $\gamma$ are as required. Note that $q'\le p^1\rest\theta$ since $q'$ forces $p^1$ into $\K_{\kappa}/\dot{H}_\theta$ and $\K_\theta$ is separative.
\end{proof}

Now, we are ready to construct our labeled system. The construction is by induction on $\omega$.

Let $\gamma_0$ be an arbitrary ordinal below $\omega_1$ and let $(q_0,r_\emptyset)$ be $(q^*,r^*)$. By Lemma \ref{L:into}, $q^*\le r^*\rest\theta$ since $q^*$ forces $r^*$ into $\K_{\kappa}/\dot{H}_\theta$ and $\K_\theta$ is separative. If $q^*\neq r^*\rest\theta$, we can extend $r^*$ appropriately to ensure the condition (a); for more details, see Claim \ref{cl:conditions} below in the successor step of the construction.

Now fix $n < \omega$ and assume that we have constructed $\gamma_n$, $q_n$ and $r_s$ for all $s \in {^n2}$. Let $\seq{s_i}{i<2^n}$ enumerate ${^n2}$. We describe how to construct 
$\gamma_{n+1}$, $q_{n+1}$, and $r_s$ for $s \in {^{n+1}}2$.

We proceed by induction on $2^n=m$. Let us start with $s_0$. By Claim \ref{C:diff} there are $q^0\le q_n$, $r'_{s_0{}^\smallfrown 0}$,$r'_{s_0{}^\smallfrown 1}\le r_{s_0}$ and $\gamma^0\ge\gamma_n$ such that $(q^0,r'_{s_0{}^\smallfrown 0})$, $(q^0,r'_{s_0{}^\smallfrown 1})$ decide $\dot{b}(\gamma^0)$ differently and $q^0$ extends both $r'_{s_0{}^\smallfrown 0}\rest\theta$ and $r'_{s_0{}^\smallfrown 1}\rest\theta$. Now fix $1 \leq i < m$, and suppose that $\gamma^{i-1}$ and $q^{i-1}$ have been constructed. By Claim \ref{C:diff} there are $q^{i}\le q^{i-1}$, $r'_{s_{i}{}^\smallfrown 0}$, $'r_{s_{i}{}^\smallfrown 1}\le r'_{s_{i}}$ and $\gamma^{i}\ge\gamma^{i-1}$ such that $(q^{i},r'_{s_{i}{}^\smallfrown 0})$, $(q^{i},r'_{s_{i}{}^\smallfrown 1})$ decide $\dot{b}(\gamma^{i})$ differently and $q^{i}$ extends both $r'_{s_i{}^\smallfrown 0}\rest\theta$ and $r'_{s_{i}{}^\smallfrown 1}\rest\theta$. 

Let $q'_{n+1}$ be $q^{m-1}$ and $\gamma_{n+1}$ be $\gamma^{m-1}$. It follows by the construction that the objects  $q'_{n+1}$, $\gamma_{n+1}$, $r'_{s^\smallfrown j}$, for $j<2$ and all $s\in {^n2}$ satisfy the desired conditions (b) and (c).

However, we have only ensured that $q'_{n+1}$ extends $r'_{s}\rest\theta$ for $s\in {^{n+1}2}$, but not that they are equal as is required in condition (a), and we have not ensured condition (d) either. To ensure conditions (a) and (d), we define appropriate extensions of $q'_{n+1}$ and $r'_{s}$ for all $s\in {^{n+1}2}$.  

\begin{claim}\label{cl:conditions}
The objects constructed above can be extended to satisfy conditions (a)--(d) of a labeled system in Definition \ref{def:system}.
\end{claim}

\begin{proof}
Since $q'_{n+1}$ extends $r'_{s}\rest\theta$ for all $s\in {^{n+1}2}$, $T_{q'_{n+1}}$ is an end-extension of $T_{r'_{s}}$ for all $s\in {^{n+1}2}$. Let $\eta+1$ be the height of $T_{q'_{n+1}}$ and let $T'$ be a one-level extension of $T_{q'_{n+1}}$ such that for every node $\tau$ in $T_{q'_{n+1}}$ on level $\eta$ we add at least countably many new nodes above $\tau$ into $T'$. The height of $T'$ is $\eta+2$, and $T'$ is normal and infinitely splitting since $T_{q'_{n+1}}$ is normal and infinitely splitting. Let $f'$ be a function such that $\dom(f')=\dom(f_{q'_{n+1}})$ and for each $\alpha\in\dom(f')$ let $f'(\alpha)\supseteq f_{q'_{n+1}}(\alpha)$ be some node of $T'$ on level $\eta+1$. Then $q_{n+1}=(T',f')$ is a condition in $\K_\theta$ and it extends $q'_{n+1}$. 

Now, we define extensions of $r'_{s}$ for $s\in {^{n+1}2}$. For $s\in {^{n+1}2}$, let $f_s$ be a function such that $\dom(f_s)=\dom(f_{r'_{s}})\cap [\theta,\kappa)$ and for each $\alpha\in\dom(f_s)$, let $f_s(\alpha)\supseteq f_{r'_{s}}(\alpha)$ be some node of $T'$ on level $\eta+1$. Moreover, ensure that $\rng(f_s)\cap\rng(f_t)=\emptyset $ for all $s\neq t \in {^{n+1}2}$. Note that we can find such functions since $T'$ is infinitely splitting and the sets $\dom(f_{r'_{s}})$ for $s\in {^{n+1}2}$ are at most countable. Set $r_s=(T',f'\cup f_s)$ for $s\in {^{n+1}2}$. Clearly, $r_s$ are conditions in $\K_\kappa$ for all $s\in {^{n+1}2}$ such that $q_{n+1}$ is equal to $r_s\rest\theta$, hence condition (a) holds. Condition (d) follows from the definition of $f_s$ for $s\in {^{n+1}2}$.

Since $q_{n+1}$ extends $q'_{n+1}$ and $r_{s}$ extends $r'_{s}$ for all 
$s\in {^{n+1}2}$, conditions (b) and (c) are still satisfied for $q_{n+1}$, $\gamma_{n+1}$, $r_{s^\smallfrown j}$, for $j<2$ and all $s\in {^n2}$.
\end{proof}

This completes our construction of a labeled system. Let $\gamma$ be the supremum of $\seq{\gamma_n}{ n<\omega }$. To finish the proof of the whole Theorem \ref{Th:negKH+noAsubtrees}, we would like to find a lower bound $q$ of the sequence $\seq{q_n}{ n<\omega }$ and a lower bound $r_x$ of sequences $\seq{r_{x\rest n}}{n<\omega}$ for all $x\in {^{\omega}2}$ such that $q$ forces $r_x$ into the quotient $\K_{\kappa}/\dot{H}_\theta$ for every $x$, thus ensuring that every $(q,r_x)$ is a condition in $\K_\theta*(\K_{\kappa}/\dot{H}_\theta)$. If this is the case, then $q$ forces that level $\gamma$ of $\dot{S}$ has size $2^\omega$, and hence $q$ forces that $\dot{S}$ is not an $\omega_1$-tree, finishing the proof. 

In the rest of the proof we will construct such conditions. Before we start, note that we do not need to ensure that $q$ forces $r_x$ into quotient for all $x\in {^\omega2}$, it is enough to have this for uncountably many $x\in{^\omega2}$.

Let $T^*=\bigcup_{n<\omega} T_{q_n}$ and $a=\bigcup_{n<\omega}\dom(f_{q_n})$. Then $T^*$ is a normal tree with limit height which does not contain a copy of ${^{<\omega+1}2}$. Let $\eta<\omega_1$ denote the height of $T^*$. By condition (a), $T^*=\bigcup_{n<\omega} T_{r_{x\rest n}}$ and $a=\bigcup_{n<\omega}\dom(f_{r_{x\rest n}})\cap\theta$ for all $x\in {^{\omega}2}$.

Note that if we want to ensure that a lower bound of $\seq{q_n}{ n<\omega }$ forces a lower bound of $\seq{r_{x\rest n}}{n<\omega}$, for some $x\in{^\omega 2}$, into the quotient, we are obliged not only to extend all cofinal branches through $T^*$ which are given by functions $\set{f_{q_n}}{n < \omega}$ (recall the proof that $\K_\kappa$ is $\omega_1$-closed), but also to extend all cofinal branches given by $\set{f_{r_{x\rest n} }}{n < \omega}$, otherwise it can happen that the lower bounds will be incompatible. Therefore if we want to ensure that a lower bound of $\seq{q_n}{ n<\omega }$ forces lower bounds of $\seq{r_{x\rest n}}{n<\omega}$ for uncountably many $x\in {^\omega 2}$ into the quotient, we are obliged to extend uncountably many cofinal branches of $T^*$. This can be done because the trees in the conditions can be wide, but we need to make sure that while doing it, we do not add a copy of ${^{<\omega+1}2}$. To argue that $T_q$ does not contain a copy of ${^{<\omega+1}2}$, we use following claim which says that if we add only countably many cofinal branches to a tree which does not contain a copy of ${^{<\omega+1}2}$, then we do not add a copy of ${^{<\omega+1}2}$.

\begin{claim}\label{Cl:notperfect}
Let $T$ be a tree with countable height $\eta+1$, where $\eta$ is limit. Let $D$ be a countable set of cofinal branches of $T\rest\eta$. If $T$ does not contain a copy of ${^{<\omega+1}2}$, then $T'$ also does not contain a copy of ${^{<\omega+1}2}$, where $T'=T\cup \set{ d}{d\in D}$.
\end{claim}

\begin{proof}
This is a straightforward consequence of the following well-known fact: If $P$ is a perfect subset of ${^\omega 2 }$ and $Q\sub P$ is countable, then $P\setminus Q$ contains a perfect set.
\end{proof}

To build suitable lower bounds, we will proceed similarly as in the proof of Lemma \ref{L:K-closed}. First apply the construction in Lemma \ref{L:K-closed} to find $q'$ which is a lower bound of $\seq{{q_n}}{n < \omega}$ such that $T_{q'}$ has height $\eta+1$ and $\dom(f_{q'})=a$. Note that each node $\tau\in T_{q'}$ on level $\eta$ determines a cofinal branch through $T^*$ which we denote $b_\tau$. Note also that $T_{q'}=T^*\cup\set{b_\tau}{\tau\in (T_{q'})_\eta}$. 

For each $x\in {^\omega2}$, let $a_x=\bigcup_{n<\omega} \dom(f_{r_x\rest n})\setminus\theta$. In order to ensure that $q$ is compatible with a lower bound of  $\seq{r_{x\rest n} }{n < \omega}$, we are obliged to extend cofinal branches given by $\set{r_{x\rest n} }{n < \omega}$. For $\alpha \in a_x$, let 

\begin{equation}
      d^x_\alpha := \bigcup \set{f_{r_{x\rest n} }(\alpha)}{ n < \omega \wedge \alpha \in \dom(f_{r_{x\rest n} })}
\end{equation}
and let $D_x=\set{d^x_\alpha}{\alpha\in a_x}$. Note that $D_x\cap D_y=\emptyset$ for all $x\neq y\in {^\omega2}$ by condition (d). Moreover, let as define $f_x$ to be a function with domain $a_x$ such that for each $\alpha\in a_x$, $f_x(\alpha)=d^x_\alpha$.

We will use the following bookkeeping device to extend cofinal branches in uncountably many $D_x$ (but not in all $D_x$ because this might add a copy of  ${^{<\omega + 1}}2$).
    
Let $\seq{\iota_\xi}{\xi < \omega_1}$ enumerate all isomorphic embeddings $\iota$ from ${^{<\omega}}2$ to $T^*$ such that, for every $x \in {^{\omega}}2$, the function $\bigcup \set{\iota(x \restriction n)}{n < \omega}$ is a cofinal branch through $T^*$. Note that this is possible, due to the fact that $\CH$ holds and $|T^*| \leq \omega_1$. For each $\xi < \omega_1$, let 
\begin{equation}
      [\iota_\xi] := \left\{\bigcup \{\iota_\xi(x \restriction n) \mid n < \omega\} \ \middle| \ x \in 
      {^{\omega}}2\right\}.
\end{equation}
Note that $[\iota_\xi] \subseteq [T^*]$, and $|[\iota_\xi]| = 2^\omega = \omega_1$. 

We now recursively build two disjoint sequences $\seq{x_\xi}{\xi<\omega_1}$ and $\seq{y_\xi}{\xi<\omega_1}$ of functions in ${^\omega 2}$. The first sequence picks functions $x$ in ${^\omega 2}$ for which we will extend the corresponding cofinal branches $D_x$. The second sequence picks functions $y$ in ${^\omega 2}$ for which we will forbid the extension of corresponding cofinal branches in $D_y$ to ensure that $T_q$ does not contain a copy of ${^{<\omega + 1}}2$. Note that we need here that the $D_x$'s are pairwise disjoint.

Suppose that we have constructed $x_\zeta$ and $y_\zeta$ for all $\zeta < \xi$, where $\xi < \omega_1$. First, set $x_\xi$ to be any function such that $x_\xi\not\in \set{x_\zeta}{\zeta<\xi}\cup\set{y_\zeta}{\zeta<\xi}$. Next, if there is $y \not\in \set{x_\zeta}{\zeta\le\xi}\cup\set{y_\zeta}{\zeta<\xi}$ such that
\begin{equation}\label{eq:nonempty}
([\iota_\xi]\setminus \set{b_\tau}{\tau\in (T_{q'})_\eta})\cap D_y\neq\emptyset,
\end{equation}
then let $y_\xi$ be any such $y$. If there is no such $y$, then let $y_\xi$ be any $y$ which is not in $\set{x_\zeta}{\zeta\le\xi}\cup\set{y_\zeta}{\zeta<\xi}$.

At the end of the construction, set $T_q = T_{q'} \cup \bigcup_{\xi<\omega_1} D_{x_\xi}$ and $f_q=f_q'$. Set $T_{r_{x_\xi}}=T_q$ and $f_{r_{x_\xi}}=f_{q'}\cup f_{x_\xi}$ for every $\xi<\omega_1$.
   
\begin{claim} \label{cl:last} The following hold:
\begin{enumerate}[(i)]
\item $q$ is a condition in $\K_\theta$;
\item $r_{x_\xi}$ is a condition in $\K_\kappa$ for all $\xi<\omega_1$;
\item $q$ forces $r_{x_\xi}$ into the quotient $\K_\kappa/\dot{H}_\theta$, for all $\xi<\omega_1$.
\end{enumerate}
\end{claim}

\begin{proof}
Note that $T_q$ is a normal tree since $T_{q'}$ is normal and $T_q = T_{q'} \cup \bigcup_{\xi<\omega_1} D_{x_\xi}$. The only nontrivial thing to check is that $T_q$ does not contain a copy of ${^{<\omega+1}}2$. Assume for contradiction that $T_q$ contains a copy of ${^{<\omega+1}}2$; hence there is $\xi<\omega_1$ such that $[\iota_\xi] \subseteq (T_q)_\eta$. Let $I$ be the set of all $\zeta<\omega_1$ such that $(D_{x_\zeta}\setminus (T_{q'})_\eta) \cap [\iota_\xi] \neq \emptyset$; in particular, $[\iota_\xi] \sub  (T_{q'})_\eta \cup (\bigcup_{\zeta\in I} D_{x_\zeta})$. Note that $I$ cannot be countable by Claim \ref{Cl:notperfect}, since $T_{q'}$ does not contain copy of ${^{<\omega+1}}2$. Hence $I$ is unbounded in $\omega_1$ which means that there is $\zeta>\xi$ such that $(D_{x_\zeta}\setminus (T_{q'})_\eta) \cap[\iota_\xi] \neq \emptyset$. Therefore in step $\xi$ of the induction, condition (\ref{eq:nonempty}) was satisfied and we chose $y_\xi$ such that the intersection $(D_{y_\xi}\setminus (T_{q'})_\eta) \cap[\iota_\xi]$ is nonempty. This is a contradiction since $D_{y_\xi}$ should be disjoint from $D_{x_\zeta}$ for all $\zeta\in I$, and we did not add any cofinal branches from $D_{y_\xi}\setminus (T_{q'})_\eta$ into $T_q$. This shows that $q$ is a condition in $\K_\theta$ and $r_{x_\xi}$ are conditions in $\K_\kappa$ for all $\xi<\omega_1$.

By the definition of $r_{x_\xi}$, $r_{x_\xi}\rest\theta= q$ and hence $q$ forces $r_{x_\xi}$ into the quotient $\K_\kappa/\dot{H}_\theta$ for every $\xi<\omega_1$.
\end{proof}

It is straightforward to check that $q$ forces $\dot{S}_\gamma$ to be uncountable: Let $F$ be any $\K_\theta$-generic which contains $q$. Then $r_{x_\xi}$ is a condition in the quotient $\K_\theta/F$ for every $\xi<\omega_1$. Let $r^*_{x_\xi}$ be an extension of $r_{x_\xi}$ which decides $\dot b(\gamma)=\tau_\xi$. Then for every $\xi\neq\zeta<\omega_1$, $\tau_\xi\neq\tau_\zeta$ since $(q,r_{x_\xi})$ and $(q,r_{x_\zeta})$ decide $\dot{b}$ differently below $\gamma$. This concludes the proof of Theorem \ref{Th:negKH+noAsubtrees}.
\end{proof}

\section{Kurepa trees and continuous images} \label{sect: cts_image}

In this section, we prove Theorem B, answering a question of L\"{u}cke and Schlicht from 
\cite{lucke_schlicht_descriptive}. In particular, we will construct a model of 
$\ZFC$ in which $\GCH$ holds, there exist $\omega_2$-Kurepa trees, and, for every 
$\omega_2$-Kurepa tree $S \subseteq {^{<\omega_2}}\omega_2$, $[S]$ is not a continuous 
image of ${^{\omega_2}}\omega_2$. 

We will also prove that, in the model that we construct, 
closed subsets of ${^{\omega_2}}\omega_2$ satisfy the strongest possible perfect 
set property compatible with the existence of $\omega_2$-Kurepa trees. Note that, 
by the discussion at the beginning of Section \ref{section: weak_kurepa}, if 
$S \subseteq{^{<\omega_2}}\omega_2$ is a Kurepa tree, then 
$\mathsf{PSP}_{\omega_1+1}([S])$ necessarily fails. In the model we construct, 
this failure will be sharp: clause (5) in the statement of Theorem \ref{main_thm} 
below will imply that, for every closed $E \subseteq {^{\omega_2}}\omega_2$, 
we have $\mathsf{PSP}_{\omega_1}(E)$.

In what follows, when referring to trees, we will sometimes write 
\emph{countably closed} in place of \emph{$({<}\omega_1)$-closed}.

\begin{theorem} \label{main_thm}
  Suppose that there is an inaccessible cardinal $\kappa$. Then there is a 
  forcing extension in which
  \begin{enumerate}
    \item $\kappa = \omega_3$;
    \item $\GCH$ holds;
    \item there is an $\omega_2$-Kurepa tree;
    \item for every $\omega_2$-Kurepa tree $S \subseteq {^{<\omega_2}}\omega_2$, 
    $[S]$ is not a continuous image of ${^{\omega_2}}\omega_2$;
    \item for every closed subset $E \subseteq {^{\omega_2}}\omega_2$, there 
    is $X \subseteq E$ with $|X| \leq \omega_2$ such that $E \setminus X$ 
    is $\omega_1$-perfect.
  \end{enumerate}
\end{theorem}

\begin{proof}
  We can assume that $\GCH$ holds in the ground model. Let $\P := 
  \Coll(\omega_2, < \kappa)$, and let $\Q$ be the forcing to add a
  countably closed $\omega_2$-Kurepa tree with $\kappa$-many cofinal branches. 
  More precisely, $\Q$ consists of all pairs $q = (T_q, f_q)$ such that
  \begin{itemize}
    \item there is an $\eta_q < \omega_2$ such that $T_q$ is a normal, splitting, 
    countably closed subtree of ${^{<\eta_q + 1}}\omega_2$;
    \item for all $\xi \leq \eta_q$, $|T_q \cap {^\xi}\omega_2| \leq 
    \omega_1$;
    \item $f_q$ is a partial function of size $\omega_1$ from $\kappa$ to 
    $T_q \cap {^{\eta_q}}\omega_2$.
  \end{itemize}
  If $q_0, q_1 \in \Q$, then $q_1 \leq q_0$ if
  \begin{itemize}
    \item $\eta_{q_1} \geq \eta_{q_0}$;
    \item $T_{q_1} \cap {^{<\eta_{q_0}+1}}\omega_2 = T_{q_0}$;
    \item $\dom(f_{q_1}) \supseteq \dom(f_{q_0})$;
    \item for all $\alpha \in \dom(f_{q_0})$, $f_{q_1}(\alpha) 
    \sqsupseteq f_{q_0}(\alpha)$.
  \end{itemize}
  We also include $(\emptyset, \emptyset)$ in $\Q$ as $1_{\Q}$. 
  By standard arguments, $\P \times \Q$ is $\omega_2$-closed and 
  has the $\kappa$-cc.
  
  For a cardinal $\delta < \kappa$, let $\P_\delta := 
  \Coll(\omega_2, < \delta)$, and let $\Q_\delta$ be the set 
  of $q \in \Q$ for which $\dom(q) \subseteq \delta$.
  It is routine to verify that, for all $\delta < \kappa$, the map 
  $\pi_\delta : \P \times \Q \ra \P_\delta \times \Q_\delta$
  defined by letting $\pi_\delta(p,q) = (p \restriction \delta, 
  (T_q, f_q \restriction \delta))$ is a projection.
  
  \begin{claim} \label{ctbly_closed_claim}
    Let $\delta < \kappa$ be a cardinal. Then, in $V^{\P_\delta 
    \times \Q_\delta}$, the quotient forcing $\R_\delta := 
    (\P \times \Q)/(\P_\delta \times \Q_\delta)$ is countably 
    closed.
  \end{claim}
  
  \begin{proof}
    Let $G_\delta \times H_\delta$ be $\P_\delta \times 
    \Q_\delta$-generic over $V$, and move to $V[G_\delta \times 
    H_\delta]$. Let $\langle (p_n,q_n) \mid n < \omega \rangle$ be 
    be a decreasing sequence from $\P \times \Q$ such that, for all 
    $n < \omega$, we have $\pi_\delta(p_n,q_n) \in G_\delta \times 
    H_\delta$. Note that this sequence is in fact in $V$, by the 
    distributivity of $\P \times \Q$. 
    
    Let $p = \bigcup\{p_n \mid n < \omega\}$. It is evident that 
    $p \in \P$ and $p \restriction \delta \in G_\delta$. We next 
    define a condition $q \in \Q$. If the sequence $\langle T_{q_n} 
    \mid n < \omega \rangle$ is eventually constant, then, by 
    removing an initial segment of the sequence, we can assume it is 
    constant. Let $T_q$ be this constant value, and let 
    $f_q = \bigcup\{f_{q_n} \mid n < \omega\}$. Then $(p,q)$ is a 
    lower bound for $\langle (p_n,q_n) \mid n < \omega \rangle$ 
    and $\pi_\delta(p,q) \in G_\delta \times H_\delta$.
    
    So suppose that $\langle T_{q_n} \mid n < \omega \rangle$ is 
    not eventually constant. Let $T^* = \bigcup\{T_{q_n} \mid n < 
    \omega \}$. Then $T^*$ is a normal subtree 
    of ${^{<\eta}}\omega_2$ for some $\eta < \omega_2$ of countable 
    cofinality. Moreover, by $\CH$, we have $|[T^*]| \leq \omega_1$. 
    Let $T_q := T^* \cup [T^*]$. By construction 
    and the fact that each $T_{q_n}$ is countably closed, it follows 
    that $T_q$ is a normal countably closed subtree of 
    ${^{<\eta+1}}\omega_2$. Since the 
    trees in the first coordinate of conditions in $\Q$ are 
    required to be countably closed, we must have $(T_q, 
    \emptyset) \in H_\delta$. Let $D = \bigcup\{\dom(f_{q_n}) \mid 
    n < \omega\}$. Define $f_q:D \ra B$ by letting 
    \[
      f_q(\alpha) = \bigcup\{f_{q_n}(\alpha) \mid n < \omega ~ \wedge
      ~ \alpha \in \dom(f_{q_n})\}
    \]
    for all $\alpha \in D$. Then $(p,q)$ is a lower bound 
    for $\langle (p_n, q_n) \mid n < \omega \rangle$ and 
    $\pi_\delta(p,q) \in G_\delta \times H_\delta$.
  \end{proof}
  
  Let $G \times H$ be $\P \times \Q$-generic over $V$ and, for all 
  cardinals $\delta < \kappa$, let $G_\delta \times H_\delta$ be 
  the $\P_\delta \times \Q_\delta$-generic filter induced by 
  $G \times H$. We claim that $V[G \times H]$ is the desired forcing 
  extension. We have $(\omega_2)^V = (\omega_2)^{V[G \times H]}$, 
  and $\kappa = (\omega_3)^{V[G \times H]}$; moreover, $\GCH$ holds in 
  $V[G \times H]$. Let $T = \bigcup\{T_q \mid q \in H\}$ and, for each 
  $\alpha < \kappa$, let 
  \[
    b_\alpha = \bigcup\{f_q(\alpha) \mid q \in H ~ \wedge ~ \alpha 
    \in \dom(f_q)\}.
  \]
  Then $T$ is a normal, countably closed subtree of ${^{<\omega_2}}\omega_2$, 
  all of whose levels have size at most $\omega_1$. Moreover, for 
  all $\alpha < \kappa$, $b_\alpha$ is a cofinal branch through $T$ 
  and, by genericity, for all $\alpha < \beta < \kappa$, we have 
  $b_\alpha \neq b_\beta$. Therefore, $T$ is an $\omega_2$-Kurepa 
  tree in $V[G \times H]$.
  
  To verify clause (4) in the statement of the theorem, we must 
  show that, in $V[G \times H]$, there is no 
  $\omega_2$-Kurepa tree $S \subseteq {^{<\omega_2}}\omega_2$ such 
  that $[S]$ is a continuous image of ${^{\omega_2}}\omega_2$. Suppose 
  for the sake of contradiction that $S$ is such a tree, and let 
  $g:{^{\omega_2}}{\omega_2} \ra [S]$ be a continuous surjection.
  Define a function $h:{^{<\omega_2}}\omega_2 \ra S \cup [S]$ by 
  letting $h(\sigma) = \bigwedge \{g(x) \mid x \in N_\sigma\}$ for 
  all $\sigma \in {^{<\omega_2}}\omega_2$, and let 
  \[
    W = \{\sigma \in {^{<\omega_2}}\omega_2 \mid \dom(h(\sigma)) \geq 
    \dom(\sigma)\}.
  \]
  
  Note that $h$ is monotone, i.e., for all $\sigma, \tau 
  \in {^{<\omega_2}}\omega_2$, if $\sigma \sqsubseteq \tau$, 
  then $h(\sigma) \sqsubseteq h(\tau)$. It follows that 
  $W$ is $(<\omega_2)$-closed, i.e., if $\eta < \omega_2$ 
  and $\langle \sigma_\xi \mid \xi < \eta \rangle$ is a 
  $\sqsubseteq$-increasing sequence of elements of $W$, then 
  $\bigcup\{\sigma_\xi \mid \xi < \eta\} \in W$.
  
  \begin{claim} \label{dense_claim}
    For all $x \in {^{\omega_2}}\omega_2$, $C_x := \{\alpha < 
    \omega_2 \mid x \restriction \alpha \in W\}$ is a club in 
    $\omega_2$.
  \end{claim}
  
  \begin{proof}
    Fix $x \in {^{\omega_2}}\omega_2$. The fact that $C_x$ is closed 
    follows immediately from the fact that $h$ is monotone.
    To see that it is unbounded, fix an $\alpha_0 < \omega_2$. We will 
    find $\alpha \in C_x \setminus \alpha_0$. Starting with 
    $\alpha_0$, recursively define an increasing 
    sequence $\langle \alpha_n \mid n < \omega \rangle$ of ordinals 
    below $\omega_2$ such that, for all $n < \omega$ and all 
    $y \in N_{x \restriction \alpha_{n+1}}$, we have
    $g(y) \restriction \alpha_n = g(x) \restriction \alpha_n$. It is 
    straightforward to build such a sequence due to the continuity of $g$. 
    
    Let $\alpha = \sup\{\alpha_n \mid n < \omega\}$. For all 
    $y \in N_{x \restriction \alpha}$ and all $n < \omega$, we have 
    $g(y) \restriction \alpha_n = g(x) \restriction \alpha_n$; therefore, 
    $g(y) \restriction \alpha = g(x) \restriction \alpha$. In particular, 
    $h(x \restriction \alpha) \sqsupseteq g(x) \restriction \alpha$, so 
    $\alpha \in C_x$.
  \end{proof}
  
  Note that $S$, $h$, and $W$ are objects of size $\omega_2$, so, by 
  the chain condition of $\P \times \Q$, we can fix a $\delta < \kappa$ 
  so that $S, h, W \in V[G_\delta \times H_\delta]$ and, in that model, 
  the quotient forcing $\R_\delta$ forces them to have the relevant properties 
  isolated above. In particular, letting $\dot{g}$ be an $\R_\delta$-name 
  for $g$, it is forced by every condition in $\R_\delta$ that $S$ is an 
  $\omega_2$-Kurepa tree, $\dot{g}$ is a continuous surjection from 
  ${^{\omega_2}}\omega_2$ to $[S]$, and $h$ and $W$ are as defined above 
  from $S$ and $\dot{g}$.
  
  Work now in $V[G_\delta \times H_\delta]$. In this model, $2^{\omega_2} 
  < \kappa$; therefore, forcing with $\R_\delta$ must add new cofinal 
  branches to $S$. Let $\dot{b}$ be an $\R_\delta$-name for a cofinal 
  branch through $S$ such that $\Vdash_{\R_\delta} \dot{b} \notin 
  V[G_\delta \times H_\delta]$. Let $\dot{x}$ be an $\R_\delta$-name 
  for an element of ${^{\omega_2}}\omega_2$ such that 
  $\Vdash_{\R_\delta} \dot{g}(\dot{x}) = \dot{b}$.
  
  Now, by recursion on $\dom(\sigma)$, we construct a labeled tree 
  $\langle (r_\sigma, b_\sigma, x_\sigma) \mid \sigma \in {^{<\omega_1}}2 \rangle$ 
  of elements of $\R_\delta \times S \times {^{<\omega_2}}\omega_2$.
  We will arrange so that the following requirements are satisfied.
  \begin{enumerate}
    \item For all $\sigma \in {^{<\omega_1}}2$, we have
    \begin{enumerate}
      \item $r_\sigma \Vdash_{\R_\delta}``\dot{b} \sqsupseteq 
      b_\sigma \text{ and } \dot{x} \sqsupseteq x_\sigma$";
      \item $x_\sigma \in W$;
      \item $h(x_\sigma) \sqsupseteq b_\sigma$.
    \end{enumerate}     
    \item For all $\sigma \sqsubseteq \tau \in {^{<\omega_1}}2$, we have 
    $r_\tau \leq r_\sigma$, $b_\tau \sqsupseteq b_\sigma$, and 
    $x_\tau \sqsupseteq x_\sigma$.     
    \item For all $\sigma \in {^{<\omega_1}}2$, $b_{\sigma^\frown 
    \langle 0 \rangle}$ and $b_{\sigma^\frown \langle 1 \rangle}$ are 
    incomparable in $S$.
  \end{enumerate}
  
  Begin by letting $r_\emptyset = (\emptyset, \emptyset)$ and 
  $b_\emptyset = x_\emptyset = \emptyset$. Next, suppose that 
  $\sigma \in {^{<\omega_1}}2$ is of 
  limit length, and suppose that $(r_{\sigma \restriction \eta}, 
  b_{\sigma \restriction \eta}, x_{\sigma \restriction \eta})$ has 
  been defined for all $\eta < \dom(\sigma)$. By Lemma \ref{ctbly_closed_claim}, 
  $\R_\delta$ is countably closed, so we can let $r_\sigma$ be a 
  lower bound of $\langle r_{\sigma \restriction \eta} \mid \eta < 
  \dom(\sigma) \rangle$. Let $b_\sigma = \bigcup\{b_{\sigma \restriction \eta} 
  \mid \eta < \dom(\sigma)\}$ and $x_\sigma = \bigcup\{
  x_{\sigma \restriction \eta} \mid \eta < \dom(\sigma)\}$.
  For all $\eta < \dom(\sigma)$, we have $r_\sigma \Vdash_{\R_\delta} ``\dot{b} 
  \sqsupseteq b_{\sigma \restriction \eta} \text{ and } \dot{x} 
  \sqsupseteq x_{\sigma \restriction \eta}$". As a result, we get 
  $r_\sigma \Vdash_{\R_\delta}``\dot{b} \sqsupseteq b_\sigma \text{ and } 
  \dot{x} \sqsupseteq x_\sigma$". In particular, it 
  follows that $b_\sigma$ is in fact in $S$. The fact that $x_\sigma \in W$ 
  and $h(x_\sigma) \sqsupseteq b_\sigma$ follows from the closure of 
  $W$ and the monotonicity of $h$. We have therefore satisfied all of 
  the requirements of the construction.
  
  Finally, suppose that $\sigma \in {^{<\omega_1}}2$ and we have 
  defined $(r_\sigma, b_\sigma, x_\sigma)$. We describe how to 
  define $(r_{\sigma ^\frown \langle i \rangle}, b_{\sigma ^\frown 
  \langle i \rangle}, x_{\sigma ^\frown \langle i \rangle})$ for 
  $i < 2$. First, since $\Vdash_{\R_\delta} \dot{b} \notin V[G_\delta 
  \times H_\delta]$, we can find $r^*_i \leq r_\sigma$ and $b^*_i 
  \sqsupseteq b_\sigma$ for $i < 2$ such that $b^*_0$ and $b^*_1$ are 
  incomparable in $S$ and, for $i < 2$, we have $r^*_i \Vdash_{\R_\delta} 
  \dot{b} \sqsupseteq b^*_i$.
  
  By Claim \ref{dense_claim}, there are forced to be unboundedly many 
  $\xi < \omega_2$ such that $\dot{x} \restriction \xi \in W$. Therefore, 
  for each $i < 2$, we can find $r_{\sigma ^\frown \langle i \rangle} 
  \leq r^*_i$ and $x_{\sigma ^\frown \langle i \rangle} \in W$ such that
  \begin{itemize}
    \item $x_{\sigma ^\frown \langle i \rangle} \sqsupseteq x_\sigma$;
    \item $\dom(x_{\sigma ^\frown \langle i \rangle}) \geq \dom(b^*_i)$;
    \item $r_{\sigma ^\frown \langle i \rangle} \Vdash_{\R_\delta} 
    \dot{x} \sqsupseteq x_{\sigma ^\frown \langle i \rangle}$.
  \end{itemize}
  Finally, for $i < 2$, let $b_{\sigma ^\frown \langle i \rangle} = 
  h(x_{\sigma^\frown \langle i \rangle})$. Note first that, since 
  \[
    r_{\sigma^\frown \langle i \rangle} \Vdash \dot{g}(\dot{x}) = 
    \dot{b} \wedge \dot{x} \sqsupseteq x_{\sigma^\frown \langle i \rangle},
  \]
  it follows that $r_{\sigma^\frown \langle i \rangle} \Vdash \dot{b} 
  \sqsupseteq b_{\sigma^\frown \langle i \rangle}$. Next observe that 
  $b_{\sigma ^\frown \langle i \rangle}$ must be in $S$, since if it were in 
  $[S]$, then we would have $r_{\sigma^\frown \langle i \rangle} \Vdash 
  \dot{b} = b_{\sigma^\frown \langle i \rangle}$, contradicting the fact 
  that $\dot{b}$ is forced not to be in $V[G_\delta \times H_\delta]$. 
  We have therefore satisfied requirements (1)--(3) above, and we can 
  continue with the construction.
  
  At the end of the construction, since $\CH$ holds, $\{b_\sigma \mid 
  \sigma \in {^{<\omega_1}}2\}$ is a subset of $S$ of size $\omega_1$. 
  Therefore, we can fix a $\xi < \omega_2$ such that $b_\sigma 
  \in {^{<\xi}}\omega_2$ for all $\sigma \in {^{<\omega_1}}2$.
  For each $\nu \in {^{\omega_1}}2$, let $b_\nu = \bigcup\{ b_{\nu \restriction 
  \eta} \mid \eta < \omega_1\}$.
  
  \begin{claim}
    For all $\nu \in {^{\omega_1}}2$, $b_\nu \in S$.
  \end{claim}
  
  \begin{proof}
    Fix $\nu \in {^{\omega_1}}2$. Let $x_\nu = \bigcup\{x_{\nu \restriction 
    \eta} \mid \eta < \omega_1\}$. By monotonicity of $h$, we have 
    $h(x_\nu) \sqsupseteq h(x_{\nu \restriction \eta}) \sqsupseteq 
    b_{\nu \restriction \eta}$ for all $\eta < \omega_1$. Therefore, by 
    definition of $b_\nu$, it follows that $h(x_\nu) \sqsupseteq b_\nu$. 
    Since $h$ maps into $S \cup [S]$, we get $b_\nu \in S$, as desired.
  \end{proof}
  
  \begin{claim}
    For all distinct $\nu_0, \nu_1 \in {^{\omega_1}}2$, we have $b_{\nu_0} \neq 
    b_{\nu_1}$.
  \end{claim}
  
  \begin{proof}
    Fix distinct $\nu_0, \nu_1 \in {^{\omega_1}}2$. Let $\eta < \omega_1$ be least 
    such that $\nu_0(\eta) \neq \nu_1(\eta)$. Without loss of generality, 
    assume that $\nu_i(\eta) = i$ for $i < 2$. Let $\sigma = \nu_0 \restriction \eta$. 
    Then, for $i < 2$, we have $b_{\nu_i} \sqsupseteq b_{\sigma ^\frown \langle i 
    \rangle}$. Since $b_{\sigma ^\frown \langle 0 \rangle}$ and $b_{\sigma 
    ^\frown \langle 1 \rangle}$ are incomparable in $S$,
    $b_{\nu_0}$ and $b_{\nu_1}$ are incomparable, as well; in particular, 
    they are not equal.
  \end{proof}
  
  It follows that $\{b_\nu \mid \nu \in {^{\omega_1}}2\}$ is a collection of 
  $\omega_2$-many elements of $S_{\leq \xi}$. This contradicts the fact that $S$ 
  is an $\omega_2$-tree, thus completing the verification of clause (4).
  
  We finally verify clause (5) in the statement of the theorem. To this end, 
  fix a closed set $E \subseteq {^{\omega_2}}\omega_2$ in $V[G \times H]$. 
  Let $\Sigma = \{\sigma \in {^{<\omega_2}}\omega_2 \mid |E \cap N_\sigma| 
  \leq \omega_2\}$. Equivalently, by the chain condition of $\P \times \Q$, 
  $\Sigma$ is the set of all $\sigma \in 
  {^{<\omega_2}}\omega_2$ for which there exists $\delta < \kappa$ such 
  that $E \cap N_\sigma \subseteq V[G_\delta \times H_\delta]$. 
  Let $X = \bigcup\{E \cap N_\sigma \mid \sigma \in \Sigma\}$. Then 
  $|X| \leq \omega_2$ and, since $E$ is closed and $\bigcup\{N_\sigma \mid 
  \sigma \in \Sigma\}$ is open, it follows that $E \setminus X$ is closed. 
  It remains to show that, for all $x_0 \in E \setminus X$, Player II has a 
  winning strategy in $G_{\omega_2}(E \setminus X, x_0, \omega_1)$. 
  To this end, fix an arbitrary $x_0 \in E \setminus X$.
  
  Recall that $T(E) = \{x \restriction \eta \mid x \in E \text{ and } \eta < 
  \omega_2\}$ is a subtree of ${^{<\omega_2}}\omega_2$; since $E$ is 
  closed, it follows that $E = [T(E)]$. Moreover, $T(E) \setminus \Sigma$ 
  is a subtree of ${^{<\omega_2}}\omega_2$, and $E \setminus X = 
  [T(E) \setminus \Sigma]$. In $V$, let 
  $\dot{E}$ be a $\P \times \Q$-name for $E$ and $\dot{X}$ be a $\P \times 
  \Q$-name for $X$. By the chain condition of 
  $\P \times \Q$, we can find a nonzero cardinal $\gamma < \kappa$ such that 
  \begin{itemize}
    \item $x_0, T(E), \Sigma, X \in V[G_\gamma \times H_\gamma]$;
    \item interpreting $\dot{E}$ in $V[G_\gamma \times H_\gamma]$ as an 
    $\bb{R}_\gamma$-name, the empty condition in $\bb{R}_\gamma$ forces all 
    of the following statements:
    \begin{itemize}
      \item $\dot{E} = [T(E)]$;
      \item $\dot{E} \setminus X = [T(E) \setminus \Sigma]$;
      \item $\forall \sigma \in T(E) \setminus \Sigma ~ (|\dot{E} 
      \cap N_\sigma| = \kappa)$.
    \end{itemize}
  \end{itemize}
  To describe a winning strategy for Player II in $G_{\omega_2}(E 
  \setminus X, x_0, \omega_1)$, we look more carefully at the quotient 
  forcing $\bb{R}_\gamma$ in $V[G_\gamma \times H_\gamma]$. Note that 
  $T$, the subtree of ${^{<\omega_2}}\omega_2$ added by $H$, is in 
  $V[G_\gamma \times H_\gamma]$. It is not hard 
  to show that $\bb{R}_\gamma$ is forcing equivalent to 
  $\bb{S} \times \bb{T}$, where 
  $\bb{S} = \Coll(\omega_2, [\gamma, \kappa))$ and 
  $\bb{T}$ is the set of all pairs $(\eta, f)$ such that
  $\eta < \omega_2$ and $f$ is a partial function of size $\omega_1$ 
  from $\kappa \setminus \gamma$ to $T \cap {^\eta}\omega_2$. Given 
  two elements $(\eta,f), (\xi,g) \in \bb{T}$, we have 
  $(\xi,g) \leq_{\bb{T}} (\eta,f)$ if and only if 
  $\xi \geq \eta$, $\dom(g) \supseteq \dom(f)$, and, for all $\alpha \in 
  \dom(f)$, we have $g(\alpha) \sqsupseteq f(\alpha)$.
  
  Recall that elements of $\Coll(\omega_2, [\gamma,\kappa))$ are all 
  functions $s$ such that $\dom(s)$ is a subset of $\mathrm{Card} \cap 
  [\gamma, \kappa)$ of size $\leq \omega_1$ and, for each $\nu \in \dom(s)$, 
  $s(\nu)$ is a partial function from $\omega_2$ to $\nu$ of size 
  $\leq \omega_1$. Given $I \subseteq \kappa \setminus \gamma$, let 
  $\bb{S}_I$ be the set of all $s \in \bb{S}$ such that 
  $\dom(s) \subseteq I$, and let $\bb{T}_I$ be the set of all 
  $(\eta,f) \in \bb{T}$ such that $\dom(f) \subseteq I$. 
  
  Note that, if $I_0, I_1 \subseteq \kappa \setminus \gamma$ and 
  $\pi: I_0 \ra I_1$ is a bijection, then 
  $\bb{T}_{I_0} \cong \bb{T}_{I_1}$ via the map $(\eta, f) \mapsto 
  (\eta, \hat{\pi}(f))$, where $\dom(\hat{\pi}(f)) = \pi[\dom(f)]$ and, 
  for all $\alpha \in \dom(f)$, we let $\hat{\pi}(f)(\alpha) = 
  f(\alpha)$. If, moreover, $\pi$ is a bijection between 
  $\mathrm{Card} \cap I_0$ and $\mathrm{Card} \cap I_1$ such that 
  $\pi(\nu) \leq \nu$ for all $\nu \in \mathrm{Card} \cap I_0$, then 
  $\pi$ introduces a projection $\tilde{\pi} : \bb{S}_{I_0} \ra 
  \bb{S}_{I_1}$ defined as follows: for each $s \in \bb{S}_{I_0}$, 
  let $\tilde{\pi}(s) \in \bb{S}_{I_1}$ be such that
  \begin{itemize}
    \item $\dom(\tilde{\pi}(s)) = \pi[\dom(s)]$;
    \item for all $\nu \in \dom(s)$, we have $\dom(\tilde{\pi}(s)(\nu)) 
    = \dom(s(\nu))$;
    \item for all $\nu \in \dom(s)$ and all $\eta \in \dom(s(\nu))$, 
    \begin{align*}
      \tilde{\pi}(s)(\pi(\nu))(\eta) = 
      \begin{cases}
        s(\nu)(\eta) & \text{if } s(\nu)(\eta) < \pi(\nu) \\ 
        0 & \text{otherwise.}
      \end{cases}
    \end{align*}
  \end{itemize}
  Suppose in particular that $\delta \in (\gamma, \kappa)$ is a cardinal, 
  say $\delta = \gamma^{+\varepsilon}$ for some $\varepsilon < \kappa$. 
  Then we can partition the interval $[\gamma, \kappa)$ into intervals 
  $\{I_\zeta \mid \zeta < \kappa\}$, where $I_\zeta = [\gamma^{+\varepsilon 
  \cdot \zeta}, \gamma^{+\varepsilon \cdot (\zeta + 1)})$ for all 
  $\zeta < \kappa$. The preceding discussion then implies that
  \begin{itemize}
    \item $\bb{R}_\gamma$ is forcing equivalent to the ${<}\omega_2$-support 
    product of $\langle \bb{S}_{I_\zeta} \times \bb{T}_{I_\zeta} \mid 
    \zeta < \kappa \rangle$;
    \item for all $\zeta < \kappa$, there is a projection from 
    $\bb{S}_{I_\zeta} \times \bb{T}_{I_\zeta}$ to $\bb{S}_{I_0} \times 
    \bb{T}_{I_0}$.
  \end{itemize}
  In particular, forcing with $\bb{R}_\gamma$ adds $\kappa$-many 
  pairwise mutually $(\bb{S}_{I_0} \times \bb{T}_{I_0})$-generic filters.
  
  We are now ready to describe Player II's winning strategy in 
  $G_{\omega_2}(E \setminus X, x_0, \omega_1)$. Some aspects of 
  the strategy can only be precisely specified after Player I makes their 
  first move, so suppose that, in 
  round 1 of the game, Player I plays the ordinal $\alpha_1 < \omega_2$.
  Let $\sigma_1 = x_0 \restriction \alpha_1$. Since 
  $\sigma_1 \in T(E) \setminus \Sigma$, moving to $V[G_\gamma \times 
  H_\gamma]$, we can find a nice $\bb{R}_\gamma$-name $\dot{y}$ that is forced 
  to be an element of $\dot{E} \cap N_{\sigma_1} \setminus V[G_\gamma 
  \times H_\gamma]$.
  By the chain condition of $\bb{R}_\gamma$, we can find a limit ordinal 
  $\varepsilon$ such that, letting $I_0 = [\gamma, \gamma^{+\varepsilon})$, 
  $\dot{y}$ is an $\bb{S}_{I_0} \times \bb{T}_{I_0}$-name. Let 
  $\bb{U}$ denote the ${<}\omega_2$-support product of $\kappa$-many copies 
  of $\bb{S}_{I_0} \times \bb{T}_{I_0}$; for $\zeta < \kappa$, let $\bb{U}(\zeta)$ 
  denote its $\zeta^{\mathrm{th}}$ factor. By the preceding discussion, there 
  is a projection from $\bb{R}_\gamma$ to $\bb{U}$, so we can view 
  $\bb{R}_\gamma$ as a two-step iteration of the form $\bb{U} \ast \dot{\bb{W}}$.
  For each $\zeta < \kappa$, let $\dot{y}_\zeta$ denote the name for the 
  interpretation of $\dot{y}$ with respect to the generic filter for 
  $\bb{U}(\zeta)$. By the product lemma, for all 
  $\zeta_0 < \zeta_1 < \kappa$, $\dot{y}_{\zeta_0}$ and $\dot{y}_{\zeta_1}$ 
  are forced to be distinct elements of $N_{\sigma_1} \cap (\dot{E} \setminus 
  X)$.
  
  Let $K \ast L$ be $\bb{U} \ast \dot{\bb{W}}$-generic over 
  $V[G_\gamma \times H_\gamma]$ such that $V[G \times H] = 
  V[G_\gamma \times H_\gamma][K \ast L]$. For $\zeta < \kappa$, let 
  $K_\zeta$ denote the $\bb{S}_{I_0} \times \bb{T}_{I_0}$-generic filter 
  induced by the $\zeta^{\mathrm{th}}$ factor of $\bb{U}$.
  In the course of the run of $G_{\omega_2}(E \setminus X, x_0, \omega_1)$ 
  in which Player II plays according to the strategy we will specify here, 
  producing the play, $\langle (\alpha_i, x_i) \mid 1 \leq i < \omega_1 \rangle$, 
  Player II will construct a strictly increasing sequence 
  $\langle \zeta_i \mid 1 \leq i < \omega_1 \rangle$ of ordinals below $\kappa$ 
  and a decreasing sequence $\langle p_i \mid 1 \leq i < \omega_1 \rangle$ of 
  conditions in $\bb{S}_{I_0} \times \bb{T}_{I_0}$ satisfying the following 
  requirements for all $1 \leq i < \omega_1$:
  \begin{enumerate}
    \item $x_i$ is the interpretation of $\dot{y}_{\zeta_i}$ in 
    $V[G_\gamma \times H_\gamma][K \ast L]$;
    \item $p_{i+1} \in K_{\zeta_i}$;
    \item $p_i$ decides the value of $\dot{y} \restriction \alpha_i$.
  \end{enumerate}
  
  We first describe Player II's first move. Let $\zeta_1 = 0$, let 
  their play $x_1$ be the interpretation of $\dot{y}_0$ in 
  $V[G_\gamma \times H_\gamma][K \ast L]$, and, let 
  $p_1 = 1_{\bb{S}_{I_0} \times \bb{T}_{I_0}}$. 
  
  Now suppose that $1 \leq j < \omega_1$ and, in our run of 
  $G_{\omega_2}(E \setminus X, x_0, \omega_1)$, the players have played 
  $\langle \alpha_i \mid i \leq j \rangle$ and $\langle x_i \mid i < j 
  \rangle$, with Player II playing according to their winning strategy 
  and also specifying $\langle (p_i, \zeta_i) \mid i < j \rangle$. 
  
  Suppose first that $j$ is a successor ordinal, say $j = j_0 + 1$. 
  First, choose $p_j \leq p_{j_0}$ such that $p_j \in K_{\zeta_{j_0}}$ 
  and $p_j$ decides the value of $\dot{y} \restriction \alpha_j$. Note 
  that, in $V[G_\gamma \times H_\gamma]$, the set
  \[
    \{f \in \bb{U} \mid \exists \zeta \in (\zeta_{j_0}, \kappa) ~ 
    (f(\zeta) \leq p_j)\}
  \]  
  is a dense open subset of $\bb{U}$. Therefore, by genericity, 
  Player II can choose $\zeta_j \in (\zeta_{j_0}, \kappa)$ such that 
  $p_j \in K_{\zeta_j}$. Then Player II plays the interpretation of 
  $\dot{y}_{\zeta_j}$ as $x_j$.
  
  Next, suppose that $j$ is a limit ordinal. In this case, we have 
  $\alpha_j = \sup\{\alpha_i \mid i < j\}$. Since $\bb{S}_{I_0} 
  \times \bb{T}_{I_0}$ is countably closed, we can fix a lower bound 
  $p_j$ for $\langle p_i \mid i < j \rangle$. As in the previous paragraph, 
  by genericity we can choose a $\zeta_j < \kappa$ such that 
  $p_j \in K_{\zeta_j}$ and $\zeta_j > \zeta_i$ for all $i < j$. 
  Then Player II plays the interpretation of $\dot{y}_{\zeta_j}$ 
  as $x_j$.
  
  It is readily verified that this describes a winning strategy for 
  Player II in $G_{\omega_2}(E \setminus X, x_0, \omega_1)$, thus completing 
  the proof that $E \setminus X$ is $\omega_1$-perfect in 
  $V[G \times H]$.
\end{proof}

\section{Full trees} \label{section: full}

Recall that, given a tree $T$ and an ordinal $\beta < \height(T)$, we let 
$[T_{<\beta}]$ denote the set of cofinal branches through $T_{<\beta}$, 
i.e., the set of all elements of $b \in \prod_{\alpha < \beta} T_\alpha$ 
such that the range of $b$ is linearly ordered by $<_T$. Note that we can 
identify each element of $T_\beta$ with an element of $[T_{<\beta}]$, namely 
the branch given by its predecessors; if $T$ is normal, then this 
identification is injective. With a slight abuse of notation, then, we 
let $[T_{<\beta}] \setminus T_\beta$ denote the set of \emph{vanishing} 
branches through $T$ of length $\beta$, i.e., the set of $b \in 
[T_{<\beta}]$ such that the range of $b$ does not have an upper bound in 
$T_\beta$.

\begin{definition}
  A tree $T$ is \emph{full} if, for every limit ordinal $\beta < 
  \height(T)$, there is at most one vanishing branch through $T$ of 
  length $\beta$.
\end{definition}

There has been some research in recent years into the existence of full 
$\kappa$-Suslin trees. For example, in \cite{shelah_full}, Shelah 
establishes the consistency of the existence of full $\kappa$-Suslin trees 
for a Mahlo cardinal $\kappa$, answering a question of Jech (cf.\ \cite{miller_list}). 
In \cite{rinot_full}, Rinot, Yadai, and You 
prove the consistency of the existence of full $\kappa$-Suslin trees at 
accessible cardinals $\kappa$; for example, they can consistently exist 
at all successors of regular uncountable cardinals. 

Here, we are interested in full, splitting trees that may contain some cofinal branches. 
We begin by investigating full trees of height and size $\omega_1$. We first 
show that, under $\diamondsuit$, we have considerable control over the number of 
cofinal branches through such trees, establishing Theorem C(1).

\begin{theorem} \label{thm: diamond}
  Suppose that $\diamondsuit$ holds. Then, for every cardinal $\nu \in 
  \omega \cup \{\omega, \omega_1, 2^{\omega_1}\}$, there is a normal, full, 
  splitting tree $T \subseteq {^{<\omega_1}}\omega_1$ such that 
  $|[T]| = \nu$.
\end{theorem}

\begin{proof}
  If $\nu = 2^{\omega_1}$, we can simply let $T = {^{<\omega_1}}\omega_1$. 
  For concreteness, we will prove the theorem in case $\nu = \omega_1$; the 
  proof for smaller values of $\nu$ is similar but easier.
  
  We will construct a tree $T$ by recursively specifying its $\alpha^{\mathrm{th}}$ 
  level $T_\alpha = T \cap {^{\alpha}}\omega_1$ for $\alpha < \omega_1$. 
  When specifying $T_\alpha$, we will also specify an injective function 
  $f_\alpha : \alpha \rightarrow T_\alpha$ with the requirement that, 
  for all $\eta < \alpha < \beta < \omega_1$, we have $f_\alpha(\eta) 
  \subseteq f_\beta(\eta)$. The idea is that, at the end of the construction, 
  for each $\eta < \omega_1$, $b_\eta := \bigcup \{f_\alpha(\eta) \mid \eta < 
  \alpha < \omega_1\}$ will be a cofinal branch through $T$, and we will arrange 
  so that \emph{every} cofinal branch through $T$ is equal to $b_\eta$ for some 
  $\eta < \omega_1$.
  
  Since $\diamondsuit$ holds, we can fix a sequence $\langle a_\alpha \mid \alpha < 
  \omega_1 \rangle$ such that
  \begin{itemize}
    \item for all $\alpha < \omega_1$, $a_\alpha:\alpha \ra \alpha$;
    \item for all $b:\omega_1 \ra \omega_1$, there are stationarily many 
    $\alpha < \omega_1$ for which $b \restriction \alpha = a_\alpha$.
  \end{itemize}
  We now describe the construction of $T$. We must set $T_0 = f_0 = \emptyset$. 
  Given $T_\alpha$ and $f_\alpha$, first form $T_{\alpha + 1}$ by splitting 
  maximally, i.e., $T_{\alpha + 1} = \{\sigma ^\frown i \mid 
  \sigma \in T_\alpha, ~ i < \omega_1\}$. For each $\eta < \alpha$, let 
  $f_{\alpha+1}(\eta) = f_\alpha(\eta)^\frown 0$, and let 
  $f_{\alpha+1}(\alpha)$ be any element of $T_{\alpha+1}$ not equal to 
  $f_{\alpha+1}(\eta)$ for some $\eta < \alpha$.
  
  Now suppose that $\beta < \omega_1$ is a limit 
  ordinal and $T_{<\beta}$, together with $\langle f_\alpha \mid \alpha < 
  \beta\}$, has been specified. For each $\eta < \beta$, 
  let $b^\beta_\eta := \bigcup \{f_\alpha(\eta) \mid \eta < \alpha < \beta\}$. 
  By construction, each $b^\beta_\eta$ is in $[T_{<\beta}]$. Now consider the function $a_\beta$ given by our 
  $\diamondsuit$-sequence. There are two cases to consider.
  
  \textbf{Case 1:} $a_\beta$ is a branch through $T_{<\beta}$ and, for all 
  $\eta < \beta$, we have $a_\beta \neq b^\beta_\eta$. In this case, 
  let $T_\beta = [T_{<\beta}] \setminus \{a_\beta\}$. For all $\eta < \beta$, 
  let $f_\beta(\eta) = b^\beta_\eta$.
  
  \textbf{Case 2:} otherwise. In this case, let $T_\beta = [T_{<\beta}]$ and, again, 
  let $f_\beta(\eta) = b^\beta_\eta$ for all $\eta < \beta$.
  
  This completes the construction of $T$ and $\langle f_\alpha \mid \alpha < \omega_1 
  \rangle$. It is easily verified that $T$ is a normal, full, splitting tree and 
  $\langle b_\eta \mid \eta < \omega_1 \rangle$ is an injective sequence of cofinal 
  branches through $T$. It remains to show that every cofinal branch through $T$ is 
  equal to $b_\eta$ for some $\eta < \omega_1$.
  
  To this end, fix $b \in {^{\omega_1}}\omega_1$ such that, for all $\eta < \omega_1$, 
  we have $b \neq b_\eta$. We will find $\beta < \omega_1$ such that 
  $b \restriction \beta \notin T$. First, let $C$ be the set of limit 
  ordinals $\beta < \omega_1$ such that
  \begin{itemize}
    \item $b[\beta] \subseteq \beta$;
    \item for all $\eta < \beta$, we have $b_\eta \restriction \beta \neq b 
    \restriction \beta$.
  \end{itemize}
  Then $C$ is a club in $\omega_1$, so we can fix $\beta \in C$ such that 
  $b \restriction \beta = a_\beta$. Now consider stage $\beta$ of the construction 
  of $T$. If $a_\beta \notin [T_{<\beta}]$, then we are done, since 
  this immediately implies that $b \restriction \beta \notin T$. If $a_\beta$ 
  \emph{is} a branch through $T_{<\beta}$, then, for all $\eta < \beta$, we have 
  $a_\beta \neq b_\eta \restriction \beta = b^\beta_\eta$. We are therefore in 
  Case 1 of the construction, and hence we have $b \restriction \beta = 
  a_\beta \notin T$, as desired.
\end{proof}

$\CH$ is a necessary condition for the existence of full, splitting trees of 
height $\omega_1$ with few cofinal branches as it is easily seen that, if $\CH$ fails, 
then every full, splitting tree of height $\omega_1$ has at least 
$2^{\aleph_0}$-many cofinal branches. However, we now show that $\CH$ is not 
a sufficient condition for this, i.e., the hypothesis of $\diamondsuit$ in 
Theorem \ref{thm: diamond} cannot be weakened to $\CH$. 
Recall that a forcing notion $\P$ is 
\emph{totally proper} if it is proper and adds no reals. We will need an iterable 
strengthening of total properness known as \emph{complete properness}, introduced 
in \cite{moore_minimal} (cf.\ also \cite{eisworth_moore_milovich}). 
Let us recall the relevant definitions, beginning with $\alpha$-properness.

\begin{definition}
  Suppose that $\Q$ is a forcing notion and $\theta$ is a sufficiently large 
  regular cardinal. A countable elementary submodel $M$ of $H(\theta)$ is said to 
  be \emph{suitable for $\Q$} if $\Q, \power(\Q) \in M$. 
  
  If $M$ is suitable for $\Q$ and $q \in \Q$, then we say that $q$ is $(M,\Q)$-generic if, for every 
  dense open subset $D$ of $\Q$ with $D \in M$ and every $r \leq_{\Q} q$, $r$ is 
  compatible with an element of $D \cap M$. If $G \subseteq \Q \cap M$ is a filter, then we say that $G$ is 
  $(M,\Q)$-generic if $G \cap D \neq \emptyset$ for every dense open subset $D$ of $\Q$ 
  with $D \in M$. A condition $q \in \Q$ is \emph{totally $(M,\Q)$-generic} if the set 
  $\{p \in M \cap \Q \mid q \leq p\}$ is an $(M,\Q)$-generic filter.
\end{definition}

\begin{definition}
  Suppose that $\Q$ is a forcing notion
  and $\alpha < \omega_1$. A \emph{suitable $\alpha$-tower for $\Q$} is a 
  continuous, $\subseteq$-increasing sequence $\langle M_\eta \mid \eta < \alpha \rangle$ 
  of countable elementary submodels of $H(\theta)$ for some sufficiently large regular 
  cardinal $\theta$ such that 
  \begin{itemize}
    \item $M_0$ is suitable for $\Q$;
    \item for all $\xi$ with $\xi + 1 < \alpha$, we have $\langle M_\eta \mid \eta \leq \xi 
    \rangle \in M_{\xi + 1}$.
  \end{itemize}
  We say that $\Q$ is \emph{$\alpha$-proper} if, for every suitable $\alpha$-tower 
  $\langle M_\eta \mid \eta < \alpha \rangle$ for $\Q$ and every $q_0 \in M_0 \cap \Q$, there is 
  $q \leq_{\Q} q_0$ such that $q$ is $(M_\eta,\Q)$-generic for all $\eta < \alpha$.
  We say that $\Q$ is \emph{totally $\alpha$-proper} if we can additionally require that 
  $q$ is totally $(M_\eta,\Q)$-generic for all $\eta < \alpha$.
\end{definition}

We now turn to Moore's notion of \emph{complete properness}.

\begin{definition}
  If $M$ and $N$ are sets, then the notation $M \rightarrow N$ denotes the existence of 
  an elementary embedding $\varepsilon : (M, \in) \rightarrow (N, \in)$ such 
  that $\varepsilon \in N$ and $N \models ``M \text{ is countable}"$, i.e., 
  $N$ contains an injection of $M$ into $\omega$. If $X \in M$ and 
  $M \rightarrow N$, as witnessed by $\varepsilon$, then we will let $X^N$ denote 
  $\varepsilon(X)$. If $X \subseteq M$ is not an element of $M$, then $X^N$ 
  denotes the pointwise image $\varepsilon[X]$.
\end{definition}

\begin{definition}
  Suppose that $\Q$ is a forcing notion, $M$ is suitable for $\Q$, and $M \rightarrow N$. 
  Then a filter $G \subseteq \Q \cap M$ is \emph{$\overrightarrow{MN}$-prebounded} if 
  whenever $N \rightarrow P$ and $G \in P$, then $p \models ``G^P \text{ has a lower 
  bound in } \Q^P$, where $G^P$ is defined via the composition $M \rightarrow 
  N \rightarrow P$.
\end{definition}

\begin{definition}
  Suppose that $\Q$ is a forcing notion. We say that $\Q$ is \emph{completely proper} if 
  whenever $M$ is suitable for $\Q$, $q \in \Q \cap M$, and $M \rightarrow N_i$ for $i < 2$, there is 
  an $(M,\Q)$-generic filter $G \subseteq \Q \cap M$ that is $\overrightarrow{MN_i}$-prebounded 
  for all $i < 2$ with $q \in G$.
\end{definition}

Moore proved in \cite[Lemma 4.11]{moore_minimal} that completely proper forcings are $2$-complete with 
respect to some completeness system $\mathbb{D}$ (cf.\ \cite[\S V.5]{proper_and_improper}). 
Combined with Shelah's iteration theorem concerning such forcings 
(\cite[Theorem VIII.4.5]{proper_and_improper}), this yields the following theorem 
(cf.\ \cite[Main Theorem]{eisworth_moore_milovich}).

\begin{theorem} \label{thm: iteration_thm}
  Suppose that $\langle \P_\eta, \dot{\Q}_\xi \mid \eta \leq \delta, \ \xi < \delta 
  \rangle$ is a countable support iteration of totally proper forcing notions such that, for all 
  $\eta < \delta$, we have
  \[
    \Vdash_{\P_\eta}``\dot{\Q}_\eta \text{ is completely proper and totally $\alpha$-proper for all } 
    \alpha < \omega_1".
  \]
  Then $\P_\delta$ is totally proper.
\end{theorem}

We are now ready to prove that the assumption of $\diamondsuit$ cannot be 
weakened to $\CH$ in Theorem \ref{thm: diamond}. In fact, we will prove that 
it is consistent with $\CH$ that every full, splitting tree of height 
$\omega_1$ contains a copy of ${^{<\omega_1}}2$, thus establishing 
Theorem C(2).

\begin{theorem}
  It is consistent that $\CH$ holds and every full, splitting tree of height 
  $\omega_1$ contains a copy of ${^{<\omega_1}}2$.
\end{theorem}

\begin{proof}
  We first note that it is enough to consider trees of height and \emph{size} 
  $\omega_1$. To see this, suppose that $\CH$ holds and $T$ is a tree of height 
  $\omega_1$. Let $\theta$ be a sufficiently large regular cardinal and let 
  $M \prec H(\theta)$ be such that
  \begin{itemize}
    \item $|M| = \omega_1$;
    \item ${^{\omega}}M \subseteq M$;
    \item $T \in M$.
  \end{itemize}
  Then $T \cap M$ is a full, splitting subtree of $T$ of size $\omega_1$; 
  if $T \cap M$ contains a copy of ${^{<\omega_1}}2$ then, \emph{a fortiori}, 
  so does $T$.
  
  Suppose that $\GCH$ holds.
  Fix for now a full, splitting tree $T$ of height and size $\omega_1$. Without 
  loss of generality, we can assume that $T$ is a subtree of 
  ${^{<\omega_1}}\omega_1$. We will describe a totally proper 
  forcing $\P(T)$ of size $\omega_1$ that adds a copy of ${^{<\omega_1}}2$ to $T$ 
  and then show that this forcing can be iterated without adding reals.
  
  Given a subtree $S \subseteq {^{<\omega_1}}\omega_1$, 
  let $\partial S$ be the set of $\sigma \in {^{<\omega_1}}\omega_1$ such that
  \begin{itemize}
    \item $\dom(\sigma)$ is a limit ordinal;
    \item $\sigma \notin S$;
    \item for all $\eta < \dom(\sigma)$, $\sigma \restriction \eta \in S$.
  \end{itemize}
  Conditions of $\P(T)$ are all pairs of the form $p = (S_p, f_p)$ such that
  \begin{enumerate}
    \item $S_p$ is a countable subtree of ${^{<\omega_1}}2$;
    \item $f_p : S_p \ra T$ is an isomorphic embedding;
    \item \label{closure_condition} 
    for all $\sigma \in \partial S_p$, we have $\bigcup \{f_p(\sigma 
    \restriction \eta) \mid \eta < \dom(\sigma)\} \in T$.
  \end{enumerate}
  If $p,q \in \P(T)$, then $q \leq p$ if and only if $S_q \supseteq S_p$ and 
  $f_q \restriction S_p = f_p$.
  
  Since $\CH$ holds, $\P(T)$ is of size $\omega_1$. Given $p = (S_p, f_p) \in 
  \P(T)$, let $T_p$ denote the $\leq_T$-downward closure of $f_p``S_p$. 
  Condition (\ref{closure_condition}) above can then be expressed as the assertion 
  that $\partial T_p \subseteq T$.  
  
  Let $\mc L$ be the set of $\beta \in \lim(\omega_1)$ such that 
  $[T_{<\beta}] \setminus T_\beta \neq \emptyset$ and, for each $\beta \in 
  \mc L$, let $b_\beta$ be the unique element of $[T_{<\beta}] \setminus 
  T_\beta$. When constructing elements of $\P(T)$, we need to be careful to 
  avoid the branches $b_\beta$ for $\beta \in \mc L$. The following lemmas show 
  that this can be done.
  
  \begin{lemma} \label{lemma: pruning}
    Suppose that $\beta \in \lim(\omega_1)$ and $b \in {^\beta}\omega_1$. 
    Suppose moreover that
    \begin{itemize}
      \item $p = (S_p, f_p) \in \P(T)$;
      \item $\sigma \in {^{<\omega_1}}2$ is such that 
      $\sigma, \sigma^\frown 0, \sigma^\frown 1 \in S_p$;
      \item $f_p(\sigma) \sqsubseteq b$ but, for all $i < 2$, we have 
      $f_p(\sigma^\frown i ) \perp b$.                 
    \end{itemize}
    Let $\xi < \beta$ be such that $f_p(\sigma^\frown i ) \perp 
    (b \restriction \xi)$ for all $i < 2$.
    Then, for every $q \leq p$ and every $\tau \in S_q$, it is not the case 
    that $f_q(\tau) \sqsupseteq b \restriction \xi$.
  \end{lemma}
  
  \begin{proof}
    Note that $f_p(\sigma) = b \restriction \eta$ for some $\eta < \xi$.
    Fix $q \leq p$ and $\tau \in S_q$. There are three cases to consider.
    
    First, if $\tau \sqsubseteq \sigma$, then $f_q(\tau) = f_p(\tau) 
    \sqsubseteq b \restriction \eta$, so $f_q(\tau) \not\sqsupseteq b \restriction \xi$.
    
    Second, if $\sigma \sqsubseteq \tau$, then there is $i < 2$ such that 
    $\sigma^\frown \langle i \rangle \sqsubseteq \tau$, and hence 
    $f_p(\sigma^\frown \langle i \rangle) \sqsubseteq f_q(\tau)$. 
    Since $f_p(\sigma^\frown \langle i \rangle) \perp b \restriction \xi$, it follows that 
    $f_q(\tau) \perp b \restriction \xi$.
    
    The remaining case is that in $\tau \perp \sigma$. In this case, $f_q(\tau) \perp 
    f_p(\sigma) = b \restriction \eta$.
  \end{proof}
  
  \begin{lemma} \label{lemma: split}
    Suppose that $M$ is suitable for $\P(T)$, $p \in M \cap \P(T)$,
    $\beta = M \cap \omega_1$, and $\beta \in \mc L$.
    Let    
    \[
      \xi := \sup\{\eta < \beta \mid \exists \sigma \in S_p ~ [
      f_p(\sigma) = b_\beta \restriction \eta]\},
    \] 
    and assume that $b_\beta \restriction \xi \in M$. 
    Then there is $q \leq p$ with $q \in M$ and $\sigma \in S_q$ such that
    \begin{itemize}
      \item $f_q(\sigma) \sqsubseteq b_\beta$;
      \item for all $i < 2$, $\sigma^\frown i  \in S_q$ 
      and $f_q(\sigma^\frown i) \perp b_\beta$.
    \end{itemize}
  \end{lemma}
  
  \begin{proof}
    We may assume that $p$ itself does not satisfy the conclusion of the theorem.
    Since $p \in \P(T)$ and $b_\beta \notin T$, we must have $\xi < \beta$. 
    Assume first that there is $\sigma \in S_p$ such that 
    $f_p(\sigma) = b_\beta \restriction \xi$. By the definition of $\xi$, 
    if $i < 2$ and $\sigma^\frown i \in S_p$, then $f_p(\sigma^\frown i) \perp 
    b_\beta$. We will define a condition 
    $q \leq p$ in $M$ such that $S_q = S_p \cup \{\sigma^\frown 0, 
    \sigma^\frown 1 \}$. It suffices to define $f_q(\sigma^\frown 
    i )$ for $i < 2$. Using the fact that $T$ is splitting, we know 
    that there are two distinct nodes $t_0, t_1 \in T_{\xi + 2} \cap M$ 
    such that, for all $i < 2$, we have $b_\beta \restriction \xi 
    \sqsubseteq t_i$ but $b_\beta \restriction (\xi + 2) \perp t_i$. 
    For $i < 2$, if $\sigma^\frown i \notin S_p$, set $f_q(\sigma^\frown i ) = t_i$. 
    It is readily verified that $q \leq p$ is as desired.     
    
    Now suppose that there is no $\sigma \in S_p$ such that $f_p(\sigma) = 
    b_\beta \restriction \xi$. By our assumptions about $p$, there must be an increasing sequence 
    of ordinals $\langle \eta_n \mid n < \omega \rangle$ and a $\sqsubseteq$-increasing 
    sequence $\langle \sigma_n \mid n < \omega \rangle$ from $S_p$ such that
    \begin{itemize}
      \item $\sup\{\eta_n \mid n < \omega\} = \xi$;
      \item $\bigcup \{\sigma_n \mid n < \omega\} \notin S_p$; and
      \item for all $n < \omega$, $f_p(\sigma_n) = b_\beta \restriction 
      \eta_n$.
    \end{itemize}
    Let $\sigma := \bigcup\{\sigma_n \mid n < \omega\}$. Then $\sigma \in M$, 
    since it is definable in $M$ as the unique element $\tau$ of 
    $\partial S_p$ such that $\bigcup \{f_p(\tau \restriction \eta) \mid 
    \eta < \dom(\tau)\} = b_\beta \restriction \xi$. We will define a condition 
    $q \leq p$ in $M$ such that $S_q = S_p \cup \{\sigma, \sigma^\frown 
    0, \sigma^\frown 1 \}$. Work entirely in $M$. 
    First, set $f_q(\sigma) = b_\beta \restriction \xi$. Then define 
    $f_q(\sigma^\frown i )$ for $i < 2$ exactly as in the previous case.
    It is again readily verified that $q \leq p$ is as desired.
  \end{proof}
  
  \begin{lemma}
    $\P(T)$ is totally $\alpha$-proper for all $\alpha < \omega_1$.
  \end{lemma}
  
  \begin{proof}
    The proof is by induction on $\alpha$. We will in fact prove the following stronger 
    statement, which will be used in the inductive step:
    \begin{quote}
      For every suitable $(\alpha+1)$-tower $\langle M_\eta \mid \eta \leq \alpha \rangle$ 
      and every $p_0 \in M_0 \cap \P(T)$, there is $p \leq p_0$ such that 
      \begin{itemize}
        \item $p$ is totally $(M_\eta, \P(T))$-generic for all $\eta \leq \alpha$;
        \item $S_p, f_p \subseteq M_\alpha$.
      \end{itemize}
    \end{quote}
    
    Fix $\beta < \omega_1$ and suppose that we have established the induction hypothesis 
    for all $\alpha < \beta$. Let $\langle M_\eta \mid \eta \leq \beta \rangle$ be a suitable 
    $(\beta+1)$-tower for $\P(T)$, and fix $p_0 \in M_0 \cap \P(T)$. We will assume that 
    $\beta$ is a limit ordinal, and hence $M_\beta = \bigcup_{\alpha < \beta} M_\alpha$; 
    the case in which $\beta$ is a successor is easier and proven similarly. Note also that 
    if a condition $p$ is totally $(M_\alpha,\P(T))$-generic for all $\alpha < \beta$, then 
    it is also totally $(M_\beta, \P(T))$-generic.
    
    Let $\delta := M_\beta \cap \omega_1$. 
    Assume that $\delta \in \mc{L}$; the case in which $\delta \notin \mc{L}$ is similar and 
    easier. There are now two cases to consider; we will deal with them in parallel:
    \begin{itemize}
      \item \textbf{Case 1:} There is $\xi^* < \delta$ such that $b_\delta \restriction 
      \xi^* \notin M_\beta$;
      \item \textbf{Case 2:} For every $\xi < \delta$, $b_\delta \restriction \xi
      \in M_\beta$.
    \end{itemize}      
    
    If we are in Case 2, let 
    \[
      \xi := \sup\{\eta < \beta \mid \exists \sigma \in S_{p_0} ~ [
      f_{p_0}(\sigma) = b_\beta \restriction \eta]\}.
    \]
    There are now two subcases to consider:
    \begin{itemize}
      \item \textbf{Case 2a: } $b_\delta \restriction \xi \in M_0$;
      \item \textbf{Case 2b: } $b_\delta \restriction \xi \notin M_0$. 
    \end{itemize}      
    If we are in Case 2a, begin by applying Lemma \ref{lemma: split} to find a $p_0' \leq p_0$ and 
    a $\sigma_0 \in S_{p_0'}$ such that
    \begin{itemize}
      \item $p_0' \in M_0$;
      \item $f_{p_0'}(\sigma_0) \sqsubseteq b_\delta$;
      \item for all $i < 2$, $\sigma_0{}^\frown i \in S_{p_0'}$ 
      and $f_{p_0'}(\sigma_0{}^\frown i ) \perp b_\delta$.
    \end{itemize}
    If we are in Case 1 or 2b, let $p_0' = p_0$ and leave $\sigma_0$ undefined for now.    
    
    Let $\langle \alpha_n \mid n < \omega \rangle$ be a strictly increasing sequence of ordinals 
    that is cofinal in $\beta$. If we are in Case 2b, additionally choose $\alpha_0$ so that 
    $\alpha_0 + 1$ is the least ordinal $\varepsilon < \beta$ such that $b_\delta 
    \restriction \xi \in M_\varepsilon$ (note that this $\varepsilon$ must be a successor 
    ordinal by the continuity of $\langle M_\eta \mid \eta < \beta \rangle$).
    
    We will build a decreasing sequence $\langle p_n \mid n < 
    \omega \rangle$ of conditions in $\P(T)$ such that, for all $n < \omega$, we have
    \begin{itemize}
      \item $p_{n+1} \in M_{\alpha_n + 1}$;
      \item $p_{n+1}$ is totally $(M_\eta,\P(T))$-generic for all $\eta \leq \alpha_n$.
    \end{itemize}
    We will also arrange so that $\langle p_n \mid n < \omega \rangle$ will have a lower bound 
    \[
      p_\infty = \left(\bigcup_{n < \omega} S_{p_n}, \bigcup_{n < \omega} f_{p_n}\right)
    \] 
    in $\P(T)$. This lower bound will then be the desired condition that is totally 
    $(M_\eta,\P(T))$-generic for all $\eta \leq \beta$. Note that, by construction, we will 
    have $S_{p_\infty} \subseteq M_\beta$.
    
    Let $\langle \gamma_n \mid 0 < n < \omega \rangle$ enumerate $\mc{L} \cap \delta$ (with 
    repetitions, if necessary) in such a way that $\gamma_n \in M_{\alpha_n+1}$ for all 
    $0 < n < \omega$. We will arrange so that, for all $0 < n < \omega$, there is 
    $\sigma_n \in S_{p_n}$ such that 
    \begin{itemize}
      \item $f_{p_n}(\sigma_n) \sqsubseteq b_{\gamma_n}$;
      \item for all $i < 2$, $\sigma_n{}^\frown i \in S_{p_n}$ and 
      $f_{p_n}(\sigma_n{}^\frown i ) \perp b_{\gamma_n}$.
    \end{itemize}
    We now describe the rest of the construction. The first step will be different from the 
    others due to the need to take care of Case 2b. We begin by applying the inductive 
    hypothesis in $M_{\alpha_0 + 1}$ to find $p_1' \leq p_0'$ such that
    \begin{itemize}
      \item $p_1' \in M_{\alpha_0 + 1}$ and $S_{p_1'}, f_{p_1'} \subseteq M_{\alpha_0}$;
      \item $p'_1$ is totally $(M_\eta, \P(T))$-generic for all $\eta \leq \alpha_0$.
    \end{itemize}
    Suppose first that we are in Case 2b. By our choice of $\alpha_0$ and the fact 
    that $S_{p'_1} \subseteq M_{\alpha_0}$, we know that $b_\delta \restriction \xi 
    \notin S_{p'_1}$. Thus, we still have
    \[
      \xi = \sup\{\eta < \beta \mid \exists \sigma \in S_{p'_1} ~ [
      f_{p'_1}(\sigma) = b_\beta \restriction \eta]\}.
    \]
    We can therefore apply Lemma \ref{lemma: split} to find 
    $p^*_1 \leq p'_1$ and $\sigma_0 \in S_{p^*_1}$ such that
    \begin{itemize}
      \item $p^*_1 \in M_{\alpha_0 + 1}$;
      \item $f_{p^*_1}(\sigma_0) \sqsubseteq b_\delta$;
      \item for all $i < 2$, $\sigma_0{}^\frown \langle i \rangle \in S_{p^*_1}$ and 
      $f_{p_1^*}(\sigma_0{}^\frown \langle i \rangle) \perp b_\delta$.
    \end{itemize}
    If we are in Case 1 or 2a, let $p^*_1 = p'_1$; in Case 1, leave $\sigma_0$ permanently undefined.
    Now, regardless of the case we are in, apply Lemma \ref{lemma: split} once again to find 
    $p_1 \leq p^*_1$ and $\sigma_1 \in S_{p_1}$ such that
    \begin{itemize}
      \item $p_1 \in M_{\alpha_0 + 1}$;
      \item $f_{p_1}(\sigma_1) \sqsubseteq b_{\gamma_1}$;
      \item for all $i < 2$, $\sigma_1{}^\frown i \in S_{p_1}$ and 
      $f_{p_1}(\sigma_1{}^\frown  i ) \perp b_{\gamma_1}$.
    \end{itemize}
    The construction is now uniform across all cases. Suppose that $0 < n < \omega$ and we have 
    constructed $p_n$. Apply the inductive hypothesis to the tower 
    $\langle M_\eta \mid \alpha_{n-1} < \eta \leq \alpha_n \rangle$ inside $M_{\alpha_n + 1}$ to find 
    $p'_{n+1} \leq p_n$ that is totally $(M_\eta, \P(T))$-generic for all 
    $\eta \in (\alpha_{n-1}, \alpha_n]$. Since $p'_{n+1} \leq p_n$ and $p_n$ is totally $(M_\eta, 
    \P(T))$-generic for all $\eta \leq \alpha_{n-1}$, it follows that $p'_{n+1}$ is in fact 
    totally $(M_\eta, \P(T))$-generic for all $\eta \leq \alpha_n$. Now apply Lemma 
    \ref{lemma: split} to find $p_{n+1} \leq p'_{n+1}$ and $\sigma_{n+1} \in S_{p_{n+1}}$ such that
    \begin{itemize}
      \item $p_{n+1} \in M_{\alpha_n + 1}$;
      \item $f_{p_{n+1}}(\sigma_{n+1}) \sqsubseteq b_{\gamma_{n+1}}$;
      \item for all $i < 2$, $\sigma_{n+1}{}^\frown i \in S_{p_{n+1}}$ 
      and $f_{p_{n+1}}(\sigma_{n+1}{}^\frown i ) \perp b_{\gamma_{n+1}}$.
    \end{itemize}
    
    This completes the construction. It remains to verify that $p_\infty \in \P(T)$. 
    Since $S_{p_\infty} \subseteq M_\beta$, the only way that $p_\infty$ could fail to be in 
    $\P(T)$ is if there is $\gamma \in \mc{L} \cap (\delta + 1)$ and $\sigma \in 
    \partial S_\infty$ such that $\bigcup \{f_{p_\infty}(\sigma \restriction \eta) \mid 
    \eta < \dom(\sigma)\} = b_\gamma$. Suppose there are such a $\gamma$ and 
    $\sigma$. 
    
    Assume first that $\gamma = \delta$. If we are in Case 1, then there is 
    $\xi^* < \delta$ such that $b_\delta \restriction \xi^* \notin M_\beta$. 
    Since $S_{p_\infty}, f_{p_\infty} \subseteq M_\beta$, there can be no 
    $\eta < \dom(\sigma)$ such that $f_{p_\infty}(\sigma \restriction \eta) 
    \sqsupseteq b_\delta \restriction \xi^*$, which is a contradiction. If we are in 
    Case 2, then recall our choice of $\sigma_0$, and let $\xi' < \delta$ be such that 
    $f_{p_1}(\sigma_0{}^\frown i ) \perp b_\delta \restriction \xi'$ 
    for all $i < 2$. Then, by Lemma \ref{lemma: pruning}, there is no $\eta < \dom(\sigma)$ 
    such that $f_{p_\infty}(\sigma \restriction \eta) \sqsupseteq b_\delta \restriction \xi'$, 
    which is again a contradiction.
    
    Finally, assume that $\gamma \in \mc{L} \cap \delta$. Then there is $0 < n < \omega$ 
    such that $\gamma = \gamma_n$. Let $\xi' < \gamma$ be such that 
    $f_{p_n}(\sigma_n{}^\frown i ) \perp b_\gamma \restriction \xi'$ 
    for all $i < 2$. Then, again by Lemma \ref{lemma: pruning}, there is no 
    $\eta < \dom(\sigma)$ such that $f_{p_\infty}(\sigma \restriction \eta) 
    \sqsupseteq b_\gamma \restriction \xi'$, yielding the final contradiction and 
    completing the proof.
  \end{proof}
  
  \begin{lemma}
    $\P(T)$ is completely proper.
  \end{lemma}
  
  \begin{proof}
    Suppose that $M$ is suitable for $\bb{Q}$ and we are given 
    $M \ra N_i$ for $i < 2$. Let $\varepsilon_i : (M, \in) \ra (N_i, 
    \in)$ witness this. Let $\delta := M \cap \omega_1$. Note that 
    $\delta \in N_0 \cap N_1$. By elementarity, 
    for $i < 2$, the restriction $\varepsilon_i \restriction \delta$ must 
    be the identity map on $\delta$, and hence $\pi \restriction (T
    \cap M)$ is the identity map on $T \cap M$. 
    Without loss of generality, assume that 
    $\delta \in \mc{L}^{N_0} \cap \mc{L}^{N_1}$; the other cases are similar 
    but easier. For each $i < 2$, let $b^i_\delta$ be such that
    \[
      N_i \models ``b^i_\delta \text{ is the unique element of } 
      [T^{N_i}_{<\delta}] \setminus T^{N_i}_\delta".
    \]
    
    Let $p \in M \cap \P(T)$ be arbitrary. 
    We will now build a decreasing sequence $\langle p_n \mid n < \omega 
    \rangle$ from $M \cap \P(T)$ below $p$ such that the upward closure of 
    $\{p_n \mid n < \omega\}$ in $M \cap \P(T)$ is an 
    $(M,\P(T))$-generic filter. Begin by 
    letting $p_0 = p$. Now define $p_1$ as follows. 
    If there is $\xi < \delta$ such that $b^0_\delta \restriction \xi 
    \notin M$, then simply let $p_1 = p_0$. Otherwise, apply Lemma 
    \ref{lemma: split} to find $p_1 \leq p_0$ and $\sigma_1 \in S_{p_1}$ 
    such that
    \begin{itemize}
      \item $p_1 \in M$;
      \item $f_{p_1}(\sigma_1) \sqsubseteq b^0_\delta$;
      \item for all $j < 2$, $\sigma_1{}^\frown j 
      \in S_{p_1}$ and $f_{p_1}(\sigma_1{}^\frown j ) 
      \perp b^0_\delta$.
    \end{itemize}
    Next, find $p_2 \leq p_1$ in the same way, but replacing $b^0_\delta$ 
    by $b^1_\delta$. Finally, let $\langle D_n \mid n < \omega \rangle$ 
    enumerate all dense open subsets of $\P(T)$ that are elements of $M$, and 
    recursively define a decreasing sequence $\langle p_n \mid 2 < n < 
    \omega \rangle$ from $M \cap \P(T)$ such that $p_n \in D_{n-3}$ for 
    all $n < \omega$. Notice that, for all $\gamma \in \mc{L} \cap \delta$, 
    the set $D^*_\gamma$ of all $q \in \P(T)$ for which there is 
    $\sigma \in S_q$ such that
    \begin{itemize}
      \item $f_q(\sigma) \sqsubseteq b_\gamma$; and
      \item for all $j < 2$ $\sigma^\frown j \in S_q$ 
      and $f_q(\sigma^\frown j ) \perp b_\gamma$
    \end{itemize}
    is a dense open subset of $\P(T)$ that is in $M$; there is therefore 
    $n < \omega$ such that $p_n \in D^*_\gamma$.
    
    Let $G$ be the upward closure of $\{p_n \mid n < \omega\}$ in 
    $M \cap \P(T)$. By construction, it is clear that $G$ is an 
    $(M, \P(T))$-generic filter. It remains to show that it is 
    $\overrightarrow{MN_i}$-prebounded for all $i < 2$. To this end, 
    fix $i < 2$ and an arrow $N_i \ra P$, witnessed by a map 
    $\epsilon : (N_i, \in) \ra (P, \in)$, such that $G \in P$. 
    Note that, for each $n < \omega$, we have 
    $\epsilon \circ \varepsilon_i(p_n) = p_n$, and hence $G^P = G$. 
    We will therefore be done if we show that 
    \[
      p_\infty = \left(\bigcup_{n < \omega} S_{p_n}, \bigcup_{n < \omega} f_{p_n}\right) 
      = \left(\bigcup_{q \in G} S_q, \bigcup_{q \in G} f_q \right)
    \]
    is in $\P(T)^P$.
    
    First note that $p_\infty \in P$, since $G \in P$. Also, the range 
    of $f_{p_\infty}$ is contained in $T^P_{<\delta} = T_{<\delta}$. 
    Therefore, the only way that $p_\infty$ could fail to be in $\P(T)^P$ 
    is if there is $\gamma \in \mc{L}^P \cap (\delta + 1)$ 
    and $\sigma \in \partial S_\infty$ such that 
    $\bigcup \{f_{p_\infty}(\sigma \restriction \eta) \mid \eta < 
    \dom(\sigma)\} = b_\gamma^P$, where $b_\gamma^P$ is such that
    \[
      P \models ``b^P_\gamma \text{ is the unique element of } 
      [T^P_{<\gamma}] \setminus T^P_\gamma".
    \] 
    Suppose that there are such a $\gamma$ and $\sigma$.
    
    Note that $\delta < N_i \cap \omega_1$, so $\epsilon \restriction (\delta + 1)$ 
    is the identity map. It follows that $\mc{L}^P \cap (\delta + 1) = 
    \mc{L}^{N_i} \cap (\delta + 1)$, and, moreover, $\mc{L}^P \cap \delta = 
    \mc{L} \cap \delta$. Suppose first that $\gamma = \delta$, in which case 
    we have $b^P_\delta = b^i_\delta$. If there is $\xi < \delta$ such 
    that $b^i_\delta \restriction \xi \notin M$, then, since 
    $S_{p_\infty}, f_{p_\infty} \subseteq M$, there can be no $\eta < 
    \dom(\sigma)$ such that $f_{p_\infty}(\sigma \restriction \eta) 
    \sqsupseteq b^i_\delta \restriction \xi$, which is a contradiction. 
    Otherwise, when constructing $p_{i+1}$, we fixed a $\sigma_{i+1} \in 
    S_{p_{i+1}}$ such that $f_{p_{i+1}}(\sigma_{i+1}) \sqsubseteq b^i_\delta$ 
    and for which there exists some $\xi' < \delta$ such that 
    $f_{p_{i+1}}(\sigma_{i+1}{}^\frown j ) \perp 
    b^i_\delta \restriction \xi'$ for all $j < 2$. Then, by 
    Lemma \ref{lemma: pruning}, there is no $\eta < \dom(\sigma)$ such 
    that $f_{p_\infty}(\sigma \restriction \eta) \sqsupseteq 
    b^i_\delta \restriction \xi'$, which is again a contradiction.
    
    Suppose finally that $\gamma < \delta$, in which case $b^P_\gamma = b_\gamma$. 
    Then there is $n < \omega$ such that $p_n \in D^*_\gamma$, and we reach 
    a contradiction exactly as in the second case in the previous paragraph. 
    Thus, $G$ is in fact $\overrightarrow{MN_i}$-prebounded, and hence 
    $\P(T)$ is completely proper.
  \end{proof}
  
  By a standard genericity argument, if $G$ is $\P(T)$-generic over 
  $V$, then $\bigcup \{f_p \mid p \in G\}$ witnesses that, in $V[G]$, 
  $T$ contains a copy of ${^{<\omega_1}}2$ (which, since $\P(T)$ is totally 
  proper, is the same when calculated in $V$ or in $V[G]$).
  
  We now define a countable support iteration $\langle \P_\eta, \dot{\Q}_\xi 
  \mid \eta \leq \omega_2, ~ \xi < \omega_2 \rangle$ such that, for each 
  $\eta < \omega_2$, there is a $\P_\eta$-name $\dot{T}_\eta$ for 
  a full, splitting subtree of ${^{<\omega_1}}\omega_1$ 
  of height and size $\omega_1$ such that
  \[
    \Vdash_{\P_\eta} \dot{\Q}_\eta = \P(\dot{T}_\eta).
  \]
  By Theorem \ref{thm: iteration_thm}, $\P_\eta$ will be totally proper 
  for each $\eta \leq \omega_2$, and hence 
  \begin{itemize}
    \item $\CH$ will hold in $V^{\P_\eta}$; and
    \item $({^{<\omega_1}}\omega_1)^{V^{\P_\eta}} = 
    ({^{<\omega_1}}\omega_1)^V$.
  \end{itemize}    
  It follows that, for each $\eta < \omega_2$, we have
  \[
    \Vdash_{\P_\eta} ``\dot{\Q}_\eta \text{ is a proper forcing of 
    size } \omega_1".
  \]
  Therefore, by \cite[Theorem 2.10]{abraham_proper_forcing}, $\P_\eta$ 
  has the $\omega_2$-cc and is of size $\leq 2^{\omega_1} = \omega_2$ 
  for each $\eta \leq \omega_2$. Thus, by a standard 
  bookkeeping argument, we can arrange so that, for every 
  $\P_{\omega_2}$-name $\dot{T}$ for a full, splitting subtree of 
  ${^{<\omega_1}}\omega_1$ of height and size $\omega_1$, there is 
  $\eta < \omega_2$ such that $\Vdash_{\P_{\omega_2}} \dot{T} = \dot{T}_\eta$.
  By genericity,
  \[
    \Vdash_{\P_{\eta+1}} ``\dot{T}_\eta \text{ contains a copy of } 
    {^{<\omega_1}}2".
  \]
  Since ${^{<\omega_1}}2$ is the same when calculated in $V$ or in 
  $V^{\P_\eta}$ for any $\eta \leq \omega_2$, it follows that
  \[
    \Vdash_{\P_{\omega_2}} ``\dot{T} \text{ contains a copy of } 
    {^{<\omega_1}}2".
  \]
  Hence, $V^{\P_{\omega_2}}$ is our desired model.
\end{proof}

If $\kappa$ is a regular uncountable cardinal, then we follow \cite{lucke_schlicht_descriptive} 
and say that a $\kappa$-tree $T$ is \emph{superthin} if $|[T_{<\delta}]| < \kappa$ 
for all limit ordinals $\delta < \kappa$. Note that cofinally splitting 
$\omega_1$-trees cannot be superthin, so the notion is primarily of interest 
for $\kappa > \omega_1$. In \cite{lucke_schlicht_descriptive}, 
L\"{u}cke and Schlicht prove that if $\kappa$ is a regular uncountable cardinal 
and there exists a superthin $\kappa$-Kurepa tree, then there exists a superthin 
$\kappa$-Kurepa tree $T \subseteq {^{<\kappa}}\kappa$ such that $[T]$ is a retract 
of ${^{\kappa}}\kappa$. They moreover show that such trees consistently exist; 
for example, if $V = L$, then superthin $\kappa$-Kurepa trees exist whenever 
$\kappa$ is the successor of a cardinal of uncountable cofinality. 
Superthinness seems to be in tension with fullness, 
but we show now that, for regular $\kappa > \omega_1$, we can consistently have 
full, superthin, $\kappa$-Kurepa trees. For concreteness, we focus on the case 
$\kappa = \omega_2$.

\begin{theorem}
  Suppose that $\diamondsuit + (2^{\omega_1} = \omega_2)$ holds. Then there is a 
  cardinality- and cofinality-preserving poset $\P$ such that, in $V^{\P}$, there is a normal, 
  splitting, full, superthin $\omega_2$-Kurepa tree.
\end{theorem}

\begin{proof}
  The poset $\P$ consists of all pairs $p = (T_p, f_p)$ such that
  \begin{itemize}
    \item there is an $\eta_p < \omega_2$ such that $T_p$ is a normal, splitting, full 
    subtree of ${^{<\eta_p + 1}}\omega_2$;
    \item for all $\xi \leq \eta_p$, we have $|T_p \cap {^{\xi}}\omega_2| \leq \omega_1$;
    \item for every limit ordinal $\xi \leq \eta_p$, we have $|[(T_p)_{<\xi}]| \leq \omega_1$;
    \item $f_p$ is a partial function of size $\omega_1$ from $\omega_3$ to 
    $T_p \cap {^{\eta_q}}\omega_2$.
  \end{itemize}
  If $p_0, p_1 \in \P$, then $p_1 \leq p_0$ if
  \begin{itemize}
    \item $\eta_{p_1} \geq \eta_{p_1}$;
    \item $T_{p_1} \cap {^{<\eta_{p_0}+1}}\omega_2 = T_{p_0}$;
    \item $\dom(f_{p_1}) \supseteq \dom(f_{p_0})$;
    \item for all $\alpha \in \dom(f_{p_0})$, we have $f_{p_1}(\alpha) \sqsupseteq 
    f_{p_0}(\alpha)$.
  \end{itemize}
  By a standard $\Delta$-system argument, and using the assumption that $2^{\omega_1} 
  = \omega_2$, it follows that $\P$ has the $\omega_3$-cc. 
  
  \begin{claim} \label{claim: strategic_closure}
    $\P$ is $(\omega_1 + 1)$-strategically closed.
  \end{claim}
  
  \begin{proof}
    We describe a winning strategy for Player II in $\Game_{\omega_1+1}(\P)$. 
    Given a (partial) play $\langle p_\alpha \mid \alpha < \gamma \rangle$ of 
    $\Game_{\omega_1+1}(\P)$, for all $\alpha < \gamma$, we let $T^\alpha$, $f^\alpha$, 
    and $\eta_\alpha$ denote $T_{p_\alpha}$, $f_{p_\alpha}$, and $\eta_{p_\alpha}$, 
    respectively. As part of Player II's winning strategy, they also fix, as they play 
    a round $\langle p_\alpha \mid \alpha \leq \omega_1 \rangle$, arbitrary surjections 
    $\pi_\alpha : \omega_1 \ra T^\alpha_{\eta_\alpha}$ for each ordinal $\alpha < \omega_1$.
    We will also ensure that, for every even ordinal $\alpha < \omega_1$, the map $f^\alpha$ is a 
    bijection between its domain and $T^\alpha_{\eta_\alpha}$.
    
    Using $\diamondsuit$, fix a sequence $\langle a_\alpha : \alpha \ra \alpha \mid \alpha < 
    \omega_1 \rangle$ such that, for all $f: \omega_1 \ra \omega_1$, the set
    \[
      \{\alpha < \omega_1 \mid f \restriction \alpha = a_\alpha\}
    \]
    is stationary in $\omega_1$.
    
    Suppose now that $\beta \leq \omega_1$ is a nonzero even ordinal and 
    $\langle p_\alpha \mid \alpha < \beta \rangle$ is a partial run of $\Game_{\omega_1 + 1}(\P)$, 
    with Player II playing thus far according to their winning strategy. Suppose also that 
    Player II has fixed surjections $\pi_\alpha : \omega_1 \ra T^\alpha_{\eta_\alpha}$ for each 
    $\alpha$ such that there exists an even ordinal $\beta'$ with $\alpha \leq \beta' < \beta$. 
    We now describe how Player II should select $p_\beta$.
    
    Suppose first that $\beta = \alpha + 1$ is a successor ordinal. In this case, simply let 
    $\eta_\beta = \eta_\alpha + 1$ and 
    \[
      T^\beta_{\eta_\beta} = \{\sigma^\frown i \mid \sigma \in 
      T^\alpha_{\eta_\alpha} \text{ and } i < \omega_1\}.
    \]
    We now describe how to choose $f^\beta$. We are going to let $\dom(f^\beta) = 
    \dom(f^\alpha) \cup e_\beta$ for some (possibly empty) $e_\beta \in [\omega_3]^{\leq 
    \omega_1}$ disjoint from $\dom(f^\alpha)$. First, for all $\gamma \in \dom(f^\alpha)$, 
    let $f^\beta(\gamma) = f^\alpha(\gamma)^\frown i$ for some $i < \omega_1$ 
    in such a way that $f^\beta \restriction \dom(f^\alpha)$ is injective (this is possible, since each 
    element of $T^\alpha_{\eta_\alpha}$ has $\omega_1$-many successors in $T^\beta_{\eta_\beta}$). 
    The resulting function may not be surjective, though, so choose some $e_\beta$ as above 
    of the appropriate cardinality and define $f^\beta$ on $e_\beta$ in such a way that the resulting 
    function is a bijection. Finally, let $\pi_\alpha$ and $\pi_\beta$ be arbitrary surjections 
    from $\omega_1$ onto $T^\alpha_{\eta_\alpha}$ and $T^\beta_{\eta_\beta}$, respectively.
    
    Suppose next that $\beta$ is a limit ordinal of countable cofinality. Let 
    $\eta_\beta = \sup\{\eta_\alpha \mid \alpha < \beta\}$. Note that, by the specification of 
    Player II's strategy at successor stages, we know that $\eta_\beta > \eta_\alpha$ for 
    all $\alpha < \beta$; in particular, $\eta_\beta$ is a limit ordinal of countable cofinality.
    Let $T^{<\beta} = \bigcup \{T^\alpha \mid \alpha < \beta\}$, so $T^{<\beta}$ is a normal, 
    splitting, full subtree of ${^{<\eta_\beta}}\omega_2$. Let 
    $d_\beta = \bigcup \{\dom(f^\alpha) \mid \alpha < \beta\}$ and, for each 
    $\gamma \in d_\beta$, let 
    \[
      b^\beta_\gamma = \bigcup \{f^\alpha(\gamma) \mid \alpha < \beta \text{ and } \gamma \in 
      \dom(f^\alpha)\}.
    \]
    Note that $b^\beta_\gamma$ is in $[T^{<\beta}]$. Our $\diamondsuit$ 
    sequence gives us a function $a_\beta : \beta \ra \beta$. Consider the subset 
    $A^\beta  = \{\pi_\alpha(a_\beta(\alpha)) \mid \alpha < \beta\}$ of $T^{<\beta}$. 
    If $A^\beta$ is linearly ordered by $\subseteq$, then let $b^\beta_\ast = \bigcup A^\beta$, 
    and note that $b^\beta_\ast \in [T^{<\beta}]$. There are now two cases to consider.
    
    \textbf{Case 1.} Suppose first that $A^\beta$ is linearly ordered and, for every $\gamma 
    \in d_\beta$, we have $b^\beta_\ast \neq b^\beta_\gamma$. In this case, let 
    $T^\beta_{\eta_\beta} = [T^{<\beta}] \setminus \{b^\beta_\ast\}$. Note that, since 
    $\cf(\beta) = \omega$, $\CH$ holds, and every level of $T^{<\beta}$ has size at most 
    $\omega_1$, we have $|T^\beta_{\eta_\beta}| = \omega_1$. To define $f^\beta$, 
    we will let $\dom(f^\beta) = d_\beta \cup e_\beta$ for some $e_\beta \in [\omega_3]^{\leq \omega_1}$ 
    disjoint from $d_\beta$. For $\gamma \in d_\beta$, let $f^\beta(\gamma) = b^\beta_\gamma$. Note that 
    $f^\beta \restriction d_\beta$ is injective since, for all even $\alpha < \beta$, 
    $f^\alpha$ is injective. As in the successor case, now choose a set $e_\beta$ of the 
    appropriate cardinality and define $f^\beta$ on $e_\beta$ so that the resulting function 
    is a bijection. Finally, let $\pi_\beta$ be an arbitrary surjection from $\omega_1$ 
    onto $T^\beta_{\eta_\beta}$.
    
    \textbf{Case 2.} If we were not in Case 1, i.e., if either $A^\beta$ is not linearly 
    ordered or $b^\beta_\ast = b^\beta_\gamma$ for some $\gamma \in d^*$, then proceed exactly 
    as in Case 1, except let $T^\beta_{\eta_\beta} = [T^{<\beta}]$.
    
    Finally, suppose that $\beta = \omega_1$. In this case, we just need to show that 
    $\langle p_\alpha \mid \alpha < \omega_1 \rangle$ has a lower bound. As above, 
    let $\eta_{\omega_1} = \sup\{\eta_\alpha \mid \alpha < \omega_1\}$, let 
    $T^{<\omega_1} = \bigcup \{T^\alpha \mid \alpha < \omega_1\}$, and let $d_{\omega_1} = 
    \bigcup\{\dom(f^\alpha) \mid \alpha < \omega_1\}$. For each $\gamma \in d_{\omega_1}$, let
    \[
      b^{\omega_1}_\gamma = \bigcup \{f^\alpha(\gamma) \mid \alpha < \omega_1 \text{ and } 
      \gamma \in \dom(f^\alpha)\}.
    \]
    For all $\gamma \in d_{\omega_1}$, we have $b^{\omega_1}_\gamma \in [T^{<\omega_1}]$.
    
    \begin{subclaim}
      For every $b \in [T^{<\omega_1}]$, there is $\gamma \in d_{\omega_1}$ such that 
      $b = b^{\omega_1}_\gamma$.
    \end{subclaim}
    
    \begin{proof}
      Suppose for the sake of contradiction that $b \in [T^{<\omega_1}]$ and, for all 
      $\gamma \in d_{\omega_1}$, we have $b \neq b^{\omega_1}_\gamma$. Recall that, for each 
      $\alpha \in \lim(\omega_1)$, the function $f^\alpha$ is bijective; therefore, there 
      is a unique $\gamma_\alpha \in \dom(f^\alpha)$ such that $f^\alpha(\gamma_\alpha) 
      \sqsubseteq b$. By assumption, for each $\alpha \in \lim(\omega_1)$, we can find 
      $\alpha^\dagger \in \lim(\omega_1)$ such that $b \restriction \alpha^\dagger 
      \neq b^{\omega_1}_{\gamma_\alpha} \restriction \alpha^\dagger$. Let 
      \[
        C = \{\alpha' \in \lim(\omega_1) \mid \forall \alpha \in \lim(\alpha') \ 
        \alpha^\dagger < \alpha'\}.
      \]
      Then $C$ is a club in $\omega_1$ and, for all $\alpha' \in C$ and all $\gamma \in 
      d_{\alpha'}$, we have $b \restriction \eta_{\alpha'} \neq b^{\alpha'}_\gamma$.
      
      Let $f: \omega_1 \ra \omega_1$ be such that, for each $\alpha < \omega_1$, 
      we have $\pi_\alpha(f(\alpha)) = b \restriction \eta_\alpha$. We can then 
      find $\alpha' \in C$ such that $f \restriction \alpha' = a_{\alpha'}$. Now 
      recall the specification of Player II's strategy at stage $\alpha'$ of this 
      run of the game. Unraveling the definitions, we have $b^{\alpha'}_\ast 
      = b \restriction \eta_{\alpha'}$, and, for every $\gamma \in d_{\alpha'}$, 
      we have $b^{\alpha'}_\ast \neq b^{\alpha'}_\gamma$. Therefore, Player II played 
      $T^{\alpha'}_{\eta_{\alpha'}} = [T^{<\alpha'}] \setminus \{b^{\alpha'}_\ast\}$. 
      In particular, $b \restriction \eta_{\alpha'} \notin T^{<\omega_1}$, contradicting 
      the assumption that $b \in [T^{<\omega_1}]$.
    \end{proof}
    
    Since $|d_{\omega_1}| = \omega_1$, it follows that $|[T^{<\omega_1}]| = \omega_1$.
    Define a subtree $T^{\omega_1}$ of ${^{<\eta_{\omega_1}+1}}\omega_2$ 
    by letting $T^{\omega_1} \cap {^{<\eta_{\omega_1}}}\omega_2 = T^{<\omega_1}$ and 
    $T^{\omega_1}_{\eta_{\omega_1}} = [T^{<\omega_1}]$. The normality of $T^{\omega_1}$ 
    follows from the fact that $f^\alpha$ is surjective onto $T^\alpha_{\eta_\alpha}$ for all 
    even ordinals $\alpha < \omega_1$, and hence each element of $T^{<\omega_1}$ is 
    an initial segment of $b^{\omega_1}_\gamma$ for some $\gamma \in d_{\omega_1}$. 
    Let $\dom(f^{\omega_1}) = d_{\omega_1}$ and, for each $\gamma \in d_{\omega_1}$, let 
    $f^{\omega_1}(\gamma) = b^{\omega_1}_\gamma$. Then $(T^{\omega_1}, f^{\omega_1})$ 
    is a lower bound for $\langle p_\alpha \mid \alpha < \omega_1 \rangle$ in $\P$.
  \end{proof}
  
  Since $\P$ has the $\omega_3$-cc and is $(\omega_1+1)$-strategically closed, it follows 
  that it preserves all cardinalities and cofinalities. Let $G$ be $\P$-generic over $V$, let 
  $T = \bigcup \{T_p \mid p \in G\}$, let $d = \bigcup \{\dom(f_p) \mid p \in G\}$, and define 
  a function $f$ with domain $d$ by letting $f(\gamma) = \bigcup \{f_p(\gamma) \mid p \in G 
  \text{ and } \gamma \in \dom(p)\}$. Then standard genericity arguments, combined with the 
  arguments of the proof of Claim \ref{claim: strategic_closure}, show that
  \begin{itemize}
    \item $T$ is a normal, splitting, full, superthin $\omega_2$-tree;
    \item $d = \omega_3$;
    \item $f$ is an injective function from $\omega_3$ to $[T]$.
  \end{itemize}
  Thus, $T$ is the desired tree as in the statement of the theorem.
\end{proof}

\section{Adding superthin subtrees} \label{section: superthin}

In this section, we prove Theorem D.
Fix for the remainder of the section a regular uncountable cardinal $\mu$ such that $\mu^{<\mu} = \mu$, 
$2^\mu = \mu^+$, and $2^{\mu^+} = \mu^{++}$, and let $\kappa = \mu^+$. Fix for now a closed set 
$E \subseteq {^{<\kappa}}\kappa$ such that $|E| > \kappa$. 

We introduce a forcing notion $\bb{P}(E)$ that will add a superthin $\kappa$-Kurepa 
subtree to $T(E)$. We first thin $E$ out to a subset such that all nonempty neighborhoods are 
large. Let $\Sigma = \{\sigma \in {^{<\kappa}}\kappa \mid |E \cap N_\sigma| \leq 
\kappa\}$, and let 
\[
  E' = E \setminus \bigcup \{N_\sigma \mid \sigma \in \Sigma\}.
\]
By replacing $E$ with $E'$, assume from now on that $E$ has the property that, for all 
$\sigma \in {^{<\kappa}}\kappa$, either $E \cap N_\sigma = \emptyset$ or 
$|E \cap N_\sigma| > \kappa$.

Conditions of $\bb{P}(E)$ are all triples of the form $p = (B^p, \gamma^p, t^p)$ such that
\begin{itemize}
  \item $B^p \in [E]^{\leq \mu}$ is nonempty;
  \item $\gamma^p < \kappa$ is such that, for all distinct $b,b' \in B^p$, we have 
  $b \restriction \gamma^p \neq b' \restriction \gamma^p$;
  \item $t^p = \{b \restriction \alpha \mid b \in B^p, \ \alpha \leq \gamma^p\}$;
  \item $t^p$ looks like an initial segment of a normal superthin subtree of 
  $T(E)$, i.e., 
  \begin{itemize}
    \item $t^p$ is a normal subtree of ${^{\leq \gamma^p}}\kappa$;
    \item all levels of $t^p$ have cardinality at most $\mu$;
    \item for all limit ordinals $\gamma \leq \gamma^p$, we have $|[t^p_{<\gamma}]| 
    \leq \mu$.
  \end{itemize}
\end{itemize} 

We note that, for a condition $p \in \bb{P}(E)$, $t^p$ is uniquely determined by 
$B^p$ and $\gamma^p$; we include it in the notation for convenience.

Given $p,q \in \bb{P}(E)$, we set $q \leq p$ if and only if
\begin{itemize}
  \item $B^q \supseteq B^p$;
  \item $\gamma^q \geq \gamma^p$;
  \item $t^q$ end-extends $t^p$, i.e., $t^q \cap {^{\leq \gamma^p}}\kappa = t^p$.
\end{itemize}
We also include $(\emptyset, \emptyset, \emptyset)$ as $1_{\P(E)}$. 

\begin{proposition} \label{prop: knaster}
  $\bb{P}(E)$ is $\kappa^+$-Knaster.
\end{proposition}

\begin{proof}
  Let $\langle p_\delta \mid \delta < \kappa^+ \rangle$ be a sequence of conditions 
  from $\bb{P}(E)$, with each $p_\delta$ of the form $(B^\delta, \gamma^\delta, t^\delta)$. 
  Since $\kappa^{<\kappa} = \kappa$, by thinning out our sequence if necessary we can 
  assume that there are fixed $\gamma < \kappa$ and $t \subseteq {^{\leq \gamma}}\kappa$ 
  such that, for all $\delta < \kappa^+$, we have $\gamma^\delta = \gamma$ and 
  $t^\delta = t$.
  
  We claim that, for all $\delta_0 < \delta_1 < \kappa^+$, the conditions $p_{\delta_0}$ and 
  $p_{\delta_1}$ are compatible. To this end, fix such $\delta_0 < \delta_1$. 
  For each $\sigma \in t_\gamma$ and each $i < 2$, there is a unique $b_{\sigma, i} \in 
  B^{\delta_i}$ such that $b_{\sigma,i} \restriction \gamma = \sigma$. Choose 
  $\gamma' \in [\gamma, \kappa)$ large enough so that, for all $\sigma \in t_\gamma$, 
  either $b_{\sigma,0} = b_{\sigma,i}$ or $b_{\sigma,0} \restriction \gamma' 
  \neq b_{\sigma,1} \restriction \gamma'$. Let $B' = B^{\delta_0} \cup B^{\delta_1}$ and 
  $t' = \{b \restriction \alpha \mid b \in B', \ \alpha \leq \gamma'\}$. Then it is readily 
  verified that $p' = (B', \gamma', t')$ is a common extension of $p^{\delta_0}$ and 
  $p^{\delta_1}$ in $\bb{P}(E)$. For example, to verify that $|[t_{<\beta}]| \leq \mu$ 
  for all limit ordinals $\beta \leq \gamma'$, fix such a $\beta$. If $\beta \leq \gamma$, 
  then the desired conclusion follows from the fact that $p^{\delta_0}$ and $p^{\delta_1}$ 
  are in $\bb{P}(E)$. If $\beta \in (\gamma, \gamma']$, then our construction immediately 
  implies that every element of the form $[t_{<\beta}]$ is of the form $b \restriction \beta$ 
  for some $b \in B'$.
\end{proof}

\begin{proposition} \label{prop: strategically closed}
  $\bb{P}(E)$ is $\kappa$-strategically closed.
\end{proposition}

\begin{proof}
  We describe a winning strategy for Player II in $\Game_\kappa(\bb{P}(E))$. The winning 
  strategy is very simple: at every even stage of the game, Player II does the minimal amount 
  of work necessary. More precisely, suppose that $\beta < \kappa$ is an even ordinal and 
  $\langle p_\alpha \mid \alpha < \beta \rangle$ is a partial play of the game, with Player II 
  playing so far according to the strategy being described here. For each $\alpha < \beta$, 
  let $p_\alpha = (B^\alpha, \gamma^\alpha, t^\alpha)$.
  
  If $\beta$ is a successor ordinal, then Player II can play arbitrarily; for instance, 
  they can simply play $p_\beta = p_{\beta-1}$. If $\beta$ is a limit ordinal, then 
  they specify $p_\beta$ by setting $B^\beta = \bigcup \{B^\alpha \mid \alpha < \beta\}$ 
  and $\gamma^\beta = \sup\{\gamma^\alpha \mid \alpha < \beta\}$ (this uniquely determines 
  $t^\beta$). 
  
  In order to prove that this is a winning strategy, we must only verify that 
  $p_\beta$ thus defined is indeed a condition in $\P(E)$. Let $\gamma = \gamma^\beta$ 
  and $t = t^\beta$. The only nontrivial condition to verify is the requirement that 
  $|[t_{<\gamma}]| \leq \mu$. If $\cf(\beta) < \mu$, then this follows immediately from 
  the fact that $\mu^{<\mu} = \mu$ and all levels of $t$ have cardinality at most $\mu$. 
  
  Thus, assume that $\cf(\beta) = \mu$. We claim that, in this case, every element of 
  $[t_{<\gamma}]$ is of the form $b \restriction \gamma$ for some $b \in B$. Since 
  $|B| \leq \mu$, this will suffice. To this end, fix $d \in [t_{<\gamma}]$.
  Let $\langle \alpha_i \mid i < \mu \rangle$ be a strictly increasing enumeration 
  of a club in $\beta$. By construction, for each $i < \mu$, there is a unique 
  $b_i \in B^{\alpha_i}$ such that $d \restriction \gamma^{\alpha_i} = b_i \restriction 
  \gamma^{\alpha_i}$. If $j < \mu$ is a limit ordinal, then the assumption that Player II 
  has played so far according to the described strategy implies that 
  $B^{\alpha_j} = \bigcup \{B^{\alpha_i} \mid i < j\}$ and $\gamma^{\alpha_j} = 
  \sup\{\gamma^{\alpha_i} \mid i < j\}$. Therefore, for each limit ordinal 
  $j < \mu$, we can find $i(j) < j$ such that $b_j \in B^{\alpha_{i(j)}}$. 
  By Fodor's Lemma, we can find a stationary $S \subseteq \mu$ and a fixed $i < \mu$ 
  such that $i(j) = i$ for all $j \in S$. Since distinct elements of $B^{\alpha_i}$ have 
  distinct restrictions to $\gamma^{\alpha_i}$, it follows that there must be a fixed 
  $b \in B^{\alpha_i}$ such that $b_j = b$ for all $j \in S$. But then 
  $d \restriction \gamma^{\alpha_j} = b \restriction \gamma^{\alpha_j}$ for all $j \in S$, 
  and hence $d = b \restriction \gamma$.
\end{proof}

By Propositions \ref{prop: knaster} and \ref{prop: strategically closed}, forcing with 
$\P(E)$ preserves all cardinalities and cofinalities. We now show that it adds a normal, 
cofinally splitting, superthin $\kappa$-Kurepa subtree of $T(E)$.

\begin{theorem}
  Suppose that $G$ is $\P(E)$-generic over $V$, and let $T = \bigcup \{t^p \mid p \in G\}$. 
  Then in $V[G]$, $T$ is a normal, cofinally splitting, superthin $\kappa$-Kurepa subtree of 
  $T(E)$.
\end{theorem}

\begin{proof}
  In $V$, let $\dot{G}$ be the canonical $\P(E)$-name for the generic filter, and 
  let $\dot{T}$ be a $\P(E)$-name for $\bigcup\{t^p \mid p \in \dot{G}\}$. 
  It is immediate that, for every $\gamma < \kappa$, the set 
  $D_\gamma = \{p \in \P(E) \mid \gamma^p \geq \gamma\}$ is dense in $\P(E)$. As a result, 
  the definition of $\P(E)$ implies that $\dot{T}$ is forced to be a superthin normal 
  subtree of $T(E)$. 
  
  We now show that $\dot{T}$ is forced to be cofinally splitting. To this end, fix a 
  condition $p \in \P(E)$ and a $\sigma \in {^{<\kappa}}\kappa$ such that 
  $p \Vdash \sigma \in \dot{T}$. Note that there must be $b \in B^p$ such that 
  $\sigma \sqsubseteq b$; otherwise, if $\gamma \in [\gamma^p, \kappa)$ is such that 
  $\sigma \in {^{\leq \gamma}}\kappa$, then the condition $q$ with $B^q = B^p$ and 
  $\gamma^q = \gamma$ would extend $p$ and force $\sigma \notin \dot{T}$. 
  By increasing $\gamma^p$ if necessary, we can assume that $\sigma \in 
  {^{\leq\gamma^p}}\kappa$. Let $\sigma' \in t^p_{\gamma^p}$ be such that 
  $\sigma \sqsubseteq \sigma'$, and let $b$ be the unique element of $B^p$ such 
  that $\sigma' \sqsubseteq b$. By our assumption about $E$, we know that 
  $|E \cap N_{\sigma'}| > \kappa$, so we can fix $b' \in E \cap N_{\sigma'}$ such that 
  $b' \neq b$. Define a condition $q$ by letting $B^q = B^p \cup \{b'\}$ and letting 
  $\gamma^q \in (\gamma^p, \kappa)$ be large enough so that $b \restriction \gamma^q 
  \neq b' \restriction \gamma^q$. Then $q$ extends $p$ and forces that $b \restriction 
  \gamma^q$ and $b' \restriction \gamma^q$ are incomparable elements of $\dot{T}$ 
  extending $\sigma$.
  
  We finally show that $\dot{T}$ is forced to have at least $\kappa^+$-many cofinal 
  branches. Let $\dot{B}$ be a $\P(E)$-name for $\bigcup \{B^p \mid p \in \dot{G}\}$. 
  It will suffice to show that $\dot{B}$ is forced to have cardinality greater than 
  $\kappa$. Suppose for the sake of contradiction that $p \in \P(E)$ and 
  $p \Vdash |\dot{B}| \leq \kappa$. Since $\P(E)$ has the $\kappa^+$-c.c., we can 
  find $q \leq p$ and $A \in [E]^\kappa$ such that $q \Vdash \dot{B} \subseteq A$.
  Let $\sigma$ be an arbitrary element of $t^q_{\gamma^q}$. Since 
  $|E \cap N_\sigma| > \kappa$, we can find $b \in (E \cap N_\sigma) \setminus (A \cup B^q)$.
  Precisely as in the previous paragraph, we can find $r \leq q$ such that $b \in B^r$, 
  contradicting the fact that $r \notin A$ and $q \Vdash \dot{B} \subseteq A$.
\end{proof}

We now show that appropriate iterations of forcings of the form $\P(E)$ are well-behaved.

\begin{theorem} \label{theorem: well-behaved}
  Suppose that $\langle \P_i, \dot{\Q}_j \mid i \leq \varepsilon, j < \varepsilon \rangle$ 
  is a $({\leq}\mu)$-support iteration such that, for all $i < \varepsilon$, there is a 
  $\P_i$-name $\dot{E}_i$ such that
  \begin{itemize}
    \item $\dot{E}_i$ is forced to be a nonempty closed subset of ${^\kappa}\kappa$ such that, 
    for all $\sigma \in T(\dot{E}_i)$, we have $|\dot{E}_i \cap N_\sigma| > \kappa$;
    \item $\Vdash_{\P_i} \dot{\Q}_i = \P(\dot{E})$.
  \end{itemize}
  Then $\P_\varepsilon$ is $\kappa^+$-Knaster and $\kappa$-strategically closed.
\end{theorem}

\begin{proof}
  Since $\kappa = \mu^+$, the iteration is taken with supports of size ${\leq}\mu$, 
  and each iterand is forced to be $\kappa$-strategically closed, standard arguments 
  imply that $\P_\varepsilon$ is $\kappa$-strategically closed (roughly speaking, 
  Player II simply plays according to their winning strategy on each coordinate). 
  
  To show that $\P_\varepsilon$ is $\kappa^+$-Knaster, we first isolate a well-behaved 
  dense subset of $\P_\varepsilon$. For concreteness, for all $j \leq \varepsilon$, 
  we will think of conditions of $\P_j$ as being functions whose domains are subsets of 
  $j$ of cardinality ${\leq}\mu$.
  
  \begin{claim} \label{claim: dense_subset}
    For all $j \leq \varepsilon$, let $\P^*_j$ be the set of $p \in \P_j$ such that, 
    for all $i \in \dom(p)$, there are $\gamma^{p,i} < \kappa$, $t^{p,i} \subseteq 
    {^{\leq \gamma^{p,i}}}\kappa$, a collection $B^{p,i}$ of $\P_i$-names, and a 
    bijection $\pi^{p,i}:B^{p,i} \ra t^{p,i}_{\gamma^{p,i}}$ such that 
    \begin{itemize}
      \item $p \restriction i \Vdash_{\P_i} p(i) = (B^{p,i}, \gamma^{p,i}, t^{p,i})$;
      \item for all $\dot{b} \in B^{p,i}$, $p \restriction i \Vdash_{\P_i} 
      \dot{b} \restriction \gamma^{p,i} = \pi^{p,i}(\dot{b})$.
    \end{itemize}        
    Then $\P^*_j$ is dense in $\P_j$.
  \end{claim}
  
  Before we prove Claim \ref{claim: dense_subset}, we establish the following useful fact.
  
  \begin{claim} \label{claim: technical_claim}
    Fix $j \leq \varepsilon$, let $\beta < \kappa$ be a limit ordinal, and suppose that 
    $\vec{p} = \langle p_\alpha \mid \alpha < \beta \rangle$ is a decreasing sequence of conditions 
    from $\P_j$ such that
    \begin{itemize}
      \item for all $\alpha < \beta$, we have $p_\alpha \in \P^*_j$, as witnessed by 
      $\langle (B^{p_\alpha, i}, \gamma^{p_\alpha,i}, t^{p_\alpha,i}, 
      \pi^{p_\alpha,i}) \mid i \in \dom(p_\alpha) 
      \rangle$;
      \item for all $\alpha < \alpha' < \beta$ and all $i \in \dom(p_\alpha)$, we have 
      $B^{p_\alpha, i} \subseteq B^{p_{\alpha'},i}$;
      \item for all limit ordinals $\alpha' < \beta$, we have
      \begin{itemize}
        \item $\dom(p_{\alpha'}) = \bigcup \{\dom(p_\alpha) \mid \alpha < \alpha'\}$;
        \item for all $i \in \dom(p_{\alpha'})$, we have $\gamma^{p_{\alpha'},i} = 
        \sup\{\gamma^{p_\alpha,i} \mid \alpha < \alpha' \text{ and } i \in \dom(p_\alpha)\}$ 
        and $B^{p_{\alpha'},i} = 
        \bigcup \{B^{p_\alpha,i} \mid \alpha < \alpha' \text{ and } i \in \dom(p_\alpha)\}$.
      \end{itemize}
    \end{itemize}
    Then $\vec{p}$ has a lower bound $p_\beta$ that is in $\P^*_j$, as witnessed by
    $\langle (B^{p_\beta,i}, \gamma^{p_\beta,i}, t^{p_\beta,i}, \pi^{p_\beta,i}) \mid 
    i \in \dom(p_\beta) \rangle$ satisfying:
    \begin{itemize}
      \item $\dom(p_\beta) = \bigcup\{\dom(p_\alpha) \mid \alpha < \beta\}$;
      \item for all $i \in \dom(p_\beta)$, we have $\gamma^{p_{\beta},i} = 
        \sup\{\gamma^{p_\alpha,i} \mid \alpha < \beta \text{ and } i \in \dom(p_\alpha)\}$ 
        and $B^{p_{\beta},i} = 
        \bigcup \{B^{p_\alpha,i} \mid \alpha < \beta \text{ and } i \in \dom(p_\alpha)\}$.
    \end{itemize}
  \end{claim}  
  
  \begin{proof}
  We define $p_\beta$ as follows. First, let $\dom(p_\beta) = \bigcup \{\dom(p_\alpha) \mid 
  \alpha < \beta\}$. Then, for each $i \in \dom(p_\beta)$, we let $p_\beta(i)$ be a 
  $\P_i$-name for the triple $(B^{p_\beta,i}, \gamma^{p_\beta,i}, t^{p_\beta,i})$, defined 
  as follows. First, let $B^{p_\beta,i} = \bigcup \{B^{p_\alpha,i} \mid \alpha < \beta 
  \text{ and } i \in B^{p_\alpha}\}$ and let $\gamma^{p_\beta,i} = \sup\{\gamma^{p_\alpha, 
  i} \mid \alpha < \beta \text{ and } i \in \dom(p_\alpha)\}$.
  
  Let $s^{p_\beta,i} = \bigcup\{t^{p_\alpha,i} \mid \alpha < \beta \text{ and } 
  i \in \dom(p_\alpha)\}$ and, for each $\dot{b} \in B^{p_\beta,i}$, let 
  \[
    d_{\dot{b}} = \bigcup \{\pi^{p_\alpha,i}(\dot{b}) \mid \alpha < \beta, \ i 
    \in \dom(p_\alpha), \text{ and } \dot{b} \in B^{p_\alpha,i}\}.
  \]
  Then $d_{\dot{b}}$ is a cofinal branch through $s^{p_\beta,i}$ and, if $\dot{b}$ and 
  $\dot{b}'$ are distinct elements of $B^{p_\beta,i}$, then they both appear in 
  $B^{p_\alpha,i}$ for some $\alpha < \beta$, and the fact that $\pi^{p_\alpha,i}$ 
  is injective implies that $d_{\dot{b}}$ and $d_{\dot{b}'}$ are distinct. We end by setting 
  $t^{p_\beta,i} = s^{p_\beta,i} \cup \{d_{\dot{b}} \mid \dot{b} \in B^{p_\beta,i}\}$ 
  and, for each $\dot{b} \in B^{p_\beta,i}$, setting $\pi^{p_\beta,i}(\dot{b}) = d_{\dot{b}}$.
  
  We claim that $p_\beta$ thus defined is as desired. The only nontrivial requirement to verify 
  is that, for all $i \in \dom(p_\beta)$, the tree $t^{p_\beta,i}$ looks like an initial 
  segment of a normal superthin subtree of ${^{<\kappa}}\kappa$. This amounts to verifying 
  that $|[s^{p_\beta,i}]| \leq \mu$. This is proven exactly as in the proof of Proposition 
  \ref{prop: strategically closed}, so we leave it to the reader.
\end{proof}
  
  \begin{proof}[Proof of Claim \ref{claim: dense_subset}]
    The proof is by induction on $j$. We will actually establish the following technical 
    strengthening of the density of $\P^*_j$:
    \begin{quote}
      Suppose that $p_{00} \in \P^*_j$, as witnessed by $\langle (B^{p_{00},i}, 
      \gamma^{p_{00},i}, t^{p_{00},i}, \pi^{p_{00},i}) \mid i \in \dom(p_{00}) \rangle$, 
      and $p_0 \leq_{\P_j} p_{00}$. Then there is $q \leq p_0$ in $\P^*_j$ witnessed 
      by $\langle (B^{q,i}, \gamma^{q,i}, t^{q,i}, \pi^{q,i}) \mid i \in 
      \dom(q) \rangle$ such that, for all $i \in \dom(p_{00})$, we have 
      $B^{p_{00},i} \subseteq B^{q,i}$.
    \end{quote}
    To see that this does indeed yield the density of $\P^*_j$, note that in the above 
    statement we can let $p_0 \in \P_j$ be arbitrary and take $p_{00}$ to be such 
    that $\dom(p_{00}) = \emptyset$. Then $p_0 \leq p_{00}$, and $p_{00}$ is trivially 
    in $\P^*_j$.
    
    Fix $j \leq \varepsilon$, and suppose that the inductive hypothesis
    has been established for all $i < j$. Fix $p_0 \leq p_{00}$ as in the statement of 
    the hypothesis; we will find $q \leq p_0$ as desired.
    
    If $j = 0$, then we can take $q = p_0$. Suppose next that $j$ is a 
    successor ordinal, say $j = j_0 + 1$. Assume that $j_0 \in \dom(p_0)$; otherwise, 
    $p \in \P^*_{j_0}$, and we can apply the inductive hypothesis to obtain the desired 
    conclusion. Suppose also that $j_0 \in \dom(p_{00})$; the argument is similar but 
    easier in the other case.
    
    Since $\P_{j_0}$ is $\kappa$-strategically closed and hence 
    $({<}\kappa)$-distributive, we can find $p_1 \leq p_0 \restriction j_0$ in 
    $\P_{j_0}$, an ordinal $\gamma^{j_0} < \kappa$, a tree 
    $t^{j_0} \subseteq {^{\leq\gamma^{j_0}}}\kappa$, a set $B^{j_0}$ of 
    $\P_{j_0}$-names, and a bijection $\pi^{j_0}:B^{j_0} \ra t^{j_0}_{\gamma^{j_0}}$ 
    such that
    \begin{itemize}
      \item $p_1 \Vdash_{\P_{j_0}} p_0(j_0) = (B^{j_0}, \gamma^{j_0}, t^{j_0})$;
      \item for all $\dot{b} \in B^{j_0}$, $p_1 \Vdash \dot{b} \restriction 
      \gamma^{j_0} = \pi^{j_0}(\dot{b})$.
    \end{itemize}
	Since $p_0 \restriction j_0 \Vdash p_0(j_0) \leq p_{00}$, by extending $p_1$ 
	further if necessary we can assume that, for each $\dot{b} \in B^{p_{00},j_0}$, 
	there is a unique $\dot{b}' \in B^{j_0}$ such that 
	\[
		p_1 \Vdash_{\P_{j_0}}``\Vdash_{\dot{\Q}_{j_0}} \dot{b}' = \dot{b}".
	\]
    By replacing each $\dot{b}'$ with $\dot{b}$ in $B^{j_0}$ and redefining 
    $\pi^{j_0}$ in the obvious way, we may thus assume that $B^{j_0} \supseteq 
    B^{p_{00},j_0}$.
    
    Now apply the inductive hypothesis to find $p_2 \leq_{\P_{j_0}} p_1$ 
    such that $p_2 \in \P^*_{j_0}$ witnessed by 
    $\langle (B^{p_2,i}, \gamma^{p_2,i}, t^{p_2,i}, \pi^{p_2,i}) \mid 
    i \in \dom(p_2) \rangle$ such that, for all $i \in \dom(p_{00}) \cap j_0$, 
    we have $B^{p_{00},i} \subseteq B^{p_2,i}$. Finally, define $q \in \P_j$ by setting 
    $\dom(q) = \dom(p_2) \cup \{j_0\}$, $q \restriction j_0 = p_2$, and 
    $q(j_0) = (B^{j_0}, \gamma^{j_0}, t^{j_0})$. Then $q$ is as desired.
    
    Suppose next that $j > 0$ is a limit ordinal. If $\cf(j) > \mu$, then there is 
    $i < j$ such that $p_0 \in \P_i$, so we can apply the inductive hypothesis to 
    obtain the desired conclusion. Thus, assume that $\cf(j) \leq \mu$. 
    Let $\langle i_\eta \mid \eta < \cf(j) \rangle$ be an increasing enumeration of 
    a club in $j$. Using the inductive hypothesis and Claim \ref{claim: technical_claim}, 
    recursively construct a sequence $\langle q_\eta \mid \eta < \cf(j) \rangle$ such 
    that, for all $\eta < \cf(j)$, we have:
    \begin{itemize}
      \item $q_\eta \in \P^*_{i_\eta}$, as witnessed by 
      $\langle (B^{q_\eta,i}, \gamma^{q_\eta,i}, t^{q_\eta,i}, \pi^{q_\eta,i}) \mid 
      i \in \dom(q_\eta) \rangle$;
      \item $q_\eta \leq_{\P_{i_\eta}} p_0 \restriction i_\eta$;
      \item for all $i \in \dom(p_{00}) \cap i_\eta$, $B^{p_{00},i} \subseteq B^{q_\eta,i}$;
      \item for all $\xi < \eta$, we have
      \begin{itemize}
        \item $q_\eta \restriction i_\xi \leq_{\P_{i_\xi}} q_\xi$;
        \item for all $i \in \dom(q_\xi)$, $B^{q_\xi,i} \subseteq B^{q_\eta,i}$;
      \end{itemize}
      \item if $\eta$ is a limit ordinal, then
      \begin{itemize}
        \item $\dom(q_\eta) = \bigcup \{\dom(q_\xi) \mid \xi < \eta\}$;
        \item for all $i \in \dom(q_\eta)$, we have $\gamma^{q_\eta,i} = \sup\{ 
        \gamma^{q_\xi,i} \mid \xi < \eta \text{ and } i \in \dom(q_\xi)\}$ and 
        $B^{q_\eta,i} = \bigcup \{B^{q_\xi,i} \mid \xi < \eta \text{ and } i \in 
        \dom(q_\xi)\}$.
      \end{itemize}
    \end{itemize}
    The construction is straightforward, so we leave it to the reader. At the end of the 
    construction, another appeal to Claim \ref{claim: technical_claim} yields a condition 
    $q \in \P^*_j$, as witnessed by $\langle (B^{q,i}, \gamma^{q,i}, t^{q,i}, \pi^{q,i}) \mid 
    i \in \dom(q) \rangle$ such that
    \begin{itemize}
      \item $q$ is a lower bound of $\langle q_\eta \mid \eta < \cf(j) \rangle$;
      \item $\dom(q) = \bigcup \{\dom(q_\eta) \mid \eta < \cf(j)\}$;
      \item for all $i \in \dom(q)$, we have $B^{q,i} = \bigcup \{B^{q_\eta,i} \mid 
      \eta < \cf(j) \text{ and } i \in \dom(q_\eta)\}$.
    \end{itemize}
    Then $q$ is as desired, completing the proof of the claim.
  \end{proof}
  
  The following will be the key claim in establishing that $\P_\varepsilon$ is 
  $\kappa^+$-Knaster.
  
  \begin{claim} \label{claim: compatibility}
    Suppose that $j \leq \varepsilon$ and, for $\ell < 2$, $p_\ell \in \P^*_j$, as witnessed 
    by 
    \[
    \langle (B^{p_\ell,i}, \gamma^{p_\ell,i}, t^{p_\ell,i}, \pi^{p_\ell,i}) \mid 
    i \in \dom(p_\ell) \rangle.
    \] 
    Suppose moreover that, for all $i \in \dom(p_0) \cap 
    \dom(p_1)$, we have $t^{p_0,i} = t^{p_1,i}$. Then $p_0 \parallel p_1$.
  \end{claim}
  
  \begin{proof}
    The proof is by induction on $j$. If $j = 0$, the conclusion is trivial. Suppose 
    next that $j$ is a successor ordinal, say $j = j_0 + 1$. Assume that $j_0 \in 
    \dom(p_0) \cap \dom(p_1)$, as otherwise we can simply apply the induction hypothesis. 
    First, apply the induction hypothesis to find $q_0$ such that 
    $q_0 \leq_{\P_{j_0}} p_\ell \restriction j_0$ for $\ell < 2$. Next, let 
    $\gamma = \gamma^{p_0, j_0} = \gamma^{p_1, j_0}$ and $t = t^{p_0, j_0} = t^{p_1, j_0}$.
    For every $u \in t_\gamma$ and $\ell < 2$, there is a unique $\dot{b}_{u,\ell} \in 
    B^{p_\ell, j_0}$ such that $q_0 \Vdash_{\P_{j_0}} \dot{b}_{u,\ell} \restriction 
    \gamma = u$. Since $\P_{j_0}$ is $({<}\kappa)$-distributive, we can find 
    $q_1 \leq_{\P_{j_0}} q_0$ such that, for all $u \in t_\gamma$,
    \begin{itemize}
      \item $q_1$ decides the statement $``\dot{b}_{u,0} = \dot{b}_{u,1}"$;
      \item if $q_1 \Vdash \dot{b}_{u,0} \neq \dot{b}_{u,1}$, then there is 
      $\gamma_u < \kappa$ such that $q_1 \Vdash \dot{b}_{u,0} \restriction \gamma_u 
      \neq \dot{b}_{u,1} \restriction \gamma_u$. 
    \end{itemize}
    Let $s_0 = \{u \in t_\gamma \mid q_1 \Vdash \dot{b}_{u,0} \neq \dot{b}_{u,1}\}$ 
    and $s_1 = t_\gamma \setminus s_0$. Let $\gamma^* = \sup\{\gamma_u \mid u \in s_0\}$ 
    (or $\gamma^* = \gamma$ if $s_0 = \emptyset$). Again using the $({<}\kappa)$-distributivity 
    of $\P_{j_0}$, find $q_2 \leq_{\P_{j_0}} q_1$ such that, for all $u \in t_\gamma$ and $\ell < 2$, 
    $q_2$ decides the value of $\dot{b}_{u,\ell} \restriction \gamma^*$, say as $b^*_{u,\ell} \in {^{\gamma^*}}\kappa$. 
    Let $t^*$ be the downward closure of the set $\{b^*_{u,\ell} \mid u \in t_\gamma, \ 
    \ell < 2\}$ in ${^{<\kappa}}\kappa$, and let 
    \[
      B^* = \{\dot{b}_{u,0} \mid u \in t_\gamma\} \cup \{\dot{b}_{u,1} \mid u \in s_0\}.
    \]
    Define a condition $q \in \P_j$ by letting $\dom(q) = \dom(q_2) \cup \{j_0\}$, 
    $q \restriction j_0 = q_2$, and $q(j_0)$ be such that 
    $q_2 \Vdash q(j_0) = (B^*, \gamma^*, t^*)$.
    
    Finally, suppose that $j$ is a limit ordinal. If $\cf(j) > \mu$, then there is 
    $j_0 < j$ such that $p_0, p_1 \in \P_{j_0}$, so we can simply apply the induction 
    hypothesis. Thus, assume that $\cf(j) \leq \mu$. Let $\langle i_\eta \mid 
    \eta < \cf(j) \rangle$ be an increasing enumeration of a club in $j$. Recursively build 
    a sequence $\langle q_\eta \mid \eta < \cf(j) \rangle$ such that, for all 
    $\eta < \cf(j)$, we have:
    \begin{itemize}
      \item $q_\eta \in \P^*_{i_\eta}$, as witnessed by $\langle (B^{q_\eta,i}, 
      \gamma^{q_\eta,i}, t^{q_\eta,i}, \pi^{q_\eta,i}) \mid i \in \dom(q_\eta) \rangle$;
      \item $q_\eta \leq_{\P_{i_\eta}} p_0 \restriction i_\eta, p_1 \restriction i_\eta$;
      \item for all $\xi < \eta$, we have
      \begin{itemize}
        \item $q_\eta \restriction i_\xi \leq_{\P_{i_\xi}} q_\xi$;
        \item for all $i \in \dom(q_\xi)$, $B^{q_\xi,i} \subseteq B^{q_\eta,i}$;
      \end{itemize}
      \item if $\eta$ is a limit ordinal, then
      \begin{itemize}
        \item $\dom(q_\eta) = \bigcup \{\dom(q_\xi) \mid \xi < \eta\}$;
        \item for all $i \in \dom(q_\eta)$, we have $\gamma^{q_\eta,i} = 
        \sup\{\gamma^{q_\xi,i} \mid \xi < \eta \text{ and } i \in \dom(q_\xi)\}$ and 
        $B^{q_\eta,i} = \bigcup\{B^{q_\xi,i} \mid \xi < \eta \text{ and } i \in 
        \dom(q_\xi)\}$.
      \end{itemize}
    \end{itemize}
    The construction is straightforward using the induction hypothesis and Claim 
    \ref{claim: technical_claim}, so we leave it to the reader. At the end, another 
    appeal to Claim \ref{claim: technical_claim} yields $q \in \P_j$ that is a lower bound 
    for $\langle q_\eta \mid \eta < \cf(j) \rangle$. This condition $q$ extends both 
    $p_0$ and $p_1$, thus establishing the claim.
  \end{proof}
  
  We are now finally ready to prove that $\P_\varepsilon$ is $\kappa^+$-Knaster. 
  To this end, suppose that $\langle p_\eta \mid \eta < \kappa^+ \rangle$ is a 
  sequence of conditions in $\P_\varepsilon$. By extending the conditions if necessary, 
  we may assume that each $p_\eta$ is in $\P^*_\varepsilon$, as witnessed by 
  $\langle (B^{\eta,i}, \gamma^{\eta,i}, t^{\eta,i}, \pi^{\eta,i}) \mid 
  i \in \dom(p_\eta) \rangle$. Since $2^\mu = \kappa$, we can find an unbounded 
  $A \subseteq \kappa^+$ such that
  \begin{itemize}
    \item $\langle \dom(p_\eta) \mid \eta \in A \rangle$ forms a $\Delta$-system, 
    with root $r$;
    \item there is a sequence $\langle t^i \mid i \in r \rangle$ such that, for all 
    $\eta \in A$ and all $i \in r$, we have $t^{\eta,i} = t^i$.
  \end{itemize}
  Then, by Claim \ref{claim: compatibility}, $\langle p_\eta \mid \eta \in A \rangle$ 
  is a sequence of pairwise compatible conditions.
\end{proof}

We can now easily prove the main consistency result of this section. To make its statement 
self-contained, we recall the standing cardinal assumptions of this section at the 
start of the theorem.

\begin{theorem}
  Suppose that $\mu$ is a regular uncountable cardinal such that $\mu^{<\mu} = \mu$, 
  $2^\mu = \mu^+$, and $2^{\mu^+} = \mu^{++}$, and let $\kappa = \mu^+$. Then there is a forcing extension in which 
  all cardinalities and cofinalities are preserved and in which 
  \begin{enumerate}
    \item there exists a $\kappa$-Kurepa tree;
    \item every $\kappa$-Kurepa tree contains a normal superthin $\kappa$-Kurepa subtree;
    \item the cardinal arithmetic of the ground model is preserved; in particular, if 
    $\GCH$ holds in the ground model, then it continues to hold in the extension.
  \end{enumerate}
\end{theorem}

\begin{proof}
  We can assume that there is a $\kappa$-Kurepa tree in $V$; if not, then first force 
  with one of the standard forcings to add a $\kappa$-Kurepa tree with $\kappa^+$-many 
  cofinal branches (e.g., the analogue of the forcing $\Q$ from the proof of 
  Theorem \ref{main_thm}).
  
  We will force with a $({\leq}\mu)$-support iteration $\langle \P_i, 
  \dot{\Q}_j \mid i \leq \kappa^+, \ j < \kappa^+ \rangle$ such that, for all 
  $i < \kappa^+$, $\dot{\Q}_i$ is forced by $\P_i$ to be either trivial forcing or of the form 
  $\P([\dot{T}_i])$, where $\dot{T}_i$ is a $\P_i$-name for a $\kappa$-Kurepa tree
  such that
  \[
    \Vdash_{\P_i} \forall \sigma \in \dot{T}_i \ |[\dot{T}_i] \cap N_\sigma| = \kappa^+.
  \]
  By Theorem \ref{theorem: well-behaved}, $\P = \P_{\kappa^+}$ will be 
  $\kappa^+$-Knaster and $\kappa$-strategically closed. It will therefore preserve 
  all cardinalities and cofinalities, and hence will preserve the $\kappa$-Kurepa 
  tree that exists in the ground model. Moreover, since the length of the iteration 
  is $\kappa^+$, it will not change the value of $2^\nu$ for any cardinal $\nu$. 
  Moreover, for every $i < \kappa^+$, the quotient forcing $\P/\P_i$ is equivalent 
  to an iteration of the same shape as $\P$ and therefore shares the same properties.
  
  We will recursively specify $\P$ by concurrently specifying a sequence 
  $\langle \dot{T}_i \mid i < \kappa^+ \rangle$ such that
  \begin{itemize}
    \item for all $i < \kappa^+$, $\dot{T}_i$ is a $\P_i$-name for a subtree of 
    ${^{<\kappa}}\kappa$;
    \item it is forced by $\P_i$ that:
    \begin{itemize}
      \item if $\dot{T}_i$ is not a $\kappa$-Kurepa tree, then $\dot{\Q}_i$ is trivial 
      forcing;
      \item if $\dot{T}_i$ is a $\kappa$-Kurepa tree, then $\dot{Q}_i$ is 
      $\P([\dot{T}^*_i])$, where $\dot{T}^*_i$ is the $\P_i$-name for the set of 
      $\sigma \in \dot{T}_i$ that are contained in $\kappa^+$-many 
      distinct cofinal branches through $\dot{T}_i$ (note that, in this case, 
      $\dot{T}^*_i$ is forced to be a $\kappa$-Kurepa tree itself).
    \end{itemize}
  \end{itemize}
  By the cardinal arithmetic assumptions and the chain condition of $\P$, there will only be 
  $\kappa^+$-many nice $\P$-names for subtrees of ${^{<\kappa}}\kappa$. Moreover, 
  also by the arithmetic and chain condition, each such name $\dot{T}$ will in fact be a
  $\P_i$-name for some $i < \kappa^+$. Therefore, by a standard bookkeeping argument, 
  we can construct the sequence $\langle \dot{T}_i \mid i < \kappa^+ \rangle$ in such a way 
  that, in $V^{\P}$, for every subtree $T$ of ${^{<\kappa}}\kappa$, there is $i < \kappa^+$ 
  such that $T$ equals the interpretation of $\dot{T}_i$ in $V^{\P}$.
  
  Let $G$ be $\P$-generic over $V$. For all $i < \kappa^+$, let $G_i$ be the $\P_i$-generic 
  filter induced by $G$. We claim that $V[G]$ is the desired forcing extension. We have already 
  argued that requirements (1) and (3) in the statement of the theorem hold. It remains to 
  verify (2). To this end, let $T \in V[G]$ be a $\kappa$-Kurepa tree. We may assume that 
  it is a subtree of ${^{<\kappa}}\kappa$. By our construction of $\P$, we can 
  find $i < \kappa$ such that $T$ equals the interpretation of $\dot{T}_i$. In particular, 
  $T \in V[G_i]$. Moreover, since the quotient forcing $\P/\P_i$ is $\kappa$-strategically 
  closed, Lemma \ref{lemma: strat_closed_silver} implies that all of the cofinal branches 
  through $T$ in $V[G]$ are already in $V[G_i]$. In particular, $T$ is a 
  $\kappa$-Kurepa tree in $V[G_i]$. Therefore, $\Q_i$ is $\P([T^*])$, where 
  $T^*$ is the set of all $\sigma \in T$ such that $\kappa^+$-many cofinal branches through 
  $T$ contain $\sigma$. Forcing with $\Q_i$ therefore adds a normal superthin 
  $\kappa$-Kurepa subtree to $T^*$, and hence to $T$. Since $T$ was arbitrary, this 
  completes the proof of the theorem.
\end{proof}

\section{Open questions} \label{section: questions}

In this final section, we present some closing remarks and survey some of the remaining 
open questions. We first make the following observations about Theorem \ref{thm: 
omega_plus_one_perfect}.
\begin{itemize}
  \item In Theorem \ref{thm: omega_plus_one_perfect}, the assumption that 
  $E$ is a continuous image of ${^\kappa}\kappa$ is slightly more than is 
  needed. As the proof makes clear, it is enough to assume that $E$ is 
  the continuous image of $[T]$ for some normal, splitting, 
  $({<}\omega_1)$-closed tree $T \subseteq {^{<\kappa}}\kappa$. Note that, 
  for such a tree $T$, the set $[T]$ is necessarily $(\omega+1)$-perfect.
  \item The conclusion of Theorem \ref{thm: omega_plus_one_perfect} falls 
  slightly short of the assertion that $E \setminus X$ is $(\omega+1)$-perfect, 
  since it only gives a winning strategy for games 
  $G_\kappa(E, x_0, \omega+1)$ rather than $G_\kappa(E \setminus X, x_0, \omega+1)$.
\end{itemize} 
In light of these observations, the following questions are natural.

\begin{question}
  Suppose that $\kappa$ is a regular uncountable cardinal and $E \subseteq 
  {^\kappa}\kappa$ is a closed set.
  \begin{enumerate}
    \item Suppose that $E$ is a continuous image of ${^\kappa}\kappa$. Must 
    there be $X \subseteq E$ such that $|X| \leq \kappa^{<\kappa}$ and 
    $E \setminus X$ is $(\omega+1)$-perfect?
    \item Suppose that there is an $(\omega+1)$-perfect set $D \subseteq 
    {^\kappa}\kappa$ such that $E$ is a continuous image of $D$. Must there 
    be $X \subseteq E$ such that $|X| \leq \kappa^{<\kappa}$ and, for 
    every $x_0 \in E \setminus X$, Player II has a winning strategy in 
    $G_\kappa(E, x_0, \omega+1)$? Must there be $X \subseteq E$ such that 
    $|X| \leq \kappa^{<\kappa}$ and $E \setminus X$ is $(\omega+1)$-perfect?
  \end{enumerate}
\end{question}

We now turn to some questions arising from Theorem \ref{main_thm}. The first 
concerns whether an inaccessible cardinal is necessary to obtain the conclusion 
of the theorem.

\begin{question}
  Suppose that $\GCH$ holds, there is an $\omega_2$-Kurepa tree and, for every 
  $\omega_2$-Kurepa tree $S \subseteq {^{<\omega_2}}\omega_2$, $[S]$ is not a 
  continuous image of ${^{\omega_2}}\omega_2$. Must $(\omega_3)^V$ be inaccessible 
  in $L$?
\end{question}

We also ask whether the opposite extreme to that obtained in Theorem \ref{main_thm} 
can hold.

\begin{question} \label{q: extreme}
  Is it consistent with $\GCH$ that there are $\omega_2$-Kurepa trees and, 
  for every $\omega_2$-Kurepa tree $T \subseteq {^{<\omega_2}}\omega_2$, 
  $[T]$ \emph{is} a continuous image (or even a retract) of 
  ${^{\omega_2}}\omega_2$?
\end{question}

In all of the known consistent examples of $\omega_2$-Kurepa trees $T \subseteq 
{^{<\omega_2}}\omega_2$, there exist $x \in [T]$ and $\sigma \in T$ such 
that $\sigma \sqsubseteq x$ and $|[T] \cap N_\sigma| \leq \omega_2$. We conjecture 
that this is in necessarily the case.

\begin{conjecture} \label{conj}
  Suppose that $\GCH$ holds, and let $T \subseteq {^{<\omega_2}}\omega_2$ be 
  an $\omega_2$-Kurepa tree such that, for every $\sigma \in T$, we have 
  $|[T] \cap N_\sigma| > \omega_2$. Then $[T]$ is not a continuous image of 
  ${^{\omega_2}}\omega_2$
\end{conjecture}

Note that, if $\GCH$ holds and $T \subseteq {^{<\omega_2}}\omega_2$ is an 
$\omega_2$-Kurepa tree, then $T' := \{\sigma \in T \mid |[T] \cap N_\sigma| > 
\omega_2\}$ is an $\omega_2$-Kurepa tree satisfying the hypothesis of Conjecture \ref{conj}. 
Therefore, a positive resolution of Conjecture \ref{conj} would entail a negative 
answer to Question \ref{q: extreme}.

\bibliographystyle{plain}
\bibliography{bib}

\end{document}